\documentclass[twoside,a4paper,10.5pt]{article}
\usepackage{amsthm,amsmath,amssymb,amscd,graphicx,amsfonts,stmaryrd}
\usepackage{enumitem}
\usepackage{mathrsfs}
\usepackage[all]{xy}
\usepackage[margin=27mm]{geometry}
\usepackage{bigints}
\usepackage{mathtools}
\usepackage{enumitem}
\usepackage{comment}
\usepackage{fancyhdr}
\usepackage{latexsym}
\usepackage{tikz}
\usetikzlibrary{intersections,calc,arrows.meta}
\usepackage{hyperref}

\allowdisplaybreaks[1]
\thepage

\mathtoolsset{showonlyrefs=true}%参照したときだけ数式番号表示

\numberwithin{equation}{section}

\theoremstyle{definition}
\newtheorem{exm} {Example}[section]
\newtheorem{dfn}[exm] {Definition}
\newtheorem{rem}[exm] {Remark}

\theoremstyle{plane}
\newtheorem{lem}[exm]{Lemma}
\newtheorem{prop}[exm]{Proposition}
\newtheorem{thm}[exm]{Theorem}
\newtheorem{cor}[exm]{Corollary}
\newtheorem{assum}[exm]{Assumption}

\newcommand{\cX}{\mathcal{X}}
\newcommand{\cY}{\mathcal{Y}}
\newcommand{\cE}{\mathcal{E}}
\newcommand{\cF}{\mathcal{F}}
\newcommand{\cD}{\mathcal{D}}
\newcommand{\cL}{\mathcal{L}}
\newcommand{\cB}{\mathcal{B}}
\newcommand{\mbM}{\mathbb{M}}
\newcommand{\RNp}{\mathbb{R}_{\geq 0}}
\newcommand{\RN}{\mathbb{R}}
\newcommand{\ZN}{\mathbb{Z}}

\DeclareMathOperator\supp{supp}
\DeclareMathOperator\idty{id}

\DeclareMathOperator\dom{dom}

\DeclareMathOperator\cl{cl}
\DeclareMathOperator\tr{tr}

\title{Scaling limits of discrete-time Markov chains and their local times on electrical networks} %タイトルをいれる
\date{}
\author{Ryoichiro Noda\thanks{Research Institute for Mathematical Sciences, Kyoto University, Kyoto, 606-8502,
JAPAN. E-mail: sgrndr@kurims.kyoto-u.ac.jp}}

\begin{document}

\maketitle
\begin{abstract}
  We establish that 
  if a sequence of electrical networks equipped with conductance measures converges in the local Gromov--Hausdorff-vague topology
  and satisfies certain non-explosion and metric-entropy conditions,
  then the sequence of associated discrete-time Markov chains and their local times also converges.
  This result applies to many examples, 
  such as 
  critical Galton--Watson trees conditioned on size, 
  uniform spanning trees,
  random recursive fractals,
  the critical Erd\H{o}s--R\'{e}nyi random graph,
  the configuration model,
  and the random conductance model on fractals.
  To obtain the convergence result,
  we characterize and study extended Dirichlet spaces associated with resistance forms,
  and we study traces of electrical networks.
\end{abstract}

\tableofcontents

%%%%%%%%%%%%%%%%%%%%%%%%%%%%%%%%%%%%%%%%%%%%%%%%%%%%%%%%%%%%%%%%%%%%%%%%%%%%%%%%%%%%%%%%%%%%%%%%%%%%%%%%%%%%%%%%%%%%%%%%%%%%%%%%%%
%%%%%%%%%%%%%%%%%%%%%%%%%%%%%%%%%%%%%%%%%%%%%%%%%%%%%%%%%%%%%%%%%%%%%%%%%%%%%%%%%%%%%%%%%%%%%%%%%%%%%%%%%%%%%%%%%%%%%%%%%%%%%%%%%%
% Introduction
%%%%%%%%%%%%%%%%%%%%%%%%%%%%%%%%%%%%%%%%%%%%%%%%%%%%%%%%%%%%%%%%%%%%%%%%%%%%%%%%%%%%%%%%%%%%%%%%%%%%%%%%%%%%%%%%%%%%%%%%%%%%%%%%%%
%%%%%%%%%%%%%%%%%%%%%%%%%%%%%%%%%%%%%%%%%%%%%%%%%%%%%%%%%%%%%%%%%%%%%%%%%%%%%%%%%%%%%%%%%%%%%%%%%%%%%%%%%%%%%%%%%%%%%%%%%%%%%%%%%%

\section{Introduction}  \label{sec: introduction}

Markov chains on graphs are very simple stochastic processes but serve in many applications,
and so the class of such processes has been an important research focus. 
In particular, 
various properties of 
reversible Markov chains on graphs are known to be equivalent to those of electrical networks, 
and the analysis of them has seen remarkable progress 
in recent decades \cite{Barlow_17_Random,Levin_Peres_17_Markov,Lyons_Peres_16_Probability}. 
The theory of resistance forms established by Kigami \cite{Kigami_01_Analysis,Kigami_12_Resistance}
is a generalization of this research direction. 
In \cite{Kigami_01_Analysis},
Kigami introduced a class of metrics called resistance metrics, 
which is a generalization of effective resistance on electrical networks, 
and in \cite{Kigami_12_Resistance},
he showed that, on a resistance metric space equipped with a Radon measure of full support,
there exists a naturally associated Markov process,
which corresponds to a continuous-time Markov chain on an electrical network. 
Originally, the theory of resistance forms was developed 
for the analysis on fractals, 
but recent research \cite{Croydon_18_Scaling,Croydon_Hambly_Kumagai_17_Time-changes,Noda_pre_Convergence} 
has confirmed that the theory is also useful for the study of scaling limits of Markov chains and their local times 
on electrical networks.
The results in \cite{Noda_pre_Convergence} 
are generalizations of those in \cite{Croydon_18_Scaling,Croydon_Hambly_Kumagai_17_Time-changes},
and in that paper,
it is proven that 
if a sequence of electrical networks equipped with measures converges in the local Gromov--Hausdorff-vague topology 
and satisfies the non-explosion condition and the metric-entropy condition, 
then the associated continuous-time Markov chains and their local times also converge. 
In this paper,
we establish a similar result for discrete-time Markov chains,
which are of interest in their own right,
naturally appearing in studies of algorithms, for example.

It may appear that the convergence of discrete-time Markov chains can be obtained
directly from the corresponding result for continuous-time Markov chains in
\cite{Noda_pre_Convergence}, and indeed this is true at the level of sample paths of Markov chains.
To see this, consider constant-speed random walks, namely continuous-time
Markov chains whose holding time at each site is exponential with rate $1$.
Given such processes $X_n$ on electrical networks $G_n$, the associated
discrete-time chains $Y_n$ are obtained by applying the random time-change
determined by the holding times.
Under suitable convergence assumptions on the networks $G_n$, the law of large
numbers implies that the (rescaled) time-change processes converge uniformly to
the identity map. Consequently, the scaling limit of the discrete-time chains
$Y_n$ follows from that of the continuous-time chains $X_n$.
(See Proposition~\ref{5. prop: convergence of processes} for details.)

In contrast, for local times, it is not possible to deduce the convergence of
the local times of the discrete-time Markov chains from their continuous-time counterparts simply as a
consequence of a law of large numbers. 
The fundamental reason is that local times
are highly sensitive to time-changes, and there is no simple time-change relation
between the discrete-time and continuous-time local times. 
Consequently,
unlike the situation for sample paths, the convergence of discrete-time local
times cannot be obtained by a straightforward transfer from the continuous-time
setting.

Thus, following the strategy of \cite{Noda_pre_Convergence}, we derive the
convergence of discrete-time local times from the convergence of the underlying
discrete-time Markov chains. Two ingredients play a key role:
\begin{enumerate}[label=\textup{(\roman*)}]
\item modulus-of-continuity estimates for discrete-time local times;
\item a trace approximation technique that allows one to approximate a
      discrete-time Markov chain on a non-compact electrical network by the chain
      restricted to a compact subnetwork (with respect to the resistance metric).
\end{enumerate}

The first ingredient is essentially established by Croydon
\cite{Croydon_15_Moduli} when the underlying networks are compact, and it provides
the basis for proving tightness of the discrete-time local times. The second
ingredient allows us to extend this tightness to the non-compact case. While trace
approximations are well understood for continuous-time Markov chains through the
theory of resistance forms and Dirichlet forms (cf.\ 
\cite[Lemma~2.6]{Croydon_Hambly_Kumagai_17_Time-changes}), no corresponding result
appears to be available for discrete-time chains. In this paper, we develop a
discrete-time version of the trace approximation method by analyzing extended
Dirichlet spaces associated with resistance forms,
see Sections~\ref{sec: resistance forms} and \ref{sec: electrical networks}. 
This technique is of independent interest and is likely to be useful in a broader range of applications beyond the present work, 
and it represents another important contribution of this paper.

To present our main results,
we begin by introducing several pieces of notation.
We write $\RNp \coloneqq [0, \infty)$, equipped with the usual Euclidean topology.
Given metric spaces $S$ and $T$,
we write $C(S, T)$ for the space of continuous functions from $S$ to $T$,
equipped with the compact-convergence topology,
that is, a sequence $(f_n)_{n \geq 1}$ in $C(S, T)$ converges to $f$
if and only if $f_n$ converges to $f$ uniformly on every compact subset of $S$.
We write $D(\RNp, S)$ for the space of cadlag functions from $\RNp$ to $S$,
equipped with the usual $J_1$-Skorohod topology (cf.\ \cite[Section~16]{Billingsley_99_Convergence}).

Given a metric space $(S, d)$,
we set, for $x \in S$ and $r>0$,
\begin{equation}
  B_{d}(x, r) 
  \coloneqq 
  \{
    y \in S \mid d(x, y) < r
  \}, 
  \quad 
  D_{d}(x, r) 
  \coloneqq 
  \{
    y \in S \mid d(x, y) \leq r
  \}.
\end{equation}
We say that $(S, d)$ is \textit{boundedly compact} 
if and only if $D_{d}(x, r)$ is compact for all $x \in S$ and $r>0$.
Note that a boundedly-compact metric space is complete, separable and locally compact.
A tuple $(S, d, \rho, \mu)$ is said to be a \textit{rooted-and-measured boundedly-compact metric space} 
if and only if $(S, d)$ is a boundedly-compact metric space,
$\rho$ is a distinguished element of $S$ called the \textit{root},
and $\mu$ is a Radon measure on $S$, 
that is,
$\mu$ is a Borel measure on $S$ 
such that $\mu(K) < \infty$ for every compact subset $K$.
Given a rooted-and-measured boundedly-compact metric space $G = (S, d, \rho, \mu)$ and $r > 0$,
we define a rooted-and-measured compact metric space $G^{(r)} = (S^{(r)}, d^{(r)}, \rho^{(r)}, \mu^{(r)})$
by setting 
\begin{equation}  \label{1. eq: dfn of restriction operator}
  S^{(r)} \coloneqq \cl(B_{d}(\rho, r)), \quad 
  d^{(r)} \coloneqq d|_{S^{(r)} \times S^{(r)}}, \quad 
  \rho^{(r)} \coloneqq \rho, \quad 
  \mu^{(r)}(\cdot) \coloneqq \mu( \cdot \cap B_d(\rho, r)),
\end{equation}
where $\cl(\cdot)$ denotes the closure of a set.
We write $\mathbb{G}$ for the collection of rooted-and-measured isometric equivalence classes 
of rooted-and-measured boundedly-compact metric spaces
and equip $\mathbb{G}$ with the local Gromov--Hausdorff-vague topology.
(See Section~\ref{sec: the local GHV} for details).

\begin{dfn} [{The space $\mathbb{F}$ and $\mathbb{F}_{c}$}]
  \label{1. dfn: the space F}
  We define the subspace $\mathbb{F}$ of $\mathbb{G}$ 
  to be the collection of $(F,R,\rho,\mu) \in \mathbb{G}$
  such that $\mu$ is of full support and $R$ is a resistance metric 
  which is associated with a recurrent resistance form,
  i.e., it holds that 
  \begin{equation} \label{1. eq: recurrence of R in the space F}
    \lim_{r \to \infty} R(\rho, B_{R}(\rho, r)^{c}) = \infty. 
  \end{equation}
  We write $\mathbb{F}_{c}$ for the subspace of $\mathbb{F}$ consisting of the tuples $(F, R, \rho, \mu) \in \mathbb{F}$ 
  such that $(F, R)$ is compact.
  (For the definitions of resistance metrics and resistance forms, 
  see Definitions~\ref{3. dfn: resistance forms} and~\ref{3. dfn: resistance metrics}.
  For the reason why \eqref{1. eq: recurrence of R in the space F} means recurrence,
  see the discussion above Definition~\ref{3. dfn: recurrent resistance forms}.)
\end{dfn}

Let $G=(F, R, \rho, \mu)$ be an element of $\mathbb{F}$
and $(\mathcal{E}, \mathcal{F})$ be a resistance form corresponding to $R$.
In Corollary~\ref{3. cor: recurrence implies regularity},
it is shown that $(\cE, \cF)$ is regular
(see Definition~\ref{3. dfn: regular resistance forms} for regular resistance forms).
The regularity of $(\mathcal{E}, \mathcal{F})$ ensures 
the existence of a related regular Dirichlet form $(\mathcal{E},\mathcal{D})$ on $L^{2}(F,\mu)$ 
and also an associated Hunt process $((X_{G}(t))_{t \geq 0}, (P_{x}^{G})_{x \in F})$,
which is recurrent by the condition \eqref{1. eq: recurrence of R in the space F}.
In the study of continuity of local times of $X_{G}$,
the metric entropy defined below plays an important role.

\begin{dfn} [{$\varepsilon$-covering, metric entropy}]  \label{1. dfn: metric entropy}
  Let $(S,d)$ be a compact metric space.
  For $\varepsilon > 0$, a subset $A$ is called an \textit{$\varepsilon$-covering} of $(S,d)$,
  if it holds that $S = \bigcup_{x \in A} D_{d}(x, \varepsilon)$.
  We define 
  \begin{equation}
    N_{d}(S,\varepsilon)
    \coloneqq
    \min \{ |A| \mid A\ \text{is an}\  \varepsilon \text{-covering of}\ (S,d) \},
  \end{equation}
  where $|A|$ denotes the cardinality of $A$.
  An $\varepsilon$-covering $A$ with $|A|=N_{d}(S,\varepsilon)$ is called a \textit{minimal $\varepsilon$-covering} of $(S,d)$.
  We call the family $\{ N_{d}(S,\varepsilon) \mid \varepsilon >0 \}$ the \textit{metric entropy} of $(S,d)$.
\end{dfn}

\begin{dfn} [{The space $\check{\mathbb{F}}$ and $\check{\mathbb{F}}_{c}$}]
  We define the subspace $\check{\mathbb{F}}$ of $\mathbb{F}$ 
  to be the collection of $(F, R, \rho, \mu) \in \mathbb{F}$
  such that, for any $r >0$, there exists $\alpha_{r} \in (0,1/2)$ satisfying
  \begin{equation} \label{1. eq: definition of vF}
    \sum_{k \geq 1} N_{R^{(r)}}(F^{(r)},2^{-k})^{2} \exp(-2^{\alpha_{r} k})  < \infty,
  \end{equation}
  and we also define $\check{\mathbb{F}}_{c} \coloneqq \mathbb{F}_{c} \cap \check{\mathbb{F}}$.
\end{dfn}

Let $G = (F, R, \rho, \mu)$ be an element of $\check{\mathbb{F}}$.
By~\cite{Noda_pre_Convergence},
the Hunt process $X_{G}$ admits a jointly continuous local time 
$L_{G} = (L_{G}(x,t))_{x \in F, t \geq 0}$ satisfying the occupation density formula 
(see \eqref{3. eq: the occupation density formula}).
We define a probability measure $P_{G}$ on $D(\RNp, F) \times C(F \times \RNp, \RNp)$
by setting 
\begin{equation} \label{1. eq: def of P_G}
  P_{G}(\cdot)
  \coloneqq
  P_{\rho} 
  \left(
    (X_{G}, L_{G}) \in \cdot
  \right).
\end{equation}
We set 
\begin{equation}  \label{1. eq: dfn of cX_G}
  \cX_{G} 
  \coloneqq 
  (F, R, \rho, \mu, P_{G}),
\end{equation}
which we regard as an element of $\mbM_{L}$ defined in Section~\ref{sec: the space M_L}.
In particular,
$\mbM_{L}$ is a Polish space containing (equivalence classes of)
tuples $(S, d, \rho, \mu, P)$
such that $(S, d, \rho, \mu)$ is a rooted-and-measured boundedly-compact metric space 
and $P$ is a probability measure on $D(\RNp, F) \times C(F \times \RNp, \RNp)$.

We next introduce electrical networks and associated resistance forms.

\begin{dfn} [{Electrical network}] \label{dfn: electrical network}
  Let $(V, E)$ be a connected, simple, undirected graph with finite or countably many vertices,
  where $V$ denotes the vertex set and $E$ denotes the edge set.
  (NB.\ A graph being simple means that it has no loops and no multiple edges.)
  For $x, y \in V$,
  we write $x \sim y$ if and only if $\{ x, y \} \in E$.
  Let $c \colon V^{2} \to \RNp$ be a function such that 
  $c(x,y) = c(y, x)$ for all $x, y \in V$,
  $c(x, y) > 0$ if and only if $x \sim y$,
  and
  \begin{equation} \label{eq: total conductance at vertex}
    c(x) 
    \coloneqq 
    \sum_{y \in V} c(x, y) < \infty,
    \qquad 
    \forall x \in V.
  \end{equation}
  We call $c(x,y)$ the \textit{conductance} on the edge $\{x,y\}$
  and $(V, E, c)$ an \textit{electrical network}.
  We equip $V$ with the discrete topology 
  and define a Radon measure $\mu$ on $V$ by setting 
  \begin{equation}
    \mu(A)
    \coloneqq 
    \sum_{x \in A} c(x),
    \qquad 
    A \subseteq V,
  \end{equation}
  which we call the \textit{conductance measure associated with the electrical network $(V, E, c)$}.
  Given an electrical network $G$,
  we write $V_{G} = V(G), E_{G} = E(G), (c_{G}(x,y))_{x, y \in V_{G}}$ and $\mu_{G}$
  for the vertex set, the edge set, the conductances,
  and the associated conductance measure, respectively.
  When we say that $G$ is a \textit{rooted electrical network},
  there exists a distinguished vertex,
  which we denote by $\rho_{G} \in V_{G}$.
\end{dfn}

Let $G$ be an electrical network.
The discrete-time Markov chain associated with $G$ 
is the discrete-time Markov chain $Y_{G}=(Y_{G}(k))_{k \geq 0}$ on $V_{G}$ 
with transition probabilities $(P_{G}(x,y))_{x, y \in V_{G}}$ given by 
\begin{equation}
  P_{G}(x,y) 
  \coloneqq
  \frac{c_{G}(x,y)}{c_{G}(x)}.
\end{equation}
We write $P_{\rho}^{G}$ for the underlying probability measure of $Y_{G}$ started at $\rho \in V_{G}$.
By setting 
\begin{equation}
  Y_{G}(t) \coloneqq Y_{G}( \lfloor t \rfloor), \qquad t \geq 0,
\end{equation}
we regard $Y_{G}$ as a random element of $D(\RNp, V_{G})$.
We then define the local time $\ell_{G} = (\ell_{G}(x, t))_{x \in V_{G}, t \geq 0}$
of $Y_{G}$ by setting 
\begin{equation} \label{eq: discrete local time}
  \ell_{G}(x, t)
  \coloneqq
  \frac{1}{c_{G}(x)}
  \int_{0}^{t} 1_{\{x\}}(Y_{G}(s))\, ds,
\end{equation}
which we regard as a random element of $C(V_{G} \times \RNp, \RNp)$.
Note that $\ell_{G}$ satisfies the occupation density formula:
\begin{equation}  \label{1. eq: discrete-time occupation density formula}
  \int_{0}^{t} 
  f(Y_{G}(s))\,
  ds 
  =
  \int_{V_{G}} 
  f(y) \ell_{G}(y, t)\,
  \mu_{G}(dy)
\end{equation}
for all $f \colon V_{G} \to \RNp$ and $t \geq 0$.

In our arguments,
the analysis of Markov chains on electrical networks via associated resistance forms 
is important,
which we now introduce.

\begin{dfn} [{Resistance forms associated with electrical networks}]  
  \label{1. dfn: resistance forms associated with electrical networks}
  Let $G$ be an electrical network.
  For functions $f,g \colon V_{G} \to \RN$,
  we set 
  \begin{equation}
    \cE_{G}(f, g) 
    \coloneqq 
    \frac{1}{2} 
    \sum_{x, y \in V_{G}} 
    c_{G}(x,y) (f(x) - f(y)) (g(x) - g(y))
  \end{equation}
  (if the right-hand side exists).
  We then define a $\RN$-linear subspace of $\RN^{V_{G}}$ by setting  
  \begin{equation}
    \cF_{G} 
    \coloneqq
    \{
      f \colon V_G \to \RN \mid \cE(f, f) < \infty
    \}.
  \end{equation}
\end{dfn}
In Section~\ref{sec: resistance forms on electrical networks},
it is proven that 
the pair $(\cE_{G}, \cF_{G})$ is a regular resistance form,
and if we write $R_{G}$ for the associated resistance metric,
then the topology on $V_{G}$ induced from $R_{G}$ is the discrete topology.

Our first result concerns a sequence of deterministic (scaled) electrical networks,
which corresponds to \cite[Theorem~1.9]{Noda_pre_Convergence}.
To establish the convergence of the associated discrete-time Markov chains and their local times,
we assume the non-explosion condition introduced in \cite{Croydon_18_Scaling} 
and the metric-entropy condition introduced in \cite{Noda_pre_Convergence}
(see Assumption~\ref{1. assum: deterministic version} and Theorem~\ref{1. thm: deterministic, main result} below).
For each $n \in \mathbb{N}$,
let $G_{n}$ be a rooted electrical network.
We simply write 
\begin{equation}
  V_{n} \coloneqq V_{G_{n}}, \quad 
  c_{n} \coloneqq c_{G_{n}}, \quad 
  \mu_{n} \coloneqq \mu_{G_{n}}, \quad 
  \rho_{n} \coloneqq \rho_{G_{n}}, \quad 
  R_{n} \coloneqq R_{G_{n}}, \quad 
  Y_{n} \coloneqq Y_{G_{n}}, \quad 
  \ell_{n} \coloneqq \ell_{G_{n}} 
\end{equation}
and we assume that each $(V_{n}, R_{n}, \rho_{n}, \mu_{n})$ belongs to $\check{\mathbb{F}}$.
Suppose that we have scaling factors $a_{n}, b_{n}$,
that is,
$(a_{n})_{n \geq 1}$ and $(b_{n})_{n \geq 1}$ are sequences of positive numbers with $a_{n} \to \infty$ and $b_{n} \to \infty$.
We then write 
\begin{equation}  \label{1. eq: dfn of scaled spaces}
  \hat{V}_{n} \coloneqq V_{n}, \quad 
  \hat{R}_{n} \coloneqq a_{n}^{-1} R_{n}, \quad
  \hat{\rho}_{n} \coloneqq \rho_{n}, \quad 
  \hat{\mu}_{n} \coloneqq b_{n}^{-1} \mu_{n}.
\end{equation}
We consider the following conditions.

\begin{assum} \label{1. assum: deterministic version} \leavevmode
  \begin{enumerate} [label = \textup{(\roman*)}]
    \item \label{1. assum item: deterministic, convergence of spaces}
      The sequence $(\hat{V}_{n}, \hat{R}_{n}, \hat{\rho}_{n}, \hat{\mu}_{n})_{n \geq 1}$ in $\check{\mathbb{F}}$ satisfies
      \begin{equation}
        (\hat{V}_{n}, \hat{R}_{n}, \hat{\rho}_{n}, \hat{\mu}_{n}) 
        \to 
        (F, R, \rho, \mu)
      \end{equation}
      in the local Gromov--Hausdorff-vague topology 
      for some $G=(F, R, \rho, \mu) \in \mathbb{G}$.
    \item \label{1. assum item: deterministic, the non-explosion condition} 
      It holds that 
      \begin{equation}
        \lim_{r \to \infty} 
        \liminf_{n \to \infty}
        \hat{R}_{n}( \hat{\rho}_{n}, B_{\hat{R}_{n}}( \hat{\rho}_{n}, r)^{c}) 
        = 
        \infty.
      \end{equation}
    \item \label{1. assum item: deterministic, the metric-entropy condition}
      For each $r>0$,
      there exists $\alpha_{r} \in (0, 1/2)$ such that 
      \begin{equation}
        \lim_{m \to \infty}
        \limsup_{n \to \infty} 
        \sum_{k \geq m} 
        N_{\hat{R}_{n}^{(r)}} (\hat{V}_{n}^{(r)}, 2^{-k})^{2} \exp(-2^{\alpha_{r} k}) 
        =
        0,
      \end{equation} 
      where we note that $\hat{V}_{n}^{(r)} = B_{\hat{R}_{n}}(\hat{\rho}_{n}, r) = B_{R_{n}}(\rho_{n}, a_{n}r)$.
  \end{enumerate}
\end{assum}

Since the spaces and measures are scaled,
we consider accordingly scaled Markov chains and their local times defined as follows:
\begin{equation}
  \hat{Y}_{n}(t) \coloneqq Y_{n}(a_{n}b_{n}t), \quad 
  \hat{\ell}_{n}(x,t) \coloneqq a_{n}^{-1} \ell_{n}(x, a_{n}b_{n}t).
\end{equation}
We then set 
\begin{equation}
  \hat{P}_{n} \coloneqq P_{\rho_{n}}^{G_{n}} \bigl( (\hat{Y}_{n}, \hat{\ell}_{n}) \in \cdot \bigr), 
  \quad 
  \hat{\cY}_{n} \coloneqq (\hat{V}_{n}, \hat{R}_{n}, \hat{\rho}_{n}, \hat{\mu}_{n}, \hat{P}_{n}).
\end{equation}
Note that $\hat{P}_{n}$ is a probability measure 
on $D(\RNp, \hat{V}_{n}) \times C(\hat{V}_{n} \times \RNp, \RNp)$
and $\hat{\cY}_{n}$ is an element of $\mbM_{L}$.

\begin{thm} \label{1. thm: deterministic, main result}
  Under Assumption~\ref{1. assum: deterministic version},
  the limiting space $G$ belongs to $\check{\mathbb{F}}$ and 
  $\hat{\cY}_{n}$ converges to $\cX_{G}$ in $\mbM_{L}$,
  where we recall $\cX_{G}$ from \eqref{1. eq: dfn of cX_G}.
\end{thm}

In our second theorem,
we consider a sequence of random electrical networks,
which corresponds to \cite[Theorem~1.11]{Noda_pre_Convergence}.
To state the result,
in the previous setting,
we assume that each $G_{n}$ is a random electrical network,
that is,
$(V_{n}, R_{n}, \rho_{n}, \mu_{n})$ is a random element of $\check{\mathbb{F}}$.
We denote by $\mathbf{P}_{n}$ the underlying complete probability measure of $G_{n}$.
Note that the scaling factors $a_{n}, b_{n}$ are assumed to be deterministic.

\begin{assum} \label{1. assum: random version} \leavevmode
  \begin{enumerate} [label = \textup{(\roman*)}]
    \item \label{1. assum item: random, convergence of spaces}
      The sequence $(\hat{V}_{n}, \hat{R}_{n}, \hat{\rho}_{n}, \hat{\mu}_{n})_{n \geq 1}$ in $\check{\mathbb{F}}$ satisfies
      \begin{equation}
        (\hat{V}_{n}, \hat{R}_{n}, \hat{\rho}_{n}, \hat{\mu}_{n}) 
        \xrightarrow{\mathrm{d}}
        (F, R, \rho, \mu)
      \end{equation}
      in the local Gromov--Hausdorff-vague topology 
      for some random element $G=(F, R, \rho, \mu)$ of $\mathbb{G}$.
      We denote by $\mathbf{P}$ 
      the underlying complete probability measure of $G$.
    \item \label{1. assum item: random, the non-explosion condition} 
      It holds that 
      \begin{equation}
        \lim_{r \to \infty} 
        \liminf_{n \to \infty}
        \mathbf{P}_{n}
        \left(
          \hat{R}_{n}( \hat{\rho}_{n}, B_{\hat{R}_{n}}( \hat{\rho}_{n}, r)^{c}) 
          \geq
          \lambda
        \right)
        = 
        1,
        \qquad 
        \forall \lambda > 0.
      \end{equation}
    \item \label{1. assum item: random, the metric-entropy condition}
      For each $r>0$,
      there exists $\alpha_{r} \in (0, 1/2)$ such that 
      \begin{equation}
        \lim_{m \to \infty}
        \limsup_{n \to \infty} 
        \mathbf{P}_{n}
        \left(
          \sum_{k \geq m} 
          N_{\hat{R}_{n}^{(r)}} (\hat{V}_{n}^{(r)}, 2^{-k})^{2} \exp(-2^{\alpha_{r} k}) 
          \geq 
          \varepsilon
        \right)
        =
        0, 
        \qquad 
        \forall \varepsilon > 0.
      \end{equation} 
  \end{enumerate}
\end{assum}

\begin{thm} \label{1. thm: random, main result}
  Under Assumption~\ref{1. assum: random version},
  the limiting space $G$ is a random element of $\check{\mathbb{F}}$,
  i.e., $\mathbf{P}(G \in \check{\mathbb{F}})=1$,
  and $\hat{\cY}_{n} \xrightarrow{\mathrm{d}} \cX_{G}$
  as random elements of $\mbM_{L}$.
\end{thm}

\begin{rem}
  By \cite[Proposition~6.1]{Noda_pre_Convergence},
  if $G$ is a random element of $\check{\mathbb{F}}$,
  then $\cX_{G}$ is a random element of $\mbM_{L}$.
  Similarly,
  in Corollary~\ref{4. cor: Y_G is measurable},
  it is proven that $\hat{\cY}_{n}$ that appears in Theorem~\ref{1. thm: random, main result} 
  is a random element of $\mbM_{L}$.
\end{rem}

\begin{rem}
  Checking the metric-entropy condition of 
  Assumption~\ref{1. assum: deterministic version}\ref{1. assum item: deterministic, the metric-entropy condition}
  or Assumption~\ref{1. assum: random version}\ref{1. assum item: random, the metric-entropy condition} 
  directly can be challenging.
  However, it can be verified via suitable volume estimates of balls.
  In particular,
  if the sequence $(G_{n})_{n \geq 1}$ satisfies the uniform volume doubling (UVD) condition 
  (see \cite[Definition1.1]{Croydon_Hambly_Kumagai_17_Time-changes}),
  then the metric-entropy condition is satisfied.
  Sufficient conditions for the metric-entropy condition
  are provided in \cite[Section 7]{Noda_pre_Convergence}.
\end{rem}

\begin{rem}
  In \cite{Noda_pre_Convergence},
  where convergence of continuous-time Markov chains and their local times is established,
  the following examples are considered.
  \begin{enumerate} [label = (\alph*)]
    \item Models in the UVD regime (cf.\ \cite{Croydon_Hambly_Kumagai_17_Time-changes}).
    \item \label{1. rem item: random gasket}
      A random recursive Sierpi\'{n}ski gasket (cf.\ \cite{Hambly_97_Brownian}).
    \item \label{1. rem item: GW trees}
      Critical Galton--Watson trees conditioned on size 
      (cf.\ \cite{Aldous_93_The_continuum,Andriopoulos_23_Convergence,Duquesne_03_A_limit}).
    \item \label{1. rem item: UST}
      Uniform spanning trees on $\ZN^d$ with $d=2, 3$ and on high-dimensional tori
      (cf. \cite{Angel_Croydon_Hernandez-Torres_Shiraishi_21_Scaling,Archer_Nachmias_Shalev_24_The_GHP,Barlow_Croydon_Kumagai_17_Subsequential}).
    \item \label{1. rem item: Erdos-renyi random graph}
      The critical Erd\H{o}s--R\'{e}nyi random graph (cf.\ \cite{Berry_Broutin_Goldschmidt_12_The_continuum}).
    \item \label{1. rem item: configuration models} 
      The critical configuration model (cf.\ \cite{Bhamidi_Sen_20_Geometry}).
  \end{enumerate}
  Our main results are also applicable to these examples.
  However, we emphasize that in \cite{Noda_pre_Convergence} counting measures are considered 
  for \ref{1. rem item: GW trees}, \ref{1. rem item: UST}, \ref{1. rem item: Erdos-renyi random graph} and \ref{1. rem item: configuration models},
  and for \ref{1. rem item: random gasket} the existence of deterministic scaling factors for measures is not mentioned.
  Thus, in order to apply our main results, 
  we need to close some small gaps; we do this in Appendix \ref{A. sec: application}.
  Moreover,
  in the appendix,
  we consider another example, the random conductance model on unbounded fractals.
\end{rem}

\begin{rem}
  In Appendix~\ref{A. sec: convergence of traces}, we also study the convergence of
  the traces of the random walks and their local times onto (possibly random)
  compact subsets. This question is of independent interest, 
  since in the main results of the paper traces are used only as a technical tool to approximate Markov chains on non-compact networks, 
  whereas the convergence of the traces themselves involves an additional subtlety
  regarding conductances on boundaries.  
  We therefore treat this topic separately in the appendix, 
  as it may be useful in other contexts as well.
\end{rem}

The remainder of the article is organized as follows.
In Section~\ref{sec: the topologies for the main results},
we introduce the space $\mathbb{G}$ and $\mbM_{L}$ used in our main results.
In Section~\ref{sec: resistance forms},
we provide new results on resistance forms by characterizing and studying associated extended Dirichlet spaces.
The results are applied to the study of traces of electrical networks 
in Section~\ref{sec: electrical networks}.
In Sections~\ref{sec: proof of main result 1} and~\ref{sec: proof of the main results 2},
we provide the proofs of Theorem~\ref{1. thm: deterministic, main result} and Theorem~\ref{1. thm: random, main result},
respectively.

%%%%%%%%%%%%%%%%%%%%%%%%%%%%%%%%%%%%%%%%%%%%%%%%%%%%%%%%%%%%%%%%%%%%%%%%%%%%%%%%%%%%%%%%%%%%%%%%%%%%%%%%%%%%%%%%%%%%%%%%%%%%%%%%%%
%%%%%%%%%%%%%%%%%%%%%%%%%%%%%%%%%%%%%%%%%%%%%%%%%%%%%%%%%%%%%%%%%%%%%%%%%%%%%%%%%%%%%%%%%%%%%%%%%%%%%%%%%%%%%%%%%%%%%%%%%%%%%%%%%%
% The topologies for the main results
%%%%%%%%%%%%%%%%%%%%%%%%%%%%%%%%%%%%%%%%%%%%%%%%%%%%%%%%%%%%%%%%%%%%%%%%%%%%%%%%%%%%%%%%%%%%%%%%%%%%%%%%%%%%%%%%%%%%%%%%%%%%%%%%%%
%%%%%%%%%%%%%%%%%%%%%%%%%%%%%%%%%%%%%%%%%%%%%%%%%%%%%%%%%%%%%%%%%%%%%%%%%%%%%%%%%%%%%%%%%%%%%%%%%%%%%%%%%%%%%%%%%%%%%%%%%%%%%%%%%%
\section{The topologies for the main results} \label{sec: the topologies for the main results}

%%%%%%%%%%%%%%%%%%%%%%%%%%%%%%%%%%%%%%%%%%%%%%%%%%%%%%%%%%%%%%%%%%%%%%%%%%%%%%%%%%%%%%%%%%%%%%%%%%%%%%%%%%%%%%%%%%%%%%%%%%%%%%%%%%
% The topologies for the main results
%%%%%%%%%%%%%%%%%%%%%%%%%%%%%%%%%%%%%%%%%%%%%%%%%%%%%%%%%%%%%%%%%%%%%%%%%%%%%%%%%%%%%%%%%%%%%%%%%%%%%%%%%%%%%%%%%%%%%%%%%%%%%%%%%%
\newcommand{\dHbar}{d_{\bar{H}, \rho}}
\newcommand{\dH}{d_{H}}
\newcommand{\cCc}{\mathcal{C}_{\mathrm{cpt}}}
\newcommand{\cC}{\mathcal{C}}
\newcommand{\cMfin}{\mathcal{M}_{\mathrm{fin}}}
\newcommand{\cM}{\mathcal{M}}
\newcommand{\dP}{d_{P}}
\newcommand{\dV}{d_{V, \rho}}

In this section,
we introduce the topologies used in our main results,
following \cite[Section~2]{Noda_pre_Convergence}.
Note that we set $a \wedge b \coloneqq  \min \{a, b\}$ 
and $a \vee b \coloneqq  \max \{ a, b \}$ for $a, b \in \RN \cup \{ \pm \infty \}$.

\subsection{The local Gromov--Hausdorff-vague topology}  \label{sec: the local GHV}

We introduce the local Gromov--Hausdorff-vague topology,
which is used to discuss convergence of rooted-and-measured boundedly-compact metric spaces.
For details,
refer to 
\cite[Section 2]{Noda_pre_Convergence} and \cite{Noda_pre_Metrization}.

Let $(S, d, \rho)$ be a rooted boundedly-compact metric space.

\begin{dfn} \label{dfn: space of closed subsets}
  We write $\cCc(S)$ (resp.\ $\cC(S)$) for the set of compact (resp.\ closed) subsets of $S$.
  Note that both $\cCc(S)$ and $\cC(S)$ include the empty set.
\end{dfn}

We first recall the \emph{Hausdorff metric} on $\cCc(S)$.
For a subset $A \subseteq S$,
the \textit{(closed) $\varepsilon$-neighborhood} of $A$ in $(S, d)$ is given by 
\begin{equation}
  A^{\varepsilon} 
  \coloneqq
  \{
    x \in S \mid \exists y \in A\ \text{such that}\ d(x,y) \leq \varepsilon
  \}.
\end{equation} 
The Hausdorff metric $d_{H}$ on $\cCc(S)$ is defined by setting
\begin{equation}
  d_{H}(A, B)
  \coloneqq
  \inf\{
    \varepsilon \geq 0 \mid A \subseteq B^{\varepsilon},\, B \subseteq A^{\varepsilon}
  \},
\end{equation}
where the infimum over the empty set is defined to be $\infty$.
It is known that $d_{H}$ is indeed a metric (allowed to take the value $\infty$ due to the empty set) on $\cCc(S)$
(see \cite[Section 7.3.1]{Burago_Burago_Ivanov_01_A_course}).
We call the topology on $\cCc(S)$ induced from $\dH$ 
the \textit{Hausdorff topology}.

A commonly used topology on $\cC(S)$ is the Fell topology (see \cite[Appendix~C]{Molchanov_17_Theory}).
We define a metric inducing this topology using the Hausdorff metric as follows.
Note that, for each subset $A \subseteq S$ and $r > 0$, 
we write 
  \begin{equation}  \label{eq: restriction of a set to a closed ball}
  A^{(r)} 
  \coloneqq 
  \cl(A \cap B_{d}(\rho, r)),
\end{equation}
where $\cl(\cdot)$ denotes the closure of a subset.

\begin{dfn}
  For each $A, B \in \cC(S)$, define
  \begin{equation} \label{dfn eq: fell metric}
    \dHbar (A, B) 
    \coloneqq 
    \int_{0}^{\infty} e^{-r} \left( 1 \wedge d_{H}(A^{(r)}, B^{(r)}) \right) dr.
  \end{equation}
\end{dfn}

The function $\dHbar$ is indeed a metric on $\cC(S)$
and a natural extension of the Hausdorff metric for non-compact sets.
The following is a basic property of $\dHbar$.

\begin{thm} [{\cite[Theorems~3.8, 3.9, and 3.11]{Noda_pre_Metrization}}]
  The function $\dHbar$ is a metric 
  on $\cC(S)$
  and the metric space $(\cC(S), \dHbar)$ is compact.
  The induced topology on $\cC(S)$ coincides with the Fell topology.
  In particular, a sequence $(A_{n})_{n \geq 1}$ converges to $A$ with respect to $\dHbar$
  if and only if 
  $A_{n}^{(r)}$ converges to $A^{(r)}$
  in the Hausdorff topology for all but countably many $r>0$.
\end{thm}

For convergence of measures, 
we use the vague topology.
So, we next introduce a metric inducing the vague topology.
Recall that $(S, d, \rho)$ is a rooted boundedly-compact metric space.

\begin{dfn} \label{dfn: space of measures}
  We write $\cMfin(S)$ (resp.\ $\cM(S)$) for the set of finite Borel (resp.\ Radon) measures on $(S, d)$.
\end{dfn}

Recall that the \textit{Prohorov metric} $d_{P}$
between $\mu, \nu \in \cMfin(S)$  
is given by 
\begin{equation}
  d_{P}(\mu, \nu) 
  \coloneqq
  \inf\{
    \varepsilon \mid \mu(A) \leq \nu(A^{\varepsilon}) + \varepsilon,\,
    \nu(A) \leq \mu(A^{\varepsilon}) + \varepsilon,\,
    \forall A \subseteq S
  \}.
\end{equation}
By extending this Prohorov metric similarly to \eqref{dfn eq: fell metric},
we define a metric on $\cM(S)$ as follows.
Note that, for each $\mu \in \cM(S)$ and $r > 0$,
we write $\mu^{(r)}$ 
for the restriction of $\mu$ to $B_{d}(\rho, r)$.

\begin{dfn} \label{dfn: vague metric}
  For each $\mu, \nu \in \cM(S)$,
  we define
  \begin{equation}
    \dV(\mu, \nu)
    \coloneqq
    \int_{0}^{\infty} e^{-r} \left( 1 \wedge d_{P}(\mu^{(r)}, \nu^{(r)}) \right) dr.
  \end{equation}
\end{dfn}

\begin{thm} [{\cite[Theorems~3.19 and 3.20]{Noda_pre_Metrization}}]
  The function $\dV$ is a metric on $\mathcal{M}(S)$.
  The metric space $(\mathcal{M}(S), \dV)$ is separable and complete.
  Let $\mu, \mu_{1}, \mu_{2}, \ldots$ be 
  elements of $\cM(S)$.
  Then these conditions are equivalent:
  \begin{enumerate} [label = \textup{(\roman*)}]
    \item $\mu_{n}$ converges to a Radon measure $\mu$ with respect to $\dV$;
    \item $\mu_{n}^{(r)}$ converges weakly to $\mu^{(r)}$ for all but countably many $r>0$;
    \item $\mu_{n}$ converges vaguely to $\mu$, that is, for all continuous functions $f \colon S \to \RN$ with compact support,
      it holds that 
      \begin{equation}
        \lim_{n \to \infty} \int_{S} f(x)\, \mu_{n}(dx) = \int_{S} f(x)\, \mu(dx).
      \end{equation}
  \end{enumerate}
\end{thm}

Now, we introduce the local Gromov--Hausdorff-vague topology.
We say that two rooted-and-measured boundedly-compact metric spaces 
$G_{i} = (S_{i}, d_{i}, \rho_{i}, \mu_{i})$, $i=1,2$,
are GHV-equivalent
if and only if there exists a root-preserving isometry $f \colon S_{1} \to S_{2}$ such that 
$\mu_{2} = \mu_{1} \circ f^{-1}$.
Note that $f$ being an isometry means that $f$ is distance-preserving and surjective (and hence bijective)
and $f$ being root-preserving means that $f(\rho_{1})=\rho_{2}$.
We write $\mathbb{G}$ for the collection of GHV-equivalence classes of 
rooted-and-measured boundedly-compact metric spaces.
We define $\mathbb{G}_{c}$ to be the collection of $(S, d, \rho, \mu) \in \mathbb{G}$
such that $(S, d)$ is compact.

\begin{rem} \label{rem: how to regard G as a set}
  From the rigorous point of view of set theory,
  neither $\mathbb{G}_c$ nor $\mathbb{G}$ is a set.
  However, it is possible to think of both as sets.
  This is because one can construct a legitimate set $\mathscr{G}$ of rooted-and-measured boundedly-compact spaces 
  such that any rooted-and-measured boundedly-compact space is GHV-equivalent to a unique element of $\mathscr{G}$.
  (see \cite[Proposition~6.2]{Noda_pre_Metrization}.)
  Therefore, in this article, 
  we will proceed with the discussion by treating $\mathbb{G}_c$ and $\mathbb{G}$ as sets
  to avoid repeatedly referring to this set-theoretic formality 
  concerning the choice of representatives.
\end{rem}

Recall that 
the \textit{(pointed) Gromov--Hausdorff--Prohorov metric} $d_{\mathrm{GHP}} $ on $\mathbb{G}_{c}$ is given 
by setting,
for $G_{i} = (S_{i}, d_{i}, \rho_{i}, \mu_{i}) \in \mathbb{G}_{c},\, i=1,2$,
\begin{equation}  \label{2. eq: dfn of GHP metric}
  d_{\mathrm{GHP}} (G_{1}, G_{2})
  \coloneqq 
  \inf_{f_{1}, f_{2}, M}
  \{
    d(f_{1}(\rho_{1}), f_{2}(\rho_{2})) 
    \vee 
    d_{H}(f_{1}(S_{1}), f_{2}(S_{2})) 
    \vee 
    d_{P}(\mu_{1} \circ f_{1}^{-1}, \mu_{2}  \circ f_{2}^{-1})
  \},
\end{equation}
where the infimum is taken over all compact metric spaces $(M, d)$ 
and all distance-preserving maps $f _i \colon S_{i} \to M,\, i=1,2$.
We equip $\mathbb{G}_c$ with the topology induced by $d_{\mathrm{GHP}}$,
and call it the \textit{(pointed) Gromov--Hausdorff--Prohorov topology}.
It is known that this topology is Polish 
and further details of this metric are found in \cite{Abraham_Delmas_Hoscheit_13_A_note,Khezeli_20_Metrization}.

An extension of $d_{\mathrm{GHP}}$ to a metric on $\mathbb{G}$ was also studied in \cite{Abraham_Delmas_Hoscheit_13_A_note,Khezeli_20_Metrization}.
It is defined in a manner similar to \eqref{dfn eq: fell metric} as follows:
for each $G_1, G_2 \in \mathbb{G}$, set
\begin{equation}
d'_{\mathrm{GHP}}(G_1, G_2)
\coloneqq
\int_0^\infty e^{-r} \bigl( 1 \wedge d_{\mathrm{GHP}}(G_1^{(r)}, G_2^{(r)}) \bigr)\, dr,
\end{equation}
where we recall the restriction operator $\cdot^{(r)}$ from \eqref{1. eq: dfn of restriction operator}.
This defines a metric on $\mathbb{G}$ that induces a Polish topology
(see \cite[Remark~3.20 and Theorem~3.27]{Khezeli_20_Metrization}).
In what follows, however, we introduce another metric $d_{\mathbb{G}}$ on $\mathbb{G}$ in Definition~\ref{dfn: GHV metric}.
Although $d_{\mathbb{G}}$ induces the same topology as $d'_{\mathrm{GHP}}$,
we prefer $d_{\mathbb{G}}$
since its formulation is more flexible
(see \cite[Section~1]{Noda_pre_Metrization} on this point).
Indeed, the same philosophy also underlies
the metrization of the space $\mathbb{M}_L$,
which will be discussed in the next subsection.
The metric $d_{\mathbb{G}}$ is defined
in a manner similar to \eqref{2. eq: dfn of GHP metric},
with a minor modification concerning the treatment of roots.

\begin{dfn} \label{dfn: GHV metric}
  For $G_{i} = (S_{i}, d_{i}, \rho_{i}, \mu_{i}) \in \mathbb{G},\, i=1,2$,
  we set 
  \begin{equation}
    d_{\mathbb{G}}(G_{1}, G_{2})
    \coloneqq
    \inf_{f_{1}, f_{2}, M} 
    \left\{
      \dHbar (f_{1}(S_{1}), f_{2}(S_{2})) \vee \dV (\mu_{1} \circ f_{1}^{-1}, \mu_{2} \circ f_{2}^{-1})
    \right\},
  \end{equation}
  where the infimum is taken over all rooted boundedly-compact metric spaces $(M, d, \rho)$ 
  and all root-and-distance-preserving maps $f _i \colon S_{i} \to M,\, i=1,2$.
\end{dfn}

\begin{thm} [{\cite[Theorem~8.9]{Noda_pre_Metrization}}]
  The function $d_{\mathbb{G}}$ is a well-defined metric on $\mathbb{G}$,
  and the metric space $(\mathbb{G}, d_{\mathbb{G}})$ is complete and separable.
\end{thm}

Regarding convergence in $\mathbb{G}$,
we have the following result.

\begin{thm} [{\cite[Theorem~8.10]{Noda_pre_Metrization}}] \label{2. thm: characterization of convergence in the local. GHV}
  Let $G=(S, d, \rho, \mu)$ and $G_{n} = (S_{n}, d^{n}, \rho_{n}, \mu_{n}),\, n \in \mathbb{N}$
  be elements in $\mathbb{G}$.
  Then, the following statements are equivalent:
  \begin{enumerate} [label = \textup{(\roman*)}]
    \item 
      $G_{n}$ converges to $G$ with respect to $d_{\mathbb{G}}$;
    \item    
      $G_{n}$ converges to $G$ with respect to $d'_{\mathrm{GHP}}$;
    \item 
      $G_{n}^{(r)}$ converges to $G^{(r)}$ in the Gromov--Hausdorff--Prohorov topology 
      for all but countably many $r>0$;
    \item 
      there exist a rooted boundedly-compact metric space $(M, d^{M}, \rho_{M})$ 
      and root-and-distance-preserving maps $f _n \colon S_{n} \to M$ and $f \colon S \to M$ such that 
      $f_{n}(S_{n}) \to f(S)$ in the Fell topology in $M$ 
      and 
      $\mu_{n} \circ f_{n}^{-1} \to \mu \circ f^{-1}$ vaguely as measures on $M$.
  \end{enumerate}  
\end{thm}

\begin{dfn} [{The local Gromov--Hausdorff-vague topology}]
  We call the topology on $\mathbb{G}$ induced by the metric $d_{\mathbb{G}}$
  (or, equivalently, $d'_{\mathrm{GHP}}$) 
  the \textit{local Gromov--Hausdorff-vague topology}.
\end{dfn}

\begin{rem} 
  The local Gromov--Hausdorff-vague topology is a little different 
  from the Gromov--Hausdorff-vague topology introduced in \cite{Athreya_Lohr_Winter_16_The_gap}.
  This is because
  the topology in \cite{Athreya_Lohr_Winter_16_The_gap} deals with the convergence of the supports of measures 
  instead of the whole spaces.
  However, regarding our main results of this paper,
  since we assume that all measures contained in $\mathbb{F}$ (recall this from Definition~\ref{1. dfn: the space F}) are of full support,
  the topology induced into $\mathbb{F}$ is same
  whether one uses the Gromov--Hausdorff-vague topology defined in \cite{Athreya_Lohr_Winter_16_The_gap} or 
  the local Gromov--Hausdorff-vague topology.
\end{rem}

\begin{rem}
  The local Gromov--Hausdorff-vague topology is strictly coarser than the Gromov--Hausdorff--Prohorov topology
  (see \cite[Remark~3.23]{Khezeli_20_Metrization}).
\end{rem}

\begin{rem}
  From Theorem~\ref{2. thm: characterization of convergence in the local. GHV}, 
  it might be more appropriate to call the local Gromov--Hausdorff-vague topology 
  the Gromov--Fell-vague topology. 
  However, since the term “Gromov--Hausdorff” has been widely adopted in the literature, 
  including for non-compact underlying spaces, 
  and is more familiar to a broader audience, 
  we follow this convention and use “Gromov--Hausdorff” throughout the paper.
\end{rem}

For later use,
we introduce another space $\mathbb{D}$ by dropping measures from $\mathbb{G}$.
This is defined precisely as follows.
We say that two rooted boundedly-compact metric spaces are equivalent 
if and only if there exists a root-preserving isometry between them.
We write $\mathbb{D}$ for the collection of equivalence classes of rooted boundedly-compact metric spaces.
If we define a metric on $\mathbb{D}$ similarly to $d_{\mathbb{G}}$ in Definition~\ref{dfn: GHV metric}
but dropping measure components,
that is, 
if we define 
\begin{equation}
  d_{\mathbb{D}}\bigl( (S_1, d_1, \rho_1), (S_2, d_2, \rho_2)\bigr)
  \coloneqq
  \inf_{f_{1}, f_{2}, M} 
  \dHbar (f_{1}(S_{1}), f_{2}(S_{2})),
\end{equation}
then we obtain a complete, separable metric $d_{\mathbb{D}}$ on $\mathbb{D}$.
The induced topology is called the \emph{local Gromov--Hausdorff topology};
see \cite[Section~4]{Noda_pre_Metrization} for details.

%%%%%%%%%%%%%%%%%%%%%%%%%%%%%%%%%%%%%%%%%%%%%%%%%%%%%%%%%%%%%%%%%%%%%%%%%%%%%%%%%%%%%%%%%%%%%%%%%%%%%%%%%%%%%%%%%%%%%%%%%%%%%%%%%%
% The space M_L
%%%%%%%%%%%%%%%%%%%%%%%%%%%%%%%%%%%%%%%%%%%%%%%%%%%%%%%%%%%%%%%%%%%%%%%%%%%%%%%%%%%%%%%%%%%%%%%%%%%%%%%%%%%%%%%%%%%%%%%%%%%%%%%%%%

\subsection{The space $\mbM_{L}$} \label{sec: the space M_L}
\newcommand{\hatC}{\widehat{C}}
\newcommand{\dhatC}{d_{\widehat{C}, \rho}}
\newcommand{\dhatCc}{d_{\widehat{C}_{c}}}

In this subsection,
following \cite[Section~2.2]{Noda_pre_Convergence},
we define an extended version of the local Gromov--Hausdorff–vague topology
on tuples consisting of
a rooted-and-measured boundedly-compact metric space
together with a probability measure on the space of cadlag paths and local-time-type functions.

To develop a Gromov--Hausdorff-type topology suitable for discussing the convergence of local times,
we first introduce an extension of the compact-convergence topology for continuous functions
to functions whose domains may differ.
Although we focus here on local-time-type functions,
the same discussion applies in a more general setting;
see \cite[Section~3.3]{Noda_pre_Metrization}.

Let $(S, d, \rho)$ be a rooted boundedly-compact metric space. 
Recall from Definition~\ref{dfn: space of closed subsets} that $\cC(S)$ denotes 
the collection of closed subsets of $S$.

\begin{dfn}
  We define 
  \begin{equation}
    \hatC (S \times \RNp, \RN) 
    \coloneqq 
    \bigcup_{X \in \cC(S)} 
    C(X \times \RNp, \RN).
  \end{equation} 
  Note that $\hatC(S \times \RNp, \RN)$ contains 
  the empty map $\emptyset_{\RN} \colon \emptyset \times \RNp \to \RN$.
  For each $L \in \hatC(S \times \RNp, \RN)$,
  if $L \in C(X \times \RNp, \RN)$,
  then we write $\dom_{1}(L) \coloneqq X$.
\end{dfn} 

\begin{dfn} [{The compact-convergence topology with variable domains}]
  Let $L, L_{1}, L_{2}, \ldots$ be elements in $\hatC(S \times \RNp, \RN)$.
  We say that $L_{n}$ converges to $L$ in the \textit{compact-convergence topology with variable domains}
  if and only if 
  the sets $\dom_{1}(L_{n})$ converge to $\dom_{1}(L)$ in the Fell topology in $S$,
  and it holds that, for all $T>0$ and $r>0$, 
  \begin{equation}  \label{2. eq: convergence in hatC, uniform convergence on compact subsets}
    \lim_{\delta \to 0}
    \limsup_{n \to \infty}
    \sup_{\substack{ x_{n} \in \dom_{1}(L_{n})^{(r)}, \\ x \in \dom_{1}(L)^{(r)}, \\ d(x_{n},x) < \delta}}
    \sup_{0 \leq t \leq T}
    |L_{n}(x_{n}, t) - L(x, t)| 
    =0,
  \end{equation}
  where we recall the restriction $\cdot^{(r)}$ from \eqref{eq: restriction of a set to a closed ball}.
\end{dfn}

By \cite[Theorem~2.18]{Noda_pre_Convergence},
the compact-convergence topology with variable domains is Polish.
This topology is a natural extension of the compact-convergence topology on $C(S \times \RNp, \RNp)$ 
in the following sense:
the inclusion map 
\begin{equation}
  C(S \times \RNp, \RNp) \ni 
  L \mapsto L 
  \in \hatC(S \times \RNp, \RN)
\end{equation}
is a topological embedding, 
i.e., a homeomorphism onto its image
(see \cite[Corollary~2.21]{Noda_pre_Convergence}).
For a precompactness criterion and a tightness criterion for $\hatC(S \times \RNp, \RN)$,
see \cite[Theorems~2.23 and~2.24]{Noda_pre_Convergence}.

Given two maps $f \colon A \to B$ and $f' \colon A' \to B'$,
we define $f \times f' \colon A \times A' \to B \times B'$ by setting
\begin{equation}
  (f \times f') (a, a')
  \coloneqq
  (f(a), f'(a')).
\end{equation}
We write $\idty_{A}$ for the identity map from $A$ to itself.

Finally,
it is possible to define the space $\mbM_{L}$.
Let $\mbM_{L}^{\circ}$ be the collection of $(S, d, \rho, \mu, \pi)$ 
such that $(S, d, \rho, \mu) \in \mathbb{G}$ and $\pi$ 
is a probability measure on $D(\RNp, S) \times \hatC(S \times \RNp, \RN)$,
where we recall that $D(\RNp, S)$ denotes the space of cadlag functions with values in $S$
equipped with the usual $J_{1}$-Skorohod topology.
To introduce an equivalence relation on $\mbM_{L}^{\circ}$,
we need a preparation.
For a distance-preserving map $f \colon S_{1} \to S_{2}$
between boundedly-compact metric spaces,
we define
\begin{gather}
  \tau_{f}^{J_{1}} \colon D(\RNp, S_{1}) \ni 
    X \mapsto f \circ X 
  \in D(\RNp, S_{2}),
  \label{2. eq: functor for cadlag functions}\\
  \tau_{f}^{\hatC} \colon \hatC(S_{1} \times \RNp, \RN) \ni 
    L \mapsto L \circ (f^{-1} \times \idty_{\RNp}) 
  \in \hatC(S_{2} \times \RNp, \RN),
\end{gather}
where the inverse map $f^{-1}$ is restricted to $f(\dom_{1}(L))$ 
so that $L \circ (f^{-1} \times \idty_{\RNp})$ is well-defined.
We then define $\tau_{f}^{J_{1} \times \hatC} \coloneqq \tau_{f}^{J_{1}} \times \tau_{f}^{\hatC}$.
For $\cX_{i} = (S_{i}, d_{i}, \rho_{i}, \mu_{i}, \pi_{i}) \in \mbM_{L}^{\circ},\, i=1,2$,
we say that $\cX_{1}$ is $(\tau^{J_{1} \times \hatC})^{-1}$-equivalent to $\cX_{2}$ 
if and only if there exists a root-preserving isometry $f \colon S_{1} \to S_{2}$ 
such that $\mu_{2} = \mu_{1} \circ f^{-1}$ and $\pi_{2} = \pi_{1} \circ (\tau_{f}^{J_{1} \times \hatC})^{-1}$.

\begin{dfn} 
  We define $\mbM_{L}$ to be the collection of $(\tau^{J_{1} \times \hatC})^{-1}$-equivalence classes of elements in $\mbM_{L}^{\circ}$.
\end{dfn}

\begin{rem}
  For the same reason given in Remark \ref{rem: how to regard G as a set}, 
  we can safely regard $\mbM_{L}$ as a set.
\end{rem}

We equip a Polish topology on $\mbM_{L}$ defined in \cite[Definition~2.27 and Theorem~2.28]{Noda_pre_Convergence}.
In particular, the topology is characterized in terms of convergence, as follows.

\begin{thm} [{Convergence in $\mbM_L$, \cite[Theorem~2.29]{Noda_pre_Convergence}}] \label{2. thm: space M_L, convergence}
  Let $\cX=(S, d, \rho, \mu, \pi)$ 
  and $\cX_{n} = (S_{n},\allowbreak d^{n},\allowbreak \rho_{n},\allowbreak \mu_{n}, \pi_{n})$, $n \in \mathbb{N}$,
  be elements in $\mbM_{L}$.
  Then $\cX_{n}$ converges to $\cX$ in $\mbM_L$
  if and only if there exist a rooted boundedly-compact metric space $(M, d^{M}, \rho_{M})$ 
  and root-and-distance-preserving maps $f _n \colon S_{n} \to M$ and $f \colon S \to M$ such that 
  $f_{n}(S_{n}) \to f(S)$ in the Fell topology,
  $\mu_{n} \circ f_{n}^{-1} \to \mu \circ f^{-1}$ vaguely as measures on $M$,
  and $\pi_{n} \circ (\tau_{f_{n}}^{J_{1} \times \hatC})^{-1} \to \pi \circ (\tau_{f}^{J_{1} \times \hatC})^{-1}$
  weakly as probability measures on $D(\RNp, M) \times \hatC(M \times \RNp, \RN)$.
\end{thm}

Let us prepare to describe precompactness and tightness criteria.
For $\xi \in D(\RNp, S)$, where $(S, d)$ is a metric space,
we define  
\begin{equation}  \label{2. eq: modulus continuity for cadlag curves}
  \tilde{w}_{S}(\xi, h, t) 
  \coloneqq 
  \inf_{(I_{k}) \in \Pi_{t}} \max_{k} \sup_{r, s \in I_{k}} d(\xi(r), \xi(s)),
  \quad 
  t, h > 0,
\end{equation}
where $\Pi_{t}$ denotes the set of partitions of the interval $[0, t)$
into subintervals $I_{k} = [u, v)$ with $v - u \geq h$ when $v < t$.

\begin{thm} [{Precompactness in $\mbM_{L}$, \cite[Theorem~2.30]{Noda_pre_Convergence}}]  \label{2. thm: precpt in M_L}
  Let $(S_{n}, d^{n}, \rho_{n}, \mu_{n}, \pi_{n}),\, n \geq 1$, be elements of $\mbM_{L}$.
  For each $n \in \mathbb{N}$,
  let $(X_{n}, L_{n})$ be a random element of $D(\RNp, S_{n}) \times \hatC(S_{n} \times \RNp, \RN)$
  whose law coincides with $\pi_{n}$.
  We denote the underlying probability measure of $(X_{n}, L_{n})$ by $P_{n}$.
  Fix a dense set $I \subseteq \RNp$.
  Then the family $\{(S_{n}, d^{n}, \rho_{n}, \pi_{n})\}_{n \geq 1}$ is precompact in $\mbM_L$
  if and only if the following conditions are satisfied.
  \begin{enumerate} [label = \textup{(\roman*)}, leftmargin = *, series = precpt in M_L]
    \item \label{2. thm item: precpt in M_L, spaces are precompact}
      The family $\{(S_{n}, d^{n}, \rho_{n}, \mu_{n})\}_{n \geq 1}$ is precompact 
      in the local Gromov--Hausdorff-vague topology. 
    \item \label{2. thm item: precpt in M_L, values of processes are precompact} 
      For each $t \in I$,
      it holds that 
      $\displaystyle 
        \lim_{r \to \infty} 
        \limsup_{n \to \infty} 
        P_{n} \left( X_{n}(t) \notin S_{n}^{(r)} \right) = 0
      $.
    \item \label{2. thm item: precpt in M_L, equicontinuity of processes} 
      For each $t > 0$,
      it holds that, for all $\varepsilon > 0$, 
      $\displaystyle 
        \lim_{h \downarrow 0} 
        \limsup_{n \to \infty} 
        P_{n} \left( \tilde{w}_{S_{n}} (X_{n}, h, t) > \varepsilon \right) = 0
      $.
    \item \label{2. thm item: precpt in M_L, boundedness of local times} 
      For each $r>0$, 
      it holds that 
      $\displaystyle 
        \lim_{M \to \infty} 
        \limsup_{n \to \infty} 
        P_{n} \left( \sup_{x \in \dom_{1}(L_{n})^{(r)}} L_{n}(x, 0) > M \right) = 0.
      $
    \item \label{2. thm item: precpt in M_L, equicontinuity of local times} 
      For each $r>0$ and $T > 0$,
      it holds that, for all $\varepsilon > 0$,
      \begin{equation}
        \lim_{\delta \downarrow 0} 
        \limsup_{n \to \infty} 
        P_{n} 
        \left(
          \sup_{\substack{ x, y \in \dom_{1}(L_{n})^{(r)}, \\ d^{n}(x, y) < \delta}}
          \sup_{\substack{0 \leq s, t \leq T, \\ |t-s| < \delta}} 
          |L_{n}(x,t) - L_{n}(y,s)| 
          > 
          \varepsilon
        \right)
        = 
        0.
      \end{equation}
  \end{enumerate}
  In that case, the following result holds.
    \begin{enumerate} [resume* = precpt in M_L]
      \item \label{2. thm item: precpt in M_L, uniform boundedness of processes}
        For each $t \geq 0$, 
        it holds that 
        $\displaystyle
          \lim_{r \to \infty} 
          \limsup_{n \to \infty} 
          P_{n}\left( X_{n}(s) \notin S_{n}^{(r)}\ \text{for some}\ s \leq t \right) 
          = 0.
        $
    \end{enumerate}
\end{thm}

\begin{thm} [{Tightness in $\mbM_{L}$, \cite[Theorem~2.31]{Noda_pre_Convergence}}]  \label{2. thm: tightness in M_L}
  For each $n \in \mathbb{N}$,
  let $(S_{n}, d^{n}, \rho_{n}, \mu_{n}, \pi_{n})$ be a random element of $\mbM_{L}$ 
  built on a probability space $(\Omega_{n}, \mathcal{F}_{n}, \mathbf{P}_{n})$.
  For each $\omega \in \Omega_{n}$,
  let $(X_{n}^{\omega}, L_{n}^{\omega})$ 
  be a random element of $D(\RNp, S_{n}) \times \hatC(S_{n} \times \RNp, \RN)$
  whose law coincides with $\pi_{n}(\omega)$.
  We denote the underlying probability measure of $(X_{n}^{\omega}, L_{n}^{\omega})$ by $P_{n}^{\omega}$.
  Fix a dense set $I \subseteq \RNp$.
  Then the family $\{(S_{n}, d^{n}, \rho_{n}, \pi_{n})\}_{n \geq 1}$ is tight
  as random elements of $\mbM_L$ if and only if the following conditions are satisfied.
  \begin{enumerate} [label = \textup{(\roman*)}, leftmargin = *]
    \item \label{2. thm item: tightness in M_L, spaces are tight}
      The family $\{(S_{n}, d^{n}, \rho_{n}, \mu_{n})\}_{n \geq 1}$ is tight 
      as random elements of $\mathbb{G}$ in the local Gromov--Hausdorff-vague topology. 
    \item \label{2. thm item: tightness in M_L, values of processes are precompact} 
      For each $t \in T$,
      it holds that, for all $\varepsilon > 0$,
      $\displaystyle 
        \lim_{r \to \infty} 
        \limsup_{n \to \infty} 
        \mathbf{P}_{n}
        \left(
          P_{n}^{\omega} \left( X_{n}^{\omega}(t) \notin S_{n}^{(r)} \right) > \varepsilon
        \right)
        = 0
      $.
    \item \label{2. thm item: tightness in M_L, equicontinuity of processes} 
      For each $t > 0$,
      it holds that, for all $\varepsilon, \delta >0$,
      \begin{equation}
        \lim_{h \downarrow 0} 
        \limsup_{n \to \infty} 
        \mathbf{P}_{n}
        \left(
          P_{n}^{\omega} \left( \tilde{w}_{S_{n}} (X_{n}^{\omega}, h, t) > \varepsilon \right)
          > 
          \delta
        \right)
         = 0.
      \end{equation}
    \item \label{2. thm item: tightness in M_L, boundedness of local times} 
      For each $r>0$, 
      it holds that, for all $\varepsilon > 0$,
      \begin{equation}
        \lim_{M \to \infty} 
        \limsup_{n \to \infty} 
        \mathbf{P}_{n}
        \left(
          P_{n}^{\omega} \left( \sup_{x \in \dom(L_{n})^{(r)}} L_{n}^{\omega}(x, 0) > M \right) 
          > 
          \varepsilon
        \right)
         = 0.
      \end{equation}
    \item \label{2. thm item: tightness in M_L, equicontinuity of local times} 
      For each $r>0$ and $T > 0$,
      it holds that, for all $\varepsilon_{1}, \varepsilon_{2} > 0$,
      \begin{equation}
        \lim_{\delta \downarrow 0} 
        \limsup_{n \to \infty} 
        \mathbf{P}_{n}
        \left(
          P_{n}^{\omega} 
          \left(
            \sup_{\substack{ x, y \in \dom_{1}(L_{n})^{(r)}, \\ d^{n}(x, y) < \delta}}
            \sup_{\substack{0 \leq s, t \leq T, \\ |t-s| < \delta}} 
            |L_{n}^{\omega}(x,t) - L_{n}^{\omega}(y,s)| 
            > 
            \varepsilon_{1}
          \right)
          > 
          \varepsilon_{2}
        \right)
        = 
        0.
      \end{equation}
  \end{enumerate}
\end{thm}

Although it is not a space of our main interest,
we introduce another space $\mbM$ for convenience.
This space is used in the proofs of our main results.
Roughly speaking,
it is the space of measured metric spaces equipped with laws of stochastic processes,
and precisely defined as follows.
Let $\mbM^{\circ}$ be the collection of $(S, d, \rho, \mu, \pi')$ 
such that $(S, d, \rho, \mu) \in \mathbb{G}$ and $\pi'$ 
is a probability measure on $D(\RNp, S)$.
Recall $\tau^{J_{1}}$ from \eqref{2. eq: functor for cadlag functions}.
For $\cX_{i} = (S_{i}, d_{i}, \rho_{i}, \mu_{i}, \pi'_{i}) \in \mbM^{\circ},\, i=1,2$,
we say that $\cX_{1}$ is $(\tau^{J_{1}})^{-1}$-equivalent to $\cX_{2}$ 
if and only if there exists a root-preserving isometry $f \colon S_{1} \to S_{2}$ 
such that $\mu_{2} = \mu_{1} \circ f^{-1}$ and $\pi_{2} = \pi_{1} \circ (\tau_{f}^{J_{1}})^{-1}$.

\begin{dfn} [{The space $\mbM$}]
  We define $\mbM$ to be the collection of $(\tau^{J_{1}})^{-1}$-equivalence classes of elements in $\mbM^{\circ}$.
\end{dfn}

We equip a Polish topology on $\mbM_{L}$ defined in \cite[Definition~2.33 and Theorem~2.34]{Noda_pre_Convergence}.
In particular, the topology is characterized in terms of convergence, as follows.

\begin{thm} [{Convergence in $\mbM$, \cite[Theorem~2.35]{Noda_pre_Convergence}}] \label{2. thm: space M, convergence}
  Let $\cX=(S, d, \rho, \mu, \pi')$ 
  and $\cX_{n} = (S_{n},\allowbreak d^{n},\allowbreak \rho_{n},\allowbreak \mu_{n}, \pi'_{n})$, $n \in \mathbb{N}$,
  be elements in $\mbM$.
  Then $\cX_{n}$ converges to $\cX$ in $\mbM$
  if and only if there exist a rooted boundedly-compact metric space $(M, d^{M}, \rho_{M})$ 
  and root-and-distance-preserving maps $f _n \colon S_{n} \to M$ and $f \colon S \to M$ such that 
  $f_{n}(S_{n}) \to f(S)$ in the Fell topology,
  $\mu_{n} \circ f_{n}^{-1} \to \mu \circ f^{-1}$ vaguely as measures on $M$
  and $\pi'_{n} \circ (\tau_{f_{n}}^{J_{1}})^{-1} \to \pi' \circ (\tau_{f}^{J_{1}})^{-1}$
  weakly as probability measures on $D(\RNp, M)$.
\end{thm}

The following are analogues of Theorems~\ref{2. thm: precpt in M_L} and \ref{2. thm: tightness in M_L}, respectively.

\begin{thm} [{Precompactness in $\mbM$, \cite[Theorem~2.36]{Noda_pre_Convergence}}]  \label{2. thm: precompactness in M}
  Fix a sequence $((S_{n}, d^{n}, \rho_{n}, \mu_{n}, \pi'_{n}))_{n \geq 1}$ of elements of $\mbM$.
  Then the family $\{(S_{n}, d^{n}, \rho_{n}, \pi_{n})\}_{n \geq 1}$ is precompact 
  if and only if 
  the conditions \ref{2. thm item: precpt in M_L, spaces are precompact},
  \ref{2. thm item: precpt in M_L, values of processes are precompact} 
  and \ref{2. thm item: precpt in M_L, equicontinuity of processes} 
  of Theorem~\ref{2. thm: precpt in M_L} are satisfied.
\end{thm}

\begin{thm} [{Tightness in $\mbM$, \cite[Theorem~2.36]{Noda_pre_Convergence}}]  \label{2. thm: tightness in M}
  For each $n \in \mathbb{N}$,
  let $(S_{n}, d^{n}, \rho_{n}, \mu_{n}, \pi'_{n})$ 
  be a random element of $\mbM$.
  Then the family $\{(S_{n}, d^{n}, \rho_{n}, \pi_{n})\}_{n \geq 1}$ is tight
  if and only if 
  the conditions \ref{2. thm item: tightness in M_L, spaces are tight},
  \ref{2. thm item: tightness in M_L, values of processes are precompact} 
  and \ref{2. thm item: tightness in M_L, equicontinuity of processes} 
  of Theorem~\ref{2. thm: tightness in M_L}
  are satisfied.
\end{thm}

%%%%%%%%%%%%%%%%%%%%%%%%%%%%%%%%%%%%%%%%%%%%%%%%%%%%%%%%%%%%%%%%%%%%%%%%%%%%%%%%%%%%%%%%%%%%%%%%%%%%%%%%%%%%%%%%%%%%%%%%%%%%%%%%%%
%%%%%%%%%%%%%%%%%%%%%%%%%%%%%%%%%%%%%%%%%%%%%%%%%%%%%%%%%%%%%%%%%%%%%%%%%%%%%%%%%%%%%%%%%%%%%%%%%%%%%%%%%%%%%%%%%%%%%%%%%%%%%%%%%%
% Resistance forms and resistance metrics
%%%%%%%%%%%%%%%%%%%%%%%%%%%%%%%%%%%%%%%%%%%%%%%%%%%%%%%%%%%%%%%%%%%%%%%%%%%%%%%%%%%%%%%%%%%%%%%%%%%%%%%%%%%%%%%%%%%%%%%%%%%%%%%%%%
%%%%%%%%%%%%%%%%%%%%%%%%%%%%%%%%%%%%%%%%%%%%%%%%%%%%%%%%%%%%%%%%%%%%%%%%%%%%%%%%%%%%%%%%%%%%%%%%%%%%%%%%%%%%%%%%%%%%%%%%%%%%%%%%%%

\section{Resistance forms}  \label{sec: resistance forms}

We start with recalling some fundamental properties of resistance forms in Section~\ref{sec: preliminary of resistance forms}.
In Section~\ref{sec: extended Dirichlet spaces of resistance forms},
we introduce and study an important subspace of a resistance form,
which characterizes the associated extended Dirichlet space.
Using the results,
we study traces of resistance forms in Section~\ref{sec: traces of resistance forms}.

%%%%%%%%%%%%%%%%%%%%%%%%%%%%%%%%%%%%%%%%%%%%%%%%%%%%%%%%%%%%%%%%%%%%%%%%%%%%%%%%%%%%%%%%%%%%%%%%%%%%%%%%%%%%%%%%%%%%%%%%%%%%%%%%%%
% Preliminary
%%%%%%%%%%%%%%%%%%%%%%%%%%%%%%%%%%%%%%%%%%%%%%%%%%%%%%%%%%%%%%%%%%%%%%%%%%%%%%%%%%%%%%%%%%%%%%%%%%%%%%%%%%%%%%%%%%%%%%%%%%%%%%%%%%

\subsection{Preliminaries}  \label{sec: preliminary of resistance forms}

In this subsection,
we recall some basic properties of resistance forms and resistance metrics,
starting with their definitions.
The reader is referred to \cite{Kigami_12_Resistance} for further background.

\begin{dfn} [{Resistance form, \cite[Definition 3.1]{Kigami_12_Resistance}}] \label{3. dfn: resistance forms}
  Let $F$ be a non-empty set.
  A pair $(\mathcal{E}, \mathcal{F})$ is called a \textit{resistance form} on $F$ if it satisfies the following conditions.
  \begin{enumerate} [label=(RF\arabic*), leftmargin=*]
    \item \label{3. dfn item: the domain F}
          The symbol $\mathcal{F}$ is a linear subspace of $\{f \colon F \to \RN\}$ containing constant functions,
          and $\mathcal{E}$ is a non-negative symmetric bilinear form on $\mathcal{F}$
          such that $\mathcal{E}(f,f)=0$ if and only if $f$ is constant on $F$.
    \item \label{3. dfn item: the quotient space of resistance form is Hilbert}
          Let $\sim$ be the equivalence relation on $\mathcal{F}$ defined by saying $f \sim g$ if and only if $f-g$ is constant on $F$.
          Then $(\mathcal{F}/\sim, \mathcal{E})$ is a Hilbert space.
    \item \label{3. dfn item: separability of points}
          If $x \neq y$, then there exists an $f \in \mathcal{F}$ such that $f(x) \neq f(y)$.
    \item \label{3. dfn item: existence of resistance metric}
          For any $x, y \in F$,
          \begin{equation} \label{3. eq: definition of resistance metric}
            R_{(\cE, \cF)}(x,y)
            \coloneqq
            \sup
            \left\{
            \frac{|f(x) - f(y)|^{2}}{\mathcal{E}(f,f)}\,
            \middle |\,
            f \in \mathcal{F},\
            \mathcal{E}(f,f) > 0
            \right\}
            < \infty.
          \end{equation}     
    \item \label{3. dfn item: Markov property of resistance forms}
          If $\bar{f}\coloneqq  (f \wedge 1) \vee 0$,
          then $\bar{f} \in \mathcal{F}$ and $\mathcal{E}(\bar{f}, \bar{f}) \leq \mathcal{E}(f,f)$ for any $f \in \mathcal{F}$.
  \end{enumerate}
\end{dfn}

For later use, we prove a version of \ref{3. dfn item: Markov property of resistance forms}.

\begin{lem} \label{lem: min and max in resistance form}
  Let $(\cE, \cF)$ be a resistance form on a non-empty set $F$.
  Fix $f, g \in \cF$.
  Then $f \wedge g \in \cF$ and 
  \begin{equation}
    \cE(f \wedge g, f \wedge g) \leq \cE(f, f) + \cE(g, g).
  \end{equation}
  The same results hold when $f \wedge g$ is replaced with $f \vee g$.
\end{lem}

\begin{proof}
  One can readily verify that, for any real numbers $a, b, c, d$,
  \begin{equation}
    |a \wedge b - c \wedge d| \leq |a-c| \vee |b -d|.
  \end{equation}
  This implies that 
  \begin{equation}
    |f(x) \wedge g(x) - f(y) \wedge g(y)| \leq |f(x) - f(y)| \vee |g(x) - g(y)|,
    \qquad 
    \forall x,y \in F.
  \end{equation}
  In particular,
  \begin{equation}
    |f(x) \wedge g(x) - f(y) \wedge g(y)|^2 \leq |f(x) - f(y)|^2 + |g(x) - g(y)|^2,
    \qquad 
    \forall x, y \in F.
  \end{equation}
  Thus, we can follow the proof of \cite[Proposition~3.15]{Kigami_12_Resistance}
  to obtain the first assertion.
  The second assertion follows from the first assertion and the following relation:
  for any $a,b \in \RN$,
  \begin{equation}
    a \vee b = - \bigl( (-a) \wedge (-b) \bigr).
  \end{equation}
\end{proof}

For the following definition,
recall the effective resistance on an electrical network with a finite vertex set 
from \cite[Section~9.4]{Levin_Peres_17_Markov}
(see also \cite[Section~2.1]{Kigami_01_Analysis}).

\begin{dfn} [{Resistance metric, \cite[Definition 2.3.2]{Kigami_01_Analysis}}]
  \label{3. dfn: resistance metrics}
  A metric $R$ on a non-empty set $F$ is called a \textit{resistance metric}
  if and only if,
  for any non-empty finite subset $V \subseteq F$,
  there exists an electrical network $G$ with the vertex set $V$ 
  such that the effective resistance on $G$ coincides with $R|_{V \times V}$.
\end{dfn}

\begin{thm} [{\cite[Theorem 2.3.6]{Kigami_01_Analysis}}]  \label{3. thm: one-to-one correspondence of fomrs and metrics}
  There exists a one-to-one correspondence between resistance forms $(\cE, \cF)$ on $F$ 
  and resistance metrics $R$ on $F$ via $R = R_{(\cE, \cF)}$.
  In other words,
  a resistance form $(\cE, \cF)$ is characterized by $R_{(\cE, \cF)}$ 
  given in \ref{3. dfn item: existence of resistance metric}.
\end{thm}

In Assumptions~\ref{1. assum: deterministic version}\ref{1. assum item: deterministic, the non-explosion condition} 
and~\ref{1. assum: random version}\ref{1. assum item: random, the non-explosion condition},
we consider effective resistance between sets.
This is precisely defined below.

\begin{dfn} [{Effective resistance between sets}]  \label{3. dfn: effective resistance between sets}
  Fix a resistance form $(\cE, \cF)$ on $F$
  and write $R$ for the corresponding resistance metric.
  For sets $A, B \subseteq F$,
  we define
  \begin{equation}
    R(A,B)
    \coloneqq
    \left(
    \inf\{
    \mathcal{E}(f,f):
    f \in \mathcal{F},\
    f|_{A}=1,\
    f|_{B}=0
    \}
    \right)^{-1},
  \end{equation}
  which is defined to be zero if the infimum is taken over the empty set.
  Note that by \ref{3. dfn item: existence of resistance metric} 
  we clearly have $R( \{x\}, \{y\}) = R(x,y)$.
\end{dfn}

Fix a resistance form $(\cE, \cF)$ on a non-empty set $F$
and write $R$ for the corresponding resistance metric.
We equip $F$ with the topology induced from $R$.
Since the condition \ref{3. dfn item: existence of resistance metric} implies that 
\begin{equation}  \label{3. eq: the basic inequality for resistance forms}
  | f(x) - f(y) |^{2}
  \leq
  \cE(f,f) R(x,y),
  \quad 
  \forall f \in \cF,
\end{equation}
$\cF$ is a subset of $C(F, \RN)$.
We will henceforth assume that 
$(F, R)$ is locally compact and separable,
and the resistance form $(\mathcal{E}, \mathcal{F})$ is regular, as described by the following.

\begin{dfn} [{Regular resistance form, \cite[Definition 6.2]{Kigami_12_Resistance}}] \label{3. dfn: regular resistance forms}
  Let $C_{c}(F, \RN)$ be the space of compactly supported, continuous functions from $F$ to $\RN$,
  equipped with the supremum norm $\| \cdot \|_{\infty}$.
  A resistance form $(\mathcal{E}, \mathcal{F})$ on $F$ is called \textit{regular}
  if and only if $\mathcal{F} \cap C_{c}(F, \RN)$ is dense in $C_{c}(F, \RN)$ with respect to $\| \cdot \|_{\infty}$.
\end{dfn}

We next introduce related Dirichlet forms and stochastic processes.
First, suppose that we have a Radon measure $\mu$ of full support on $F$.
Let $\cB(F)$ be the Borel $\sigma$-algebra on $F$
and $\cB^{\mu}(F)$ be the completion of $\cB(F)$ with respect to $\mu$.

\begin{dfn} [{The spaces $\cL(F, \mu)$ and $L^{2}(F, \mu)$}]  \label{3. dfn: the spaces cL and L^2}
  Two extended real-valued functions are said to be $\mu$-equivalent
  if they coincide outside a $\mu$-null set.
  We define $\cL(F, \mu)$ to be the space of $\mu$-equivalence classes 
  of $\cB^{\mu}(F)$-measurable extended real-valued functions.
  The $L^{2}$-space $L^{2}(F,\mu)$ is the subspace of $\cL(F, \mu)$ 
  consisting of square-integrable functions equipped with the usual $L^{2}$-norm.
\end{dfn}

Now, we define a bilinear form $\mathcal{E}_{1}$ on $\mathcal{F} \cap L^{2}(F, \mu)$
by setting
\begin{equation} \label{eq: inner product E_1}
  \mathcal{E}_{1}(f,\, g)
  \coloneqq
  \mathcal{E}(f,\, g)
  +
  \int_{F}fg\, d\mu.
\end{equation}
Then $(\mathcal{F} \cap L^{2}(F, \mu), \mathcal{E}_{1})$ is a Hilbert space (see \cite[Theorem~2.4.1]{Kigami_01_Analysis}).
The Dirichlet form $(\cE, \cD)$ on $L^{2}(F, \mu)$ and the extended Dirichlet space $\cD_{e}$ 
associated with $(F, R, \mu)$ are given as follows.

\begin{dfn} [{\cite{Chen_Fukushima_12_Symmetric,Fukushima_Oshima_Takeda_11_Dirichlet}}]
  \label{3. dfn: Dirichlet form and extended Dirichlet space}
  We define the Dirichlet form $(\cE, \cD)$
  by setting $\mathcal{D}$ to be the closure of $\mathcal{F} \cap C_{c}(F, \RN)$ 
  in $(\mathcal{F} \cap L^{2}(F, \mu), \mathcal{E}_{1})$.
  The extended Dirichlet space $\mathcal{D}_{e}$ of $(\cE, \cD)$ is the subspace of $\cL(F, \mu)$
  consisting of $f$ such that $|f| < \infty$, $\mu$-a.e.\
  and there exists an $\mathcal{E}$-Cauchy sequence $(f_{n})_{n \geq 0}$ in $\mathcal{D}$ with $f_{n}(x) \to f(x)$, 
  $\mu$-a.e.\ $x$.
\end{dfn}

Since we assume that the resistance form $(\mathcal{E},\mathcal{F})$ is regular,
we have from \cite[Theorem~9.4]{Kigami_12_Resistance}
that the associated Dirichlet form $(\mathcal{E}, \mathcal{D})$ is regular 
(see \cite{Fukushima_Oshima_Takeda_11_Dirichlet} for the definition of a regular Dirichlet form).
Moreover, standard theory gives us the existence of an associated Hunt process $((X_{t} )_{t \geq 0}, (P_{x} )_{x \in F} )$
(e.g. \cite[Theorem~7.2.1]{Fukushima_Oshima_Takeda_11_Dirichlet}),
which we refer to as the Hunt process associated with $(F, R, \mu)$.
Note that such a process is, in general, only specified uniquely for starting points outside a set of zero capacity.
However, in this setting,
every point has strictly positive capacity (see \cite[Theorem~9.9]{Kigami_12_Resistance}),
and so the process is defined uniquely everywhere.

\begin{rem}
  In \cite[Chapter~9]{Kigami_12_Resistance},
  in addition to the above assumptions,
  $(F, R)$ is assumed to be complete,
  but it is easy to remove this assumption.
\end{rem}

We now recall the definition of local times.
Let $(\Omega, \mathcal{F})$ be the measurable space where the probability measures $(P_{x})_{x \in F}$ are defined.
We denote the minimum completed admissible filtration of the Hunt process $X$ by $(\mathcal{F}_{t})_{t \geq 0}$,
the family of the translation (shift) operators for $X$ by $(\theta_{t})_{t \geq 0}$
and the lifetime of $X$ by $\zeta$
(see \cite{Fukushima_Oshima_Takeda_11_Dirichlet} for these definitions).

\begin{dfn} [{PCAF and local time}]
  A non-decreasing, continuous, $(\mathcal{F}_{t})_{t \geq 0}$-adapted process $A = (A_{t})_{t \geq 0}$ on $(\Omega, \mathcal{F})$
  is called a \textit{positive continuous additive functional (PCAF)} of $X$ 
  if for all $x \in F$ it holds $P_{x}$-a.s.\ that $A_{0} = 0$, $A_{t} = A_{\zeta}$ for all $t \geq \zeta$ 
  and $A_{s+t} = A_{s} + A_{t} \circ \theta_{s}$ for all $s,t \geq 0$.
  A PCAF $A = (A_{t})_{t \geq 0}$ of $X$ is called a \textit{local time} of $X$ at $x \in F$ 
  if $P_{x}(T_{A} = 0) = 1$ and $P_{y}(T_{A} = 0) = 0$ for all $y \neq x$,
  where we set $T_{A}(\omega) \coloneqq \inf \{t \geq 0 : A_{t}(\omega) > 0\}$.
\end{dfn}

The metric-entropy condition \eqref{1. eq: definition of vF} that assumed for the spaces in $\check{\mathbb{F}}$ 
implies that the associated Hunt process admits a jointly continuous local time.

\begin{prop}  [{\cite[p.~18 and Corollary~4.16]{Noda_pre_Convergence}}]
  In the above setting,
  the Hunt process $X$ admits a jointly measurable local time $L=(L(x, t))_{x \in F, t \geq 0}$
  satisfying the occupation density formula,
  that is,
  it holds that, for all $x \in F$, $t \geq 0$ and all non-negative measurable functions $f \colon F \to \RNp$,
  \begin{equation} \label{3. eq: the occupation density formula}
    \int_{0}^{t} f(X_{s})ds = \int_{F}f(y) L_{t}(y)\, \mu(dy), \quad
    P_{x} \text{-a.s.}
  \end{equation}
  Moreover,
  if the resistance metric space $(F, R)$ is boundedly compact and satisfies \eqref{1. eq: definition of vF},
  then the local time $L$ can be chosen so that $L$ is jointly continuous on $F \times \RNp$,
  $P_{x}$-a.s.\ for all $x \in F$.
\end{prop}

The recurrence and transience of the Dirichlet form $(\cE, \cD)$ were studied in \cite{Croydon_18_Scaling},
and we have the following characterization in terms of the resistance metric.

\begin{lem} [{\cite[Lemma~2.3]{Croydon_18_Scaling}}] \label{3. lem: characterization of recurrence by resistance}
  In the above setting,
  assume that $(F, R)$ is boundedly compact.
  Then the associated regular Dirichlet form $(\mathcal{E}, \mathcal{D})$ is recurrent 
  in the sense of \cite[p.~55]{Fukushima_Oshima_Takeda_11_Dirichlet}
  if and only if 
  \begin{equation}  \label{3. eq: the resistance tends to infinity}
    \lim_{r \to \infty} 
    R(\rho, B_{R}(\rho, r)^{c}) 
    = 
    \infty.
  \end{equation}
  for some (or equivalently, any) $\rho \in F$.
\end{lem}

By Lemma~\ref{3. lem: characterization of recurrence by resistance},
the recurrence of the Dirichlet form $(\cE, \cD)$ is independent of the measure $\mu$
and is characterized by the condition \eqref{3. eq: the resistance tends to infinity}.
Therefore, it is natural to introduce the notion of recurrent resistance forms and resistance metrics as follows.

\begin{dfn} [{Recurrent resistance form and resistance metric}]\label{3. dfn: recurrent resistance forms}
  Let $(\cE, \cF)$ be a resistance form on $F$ and write $R$ for the corresponding resistance metric.
  We say that $(\cE, \cF)$ and $R$ are \textit{recurrent} 
  if and only if 
  the following condition is satisfied:
  \begin{enumerate} [label = \textup{(RRF)}, leftmargin = *]
    \item \label{item: recurrence of resistance form}
      there exists an increasing sequence $(U_n)_{n \geq 1}$ of relatively compact open subsets of $F$
      such that $\bigcup_{n \geq 1} U_n = F$ and 
      \begin{equation}
        \lim_{n \to \infty} R(\rho, U_n^c) = \infty
      \end{equation}
      for some $\rho \in F$.
  \end{enumerate}
\end{dfn}

\begin{rem}
  Since we do not assume that $(F, R)$ is boundedly compact in Definition~\ref{3. dfn: recurrent resistance forms},
  we use the sequence $(U_n)_{n \geq 1}$ instead of the balls $B_R(\rho, r)$.
  It is straightforward to see that the definition does not depend on the particular choice of $(U_n)_{n \geq 1}$ or of $\rho$.
  Namely, if $(\cE, \cF)$ and $R$ satisfy \ref{item: recurrence of resistance form},
  then, for any increasing sequence $(\tilde{U}_n)_{n \geq 1}$ of relatively compact open subsets of $F$,
  we have 
  \begin{equation}
     \lim_{n \to \infty} R( \tilde{\rho}, \tilde{U}_n^c) = \infty
  \end{equation}
  for any $\tilde{\rho} \in F$.
\end{rem}

\begin{rem}
  By \ref{item: recurrence of resistance form},
  the recurrence of a resistance metric $R$ implies that the induced topology is separable and locally compact.
\end{rem}

Using the notion of recurrent resistance forms and resistance metrics introduced in Definition~\ref{3. dfn: recurrent resistance forms},
we establish an extension of Lemma~\ref{3. lem: characterization of recurrence by resistance}
to resistance metric spaces that are not necessarily boundedly compact, 
in Corollary~\ref{3. cor: recurrent Dirichlet and resistance forms} in the next subsection.

%%%%%%%%%%%%%%%%%%%%%%%%%%%%%%%%%%%%%%%%%%%%%%%%%%%%%%%%%%%%%%%%%%%%%%%%%%%%%%%%%%%%%%%%%%%%%%%%%%%%%%%%%%%%%%%%%%%%%%%%%%%%%%%%%%
% Extended Dirichlet spaces of resistance forms
%%%%%%%%%%%%%%%%%%%%%%%%%%%%%%%%%%%%%%%%%%%%%%%%%%%%%%%%%%%%%%%%%%%%%%%%%%%%%%%%%%%%%%%%%%%%%%%%%%%%%%%%%%%%%%%%%%%%%%%%%%%%%%%%%%

\subsection{Extended Dirichlet spaces of resistance forms}  \label{sec: extended Dirichlet spaces of resistance forms}

Recall from Section~\ref{sec: preliminary of resistance forms}
that a resistance form equipped with a measure determines a Dirichlet form and an extended Dirichlet space.
In this section,
we prove that the extended Dirichlet space is independent of the measure
and the extended Dirichlet space coincides with the domain of the resistance form 
if and only if the resistance form is recurrent
(see Theorems~\ref{3. thm: characterization of the extended Dirichlet space} 
and~\ref{3. thm: recurrence characterization with F^1} below).
This coincidence of domains will be useful 
for studying traces of resistance forms in Section~\ref{sec: traces of resistance forms}.

Throughout this subsection,
we fix a resistance form $(\mathcal{E}, \mathcal{F})$ on a non-empty set $F$,
and write $R$ for the associated resistance metric.
Fix an element $x_{0} \in F$ and 
define an inner product $\mathcal{E}^{(1)}$ on $\mathcal{F}$ 
by setting 
\begin{equation}
  \mathcal{E}^{(1)} (u,v)
  \coloneqq
  \mathcal{E}(u,v) + u(x_{0})v(x_{0}).
\end{equation}
Note that the inner product $\cE^{(1)}$ differs from $\cE_1$ defined in \eqref{eq: inner product E_1}.
It is easy to check the following result using \ref{3. dfn item: the quotient space of resistance form is Hilbert}.

\begin{lem} \label{3. lem: (F, E^1) is Hilbert}
  The inner product space $(\cF, \cE^{(1)})$ is a Hilbert space.
\end{lem}

\begin{dfn} [{The space $\cF^{(1)}$}]
  We define $\cF^{(1)}$ to be the closure of $C_{c}(F, \RN) \cap \mathcal{F}$ in $(\mathcal{F}, \mathcal{E}^{(1)})$.
\end{dfn}

\begin{rem}
  Using \eqref{3. eq: the basic inequality for resistance forms},
  one can check that the space $\cF^{(1)}$ is independent of the choice of $x_{0}$.
\end{rem}

Below, we verify that the topology on $\cF^{(1)}$ induced by $\cE^{(1)}$ 
is stronger than the compact-convergence topology.

\begin{prop}  \label{3. prop: E^1 convergence implies compact convergence}
  If a sequence $(u_{n})_{n \geq 1}$ in $\cF^{(1)}$ converges to $u \in \cF^{(1)}$ with respect to $\cE^{(1)}$,
  then $u_{n} \to u$ uniformly on every compact subset of $F$.
\end{prop}

\begin{proof}
  Fix a non-empty compact subset $K \subseteq F$. 
  Using \eqref{3. eq: the basic inequality for resistance forms},
  we deduce that, for any $x \in K$,
  \begin{equation}
    |u_{n}(x) - u(x)|
    \leq 
    \sqrt{\cE(u_{n}-u, u_{n}-u) R(x, x_{0})} + |u_{n}(x_{0}) - u(x_{0})|.
  \end{equation}
  Since we have that $\sup_{x \in K} R(x, x_{0}) < \infty$,
  we obtain the desired result.
\end{proof}

In Theorem~\ref{3. thm: characterization of the extended Dirichlet space} below, 
we assume that the resistance form $(\cE, \cF)$ is regular
and that $(F, R)$ is separable and locally compact.
Let $\mu$ be a Radon measure on $(F, R)$ of full support.
We write $(\cE, \cD)$ for the Dirichlet form associated with $(F, R, \mu)$ 
and $\cD_{e}$ for the extended Dirichlet space of $(\cE, \cD)$.
From \cite[Theorem 2.1.7]{Fukushima_Oshima_Takeda_11_Dirichlet} and \cite[Proposition 9.13]{Kigami_12_Resistance},
every function $u \in \mathcal{D}_{e}$ has a continuous modification,
and so we may regard $\mathcal{D}_{e}$ as a subspace of $C(F, \RN)$.

\begin{thm}  \label{3. thm: characterization of the extended Dirichlet space}
  In the above setting,
  it holds that $\mathcal{D}_{e} = \cF^{(1)}$.
\end{thm}

\begin{proof}
  We first show $\mathcal{D}_{e} \subseteq \cF^{(1)}$.
  Fix a continuous function $u \in \mathcal{D}_{e}$ and 
  let $(u_{n})_{n \geq 1}$ be an $\mathcal{E}$-Cauchy sequence in $\mathcal{D}$ 
  such that 
  $u_{n}(x) \to u(x),$ $\mu$-a.e.\ $x$.
  By the definition of $\mathcal{D}$ (see Definition~\ref{3. dfn: Dirichlet form and extended Dirichlet space}),
  we may assume that $u_{n} \in C_{c}(F, \RN) \cap \mathcal{F}$.
  Choose $x_{1} \in F$ such that $u_{n}(x_{1}) \to u(x_{1})$.
  By \ref{3. dfn item: the quotient space of resistance form is Hilbert},
  we can find $v \in \mathcal{F}$ such that 
  \begin{equation}
    \mathcal{E}(u_{n}-v, u_{n}-v) \to 0, \quad
    v(x_{1}) = u(x_{1}).
  \end{equation}
  It is then the case that $v$ is an element of $\cF^{(1)}$.
  Since we have from \eqref{3. eq: the basic inequality for resistance forms} that 
  \begin{align} \label{3. eq: approximation of function from one-point convergence}
    |u_{n}(x) - v(x)|
    \leq 
    \sqrt{R(x,x_{1}) \mathcal{E}(u_{n} -v, u_{n} -v)} + |u_{n}(x_{1}) - v(x_{1})|,
  \end{align}
  it follows that $u_{n}(x) \to v(x)$ for all $x \in F$.
  Therefore $u = v$, $\mu$-a.e. 
  The continuity of $u$ and $v$ and 
  the fact that $\mu$ is of full support yield $u=v \in \cF^{(1)}$.
  The other inclusion $\cF^{(1)} \subseteq \mathcal{D}_{e}$
  is easy to prove using a similar estimate to \eqref{3. eq: approximation of function from one-point convergence}.
\end{proof}

By Theorem~\ref{3. thm: characterization of the extended Dirichlet space},
the study of the extended Dirichlet space reduces to that of the resistance form.
Henceforth, we return to the initial setting.
In other words, $(\cE, \cF)$ is simply a resistance form,
which is not necessarily regular,
and the corresponding resistance metric is not necessarily separable or locally compact.

\begin{thm} \label{3. thm: recurrence characterization with F^1}
  The following statements are equivalent.
  \begin{enumerate} [label = \textup{(\roman*)}]
    \item \label{3. thm item: recurrence, original dfn}
      The resistance form $(\cE, \cF)$ is recurrent in the sense of Definition~\ref{3. dfn: recurrent resistance forms}.
    \item \label{3. thm item: recurrence, F^1 contains 1}
      It holds that $1_{F} \in \cF^{(1)}$.
    \item \label{3. thm item: recurrence, F equals to F^1}
      It holds that $\cF = \cF^{(1)}$.
  \end{enumerate}
\end{thm}

Before proving the above theorem,
we provide a corollary of this,
which is an extension of Lemma~\ref{3. lem: characterization of recurrence by resistance}
to the setting where the resistance metric space is not necessarily boundedly compact.

\begin{cor} \label{3. cor: recurrent Dirichlet and resistance forms}
  Let $\mu$ be a fully-supported Radon measure on $F$ 
  and assume that $(F, R)$ is separable and locally compact.
  Then the associated Dirichlet form is recurrent if and only if the resistance form is recurrent.
\end{cor}

\begin{proof}
  By \cite[Theorem~1.6.3]{Fukushima_Oshima_Takeda_11_Dirichlet} and \ref{3. dfn item: the domain F},
  the recurrence of the Dirichlet form $(\cE, \cD)$ is equivalent to the condition that $1_F \in \cF^{(1)}$.
  Thus, the desired result follows from Theorem~\ref{3. thm: recurrence characterization with F^1}.
\end{proof}

To prove Theorem~\ref{3. thm: recurrence characterization with F^1},
we introduce the notion of a \textit{Ces\`{a}ro mean sequence}.
Given a sequence $(u_{n})_{n \geq 1}$ in $C(F, \RN)$,
the Ces\`{a}ro mean sequence $(v_{n})_{n \geq 1}$ in $C(F, \RN)$
is defined by setting 
\begin{equation}
  v_{n}(x)
  \coloneqq
  \frac{1}{n} \sum_{l =1}^{n} u_{l}(x).
\end{equation}

\begin{prop}  \label{3. prop: find a convergent Cesaro mean sequence}
  If $u_{n} \in \mathcal{F}$ converges to a function $u$ on $F$ pointwise
  and $\sup_{n} \mathcal{E}(u_{n}, u_{n}) < \infty$,
  then $u \in \mathcal{F}$ and a Ces\`{a}ro mean sequence 
  of a suitable subsequence of $(u_{n})_{n \geq 1}$ converges to $u$ with respect to $\mathcal{E}^{(1)}$.
\end{prop}

\begin{proof}
  By Lemma~\ref{3. lem: (F, E^1) is Hilbert} and the assumption $\sup_{n} \mathcal{E}^{(1)}(u_{n}, u_{n}) < \infty$,
  we may apply the Banach-Saks theorem (cf.\ \cite[Theorem~A.4.1]{Chen_Fukushima_12_Symmetric})
  to obtain
  a Ces\`{a}ro mean sequence $(w_{n(k)})_{k \geq 1}$
  of a suitable subsequence $(u_{n(k)})_{k \geq 1}$
  converges to some $w \in \mathcal{F}$ with respect to $\mathcal{E}^{(1)}$.
  Using the pointwise convergence $u_{n(k)} \to u$,
  we deduce that $w_{n(k)} \to u$ pointwise,
  which implies $w=u$.
  Now the desired result is immediate.
\end{proof}

As consequences of Proposition~\ref{3. prop: find a convergent Cesaro mean sequence},
we obtain two corollaries below
that are useful to approximate functions in $\cF^{(1)}$ 
by more tractable functions.

\begin{cor} \label{3. cor: find a uniformly bounded sequence for a function in F^1}
  Suppose that $u \in \cF^{(1)}$ and $\| u\|_{\infty} < \infty$,
  where we recall that $\| \cdot \|_{\infty}$ denotes the supremum norm.
  Then there exists a sequence $(u_{n})_{n \geq 1}$ in $C_{c}(F, \RN) \cap \mathcal{F}$ 
  such that $u_{n} \to u$ with respect to $\mathcal{E}^{(1)}$ and 
  $|u_n(x)| \leq |u(x)|$ for all $x \in F$ and $n \geq 1$.
\end{cor}

\begin{proof}
  Choose $u_{n} \in C_{c}(F, \RN) \cap \mathcal{F}$ 
  such that $u_{n} \to u$ with respect to $\mathcal{E}^{(1)}$.
  Define $\tilde{u}_{n} \in C_{c}(F, \RN) \cap \mathcal{F}$ by setting 
  \begin{equation}
    \tilde{u}_{n}
    \coloneqq
    \bigl( (- |u|) \vee u_n \bigr) \wedge |u|.
  \end{equation}
  Clearly, $|\tilde{u}_n(x)| \leq |u(x)|$ for all $x \in F$ and $n \geq 1$.
  We then have that $\tilde{u}_{n} \to u$ pointwise and 
  Lemma~\ref{lem: min and max in resistance form} yields that
  \begin{equation}
    \sup_{n} \mathcal{E}(\tilde{u}_{n}, \tilde{u}_{n})
    \leq 
    \sup_{n} \mathcal{E}(u_{n}, u_{n})
    + 
    2\, \cE(|u|, |u|)
    < \infty.
  \end{equation}
  Therefore, the desired result follows from Proposition~\ref{3. prop: find a convergent Cesaro mean sequence}.
\end{proof}

\begin{cor} \label{3. cor: find a bounded sequence for a function in F^1}
  For any $u \in \mathcal{F}$,
  there exists a sequence $(u_{n})_{n \geq 1}$ in $\mathcal{F}$
  such that $u_{n} \to u$ with respect to $\mathcal{E}^{(1)}$ 
  and $|u_{n}(x)| \leq n \wedge |u(x)|$ for all $x \in F$.
\end{cor}

\begin{proof}
  Set $u_{n} \coloneqq ((-n) \vee u) \wedge n$.
  The result is proven 
  by the same argument as Corollary~\ref{3. cor: find a uniformly bounded sequence for a function in F^1}.
\end{proof}

We are ready to prove Theorem~\ref{3. thm: recurrence characterization with F^1}.

\begin{proof} [{Proof of Theorem~\ref{3. thm: recurrence characterization with F^1}}]
  Assume \ref{3. thm item: recurrence, original dfn}.
  Then, by definition, 
  we can find an increasing sequence $(U_n)_{n \geq 1}$ of relatively compact open subsets 
  and functions $\varphi_n \in \cF$ such that 
  $\bigcup_{n \geq 1} U_n = F$,
  $\varphi_n(x_{0}) = 1$,
  $\varphi_n|_{U_n^c} = 0$, 
  and $\cE(\varphi_n, \varphi_n) \to 0$ as $n \to \infty$.
  In particular, 
  $\varphi_n \in C_c(F, \RN)$ for all $n$,
  and 
  $\cE(1_{F} - \varphi_n, 1_{F} - \varphi_n) \to 0$ as $n \to \infty$,
  which implies \ref{3. thm item: recurrence, F^1 contains 1}.

  Next, assume \ref{3. thm item: recurrence, F^1 contains 1}.
  Using Corollary~\ref{3. cor: find a uniformly bounded sequence for a function in F^1},
  we can find functions $\varphi_{n} \in C_{c}(F, \RN) \cap \mathcal{F}$ such that 
  $\varphi_{n} \to 1_{F}$ with respect to $\mathcal{E}^{(1)}$ and $\| \varphi_{n} \|_{\infty} \leq 1$.
  Fix $u \in \mathcal{F}$.
  The inequality in \cite[Lemma 6.5]{Kigami_12_Resistance} yields that 
  \begin{equation}
    \mathcal{E}(u \cdot \varphi_{n}, u \cdot \varphi_{n})
    \leq 
    2\|u\|_{\infty} \mathcal{E}(\varphi_{n}, \varphi_{n})
    + 
    2\|\varphi_{n}\|_{\infty} \mathcal{E}(u, u).
  \end{equation}
  If $u$ is a bounded function,
  then the right-hand side of the above inequality is uniformly bounded.
  Hence we deduce that $u \in \cF^{(1)}$ by Proposition~\ref{3. prop: find a convergent Cesaro mean sequence}.
  When $u$ is not bounded,
  we use Corollary~\ref{3. cor: find a bounded sequence for a function in F^1}
  and choose functions $u_{n} \in \mathcal{F}$ 
  such that $u_{n} \to u$ with respect to $\mathcal{E}^{(1)}$ and $\|u_{n}\|_{\infty} < \infty$.
  Since we have that $u_{n} \in \cF^{(1)}$ and $\cF^{(1)}$ is closed with respect to $\mathcal{E}^{(1)}$
  (by its definition),
  we obtain that $u \in \cF^{(1)}$.
  Therefore, we deduce that $\cF \subseteq \cF^{(1)}$,
  which implies \ref{3. thm item: recurrence, F equals to F^1}.

  Finally, assume \ref{3. thm item: recurrence, F equals to F^1}.
  It is then the case that $1_{F} \in \cF^{(1)}$,
  and hence it follows from the definition of $\cF^{(1)}$ that 
  there exists $\varphi_{n} \in C_{c}(F, \RN) \cap \cF$ such that 
  $\varphi_{n} \to 1_{F}$ with respect to $\cE^{(1)}$.
  If necessary,
  by considering sufficiently large $n$ with $\varphi_n(x_0) > 0$
  and replacing $\varphi_{n}$ by $\varphi_{n}/\varphi_{n}(x_{0})$,
  we may assume that $\varphi_{n}(x_{0}) = 1$ for all $n$.
  By Proposition~\ref{3. prop: find a convergent Cesaro mean sequence},
  we can find a subsequence $(\varphi_{n(k)})_{k \geq 1}$ 
  whose Ces\`{a}ro mean sequence $(\psi_k)_{k \geq 1}$ converges to $1_F$ with respect to $\cE^{(1)}$.
  Note that $\psi_k(x_0) = 1$ for all $k$.
  Write $U_k$ for the interior of the support of $\psi_k$.
  Taking the Ces\`{a}ro mean sequence ensures that the sequence $(U_k)_{k \geq 1}$ is increasing.
  We have from Proposition~\ref{3. prop: E^1 convergence implies compact convergence} that 
  $\psi_k \to 1_F$ in the compact-convergence topology.
  This implies that 
  $\bigcup_{k \geq 1} U_k = F$.
  Moreover, by Definition~\ref{3. dfn: effective resistance between sets},
  we have that 
  \begin{equation}
    R(x_{0}, U_k^{c}) 
    \geq 
    \cE(\psi_k, \psi_k)^{-1},
    \qquad 
    \forall k \geq 1.
  \end{equation}
  Since $\psi_k \to 1_F$ with respect to $\cE$ 
  and $\cE(1_F, 1_F) = 0$,
  we deduce that 
  \begin{equation}
    \lim_{k \to \infty} R(x_0, U_k^c) = \infty.
  \end{equation}
  Therefore, we obtain \ref{3. thm item: recurrence, original dfn}.
\end{proof}

As another corollary of Theorem~\ref{3. thm: recurrence characterization with F^1} other than Corollary~\ref{3. cor: recurrent Dirichlet and resistance forms},
we obtain the following.

\begin{cor} \label{3. cor: recurrence implies regularity}
  If the resistance form $(\cE, \cF)$ is recurrent,
  then the resistance form is regular.
\end{cor}

\begin{proof}
  Fix $u \in C_{c}(F, \RN)$ and $\varepsilon > 0$.
  It suffices to find a function $w \in C_c(F, \RN) \cap \cF$ such that $\|u - w\|_\infty < \varepsilon$.
  By Theorem~\ref{3. thm: recurrence characterization with F^1},
  there exists a sequence $(\varphi_{n})_{n \geq 1}$ in $C_{c}(F, \RN) \cap \cF$ 
  such that $\varphi_{n} \to 1_{F}$ with respect to $\cE^{(1)}$.
  It follows from Proposition~\ref{3. prop: E^1 convergence implies compact convergence} that 
  \begin{equation}  
    c_{n} \coloneqq \inf_{x \in \supp(u)} \varphi_{n}(x) > 0
  \end{equation}
  for all sufficiently large $n$,
  where $\supp(\cdot)$ denotes the support of functions.
  Fix such an $n$ and write $K \coloneqq \supp(\varphi_{n})$.
  Note that $\supp(u) \subseteq K$.
  Define $\psi \in C_{c}(F, \RN) \cap \cF$ by setting 
  $\psi(x) \coloneqq 0 \vee (c_{n}^{-1} \varphi_{n}(x) \wedge 1)$.
  Note that $\psi|_{\supp(u)} \equiv 1$ and $\| \psi \|_{\infty} \leq 1$.
  A general version of 
  the Stone Weierstrass theorem 
  (c.f.\ \cite[Theorem~2.4.11]{Dudley_02_real}) yields that 
  $\{v|_{K} \mid v \in C_{c}(F, \RN) \cap \cF \}$ is dense in $C_{c}(K, \RN)$
  (see the proof of \cite[Theorem~6.3]{Kigami_12_Resistance}).
  Thus, we can find $v \in C_{c}(F, \RN) \cap \cF$ satisfying
  $\sup_{x \in K} |u(x) - v(x)| < \varepsilon$.
  Define $w \coloneqq v \cdot \psi$, 
  which belongs to $C_{c}(F, \RN) \cap \cF$ by \cite[Lemma~6.5]{Kigami_12_Resistance}.
  For $x \in \supp(u)$,
  we have that $|u(x) - w(x)| = |u(x) - v(x)| < \varepsilon$.
  For $x \in K \setminus \supp(u)$,
  since we have that $u(x) = 0$,
  it follows that  
  \begin{equation}
    |u(x) - w(x)|
    =
    |u(x) \cdot \psi(x) - v(x) \cdot \psi(x)|
    \leq 
    |u(x) - v(x)|
    < \varepsilon.
  \end{equation}
  For $x \notin K$, we have that $|u(x) - w(x)| = 0$.
  Therefore, we deduce that $\| u-w \|_{\infty} < \varepsilon$,
  which completes the proof.
\end{proof}

%%%%%%%%%%%%%%%%%%%%%%%%%%%%%%%%%%%%%%%%%%%%%%%%%%%%%%%%%%%%%%%%%%%%%%%%%%%%%%%%%%%%%%%%%%%%%%%%%%%%%%%%%%%%%%%%%%%%%%%%%%%%%%%%%%
% Traces of resistance forms
%%%%%%%%%%%%%%%%%%%%%%%%%%%%%%%%%%%%%%%%%%%%%%%%%%%%%%%%%%%%%%%%%%%%%%%%%%%%%%%%%%%%%%%%%%%%%%%%%%%%%%%%%%%%%%%%%%%%%%%%%%%%%%%%%%

\subsection{Traces of resistance forms} \label{sec: traces of resistance forms}

In this subsection, 
we study extended Dirichlet forms associated with traces of resistance forms introduced in \cite[Chapter~8]{Kigami_12_Resistance}.
In particular, we establish that processes associated with traces of resistance forms 
coincide with trace processes, in Theorem~\ref{3. thm: trace and resistance metric space} below.
Throughout this subsection,
we fix a resistance form $(\cE, \cF)$ on a non-empty set $F$,
and write $R$ for the corresponding resistance metric.
Moreover, we fix a non-empty subset $B \subseteq F$.

\begin{dfn} [{The space $\cF^{(1)}|_{B}$}]
  We define a subspace of $C(B, \RN)$ by 
  \begin{equation}
    \cF^{(1)}|_{B}
    \coloneqq
    \{
      u|_{B} \mid u \in \cF^{(1)}
    \}.
  \end{equation}
\end{dfn}

The following two assertions are proved 
in the same way as \cite[Lemmas~8.2 and~8.5]{Kigami_12_Resistance}
and hence we omit the proofs.

\begin{prop} \label{3. prop: energy minimizer in F^1}
  For each $\varphi \in \cF^{(1)}|_{B}$,
  there exists a unique function $h_{B}^{(1)}(\varphi) \in \cF^{(1)}$
  such that $h_{B}^{(1)}(\varphi)|_{B} = \varphi$ and 
  \begin{equation}
    \mathcal{E}(h_{B}^{(1)}(\varphi), h_{B}^{(1)}(\varphi))
    =
    \inf\{
      \mathcal{E}(u,u) \mid u \in \cF^{(1)},\ u|_{B}=\varphi
    \}.
  \end{equation}
\end{prop}

\begin{prop}  \label{3. prop: characterization of minimizer by othogonality}
  Fix $\varphi \in \cF^{(1)}|_{B}$ and $u \in \cF^{(1)}$.
  Then $u=h_{B}^{(1)}(\varphi)$ 
  if and only if $u|_{B} = \varphi$ 
  and $\mathcal{E}(u,v)=0$ for all $v \in \cF^{(1)}$ such that $v|_{B}=0$.
  As a consequence,
  the map $h_{B}^{(1)} \colon \cF^{(1)}|_{B} \to \cF^{(1)}$ is linear.
\end{prop}

\begin{rem} \label{3. rem: difference of harmonic extensions}
  In \cite[Definition 8.3]{Kigami_12_Resistance},
  the $B$-harmonic function $h_{B}(\varphi)$ with boundary value $\varphi$ is defined. 
  The difference between $h_{B}(\varphi)$ and $h_{B}^{(1)}(\varphi)$ is 
  that $h_{B}(\varphi)$ is the minimizer of $\mathcal{E}(u,u)$ over $u \in \mathcal{F}$
  while $h_{B}^{(1)}(\varphi)$ is the minimizer over $u \in \cF^{(1)}$.
  Therefore, 
  if the resistance form $(\mathcal{E}, \mathcal{F})$ is recurrent,
  then $h_{B}(\varphi) = h_{B}^{(1)}(\varphi)$ by Theorem~\ref{3. thm: recurrence characterization with F^1}.
\end{rem}

Henceforth, we assume that $(F, R)$ is separable and locally compact.
Fix a fully-supported Radon measure $\mu$ on $F$,
write $(\cE, \cD)$ and $\cD_{e}$
for the associated Dirichlet form and the extended Dirichlet space, respectively.
Moreover, we suppose that we have another Radon measure $\nu$ on $F$
such that the (topological) support of $\nu$ is $B$.
Note that this implies that $B$ is closed.
Set
\begin{gather}
  \check{\mathcal{D}}_{e}^{\nu}
  \coloneqq 
  \{
    \varphi \in \mathcal{L}(B, \nu)
    \mid \exists u \in \mathcal{D}_{e} \ \text{such that}\ \varphi=u|_{B}, \ \nu \text{-a.e.}
  \}, \\
  \check{\mathcal{D}}^{\nu}
  \coloneqq 
  \check{\mathcal{D}}_{e}^{\nu} \cap L^{2}(B, \nu),
\end{gather}
where we recall the space $\cL(B, \nu)$ from Definition~\ref{3. dfn: the spaces cL and L^2}.

\begin{lem} \label{3. lem: the extended Dirichlet space for trace}
  Every function $\varphi \in \check{\mathcal{D}}_{e}^{\nu}$ has a unique continuous modification
  and so we can regard $\check{\mathcal{D}}_{e}^{\nu}$ as a subspace of $C(B, \RN)$.
  Then it holds that $\check{\mathcal{D}}_{e}^{\nu} = \cF^{(1)}|_{B}$.
\end{lem}

\begin{proof}
  Fix $\varphi \in \check{\mathcal{D}}_{e}^{\nu}$
  and choose $u \in \mathcal{D}_{e}$ such that $u|_{B}= \varphi,\, \nu$-a.e.
  By Theorem~\ref{3. thm: characterization of the extended Dirichlet space},
  $u|_{B}$ is a continuous modification of $\varphi$.
  Since $B$ is the support of $\nu$,
  uniqueness follows.
  The last assertion is immediate from the definitions of $\cF^{(1)}|_{B}$ and $\check{\mathcal{D}}_{e}^{\nu}$.
\end{proof}

Write $X = ((X_{t})_{t \geq 0}, (P_{x})_{x \in F})$ 
for the Hunt process associated with $(\cE, \cD)$.
Let $\sigma_{B}$ denote the hitting time of $B$, i.e.,
\begin{equation}  \label{3. eq: dfn of hitting time}
  \sigma_{B} 
  \coloneqq 
  \inf\{
    t > 0 \mid X_{t} \in B
  \},
\end{equation}
For $\varphi \in \check{\mathcal{D}}_{e}^{\nu}$,
its harmonic extension via the process $X$ is defined by 
\begin{equation} \label{3. eq: the harmonic extension via process}
  \check{h}_{B}(\varphi)(x)
  \coloneqq
  E_{x} [u(X_{\sigma_{B}}) \cdot 1_{\{\sigma_{B} < \infty\}}],
\end{equation}
where we choose $u \in \mathcal{D}_{e}$ such that $u|_{B}=\varphi,\, \nu$-a.e.
Note that $\check{h}_{B}(\varphi)$ is independent of the choice of $u$
(see \cite[Lemma 6.2.1]{Fukushima_Oshima_Takeda_11_Dirichlet}).
The following result says that the harmonic extension $\check{h}_{B}(\varphi)$ via the associated process $X$ 
coincides with the harmonic extension $h_{B}^{(1)}(\varphi)$ as the energy minimizer over $\cF^{(1)}$
defined in Proposition~\ref{3. prop: energy minimizer in F^1}.

\begin{thm} \label{3. thm: coincidence of two harmonic extensions}
  For any $\varphi \in \cF^{(1)}|_{B}$,
  it holds that $h_{B}^{(1)}(\varphi)(x) = \check{h}_{B}(\varphi)(x)$ for all $x \in F$.
\end{thm}

\begin{proof}
  By \cite[Theorem 3.4.8]{Chen_Fukushima_12_Symmetric},
  $\check{h}_{B}(\varphi)$ is a continuous function belonging to $\mathcal{D}_{e}$
  such that $\mathcal{E}(\check{h}_{B}(\varphi), v) = 0$ for any $v \in \mathcal{D}_{e}$ with $v|_{B}=0$.
  The desired result follows immediately from Theorem~\ref{3. thm: characterization of the extended Dirichlet space},
  Proposition~\ref{3. prop: characterization of minimizer by othogonality}
  and Lemma~\ref{3. lem: the extended Dirichlet space for trace}.
\end{proof}

For $\varphi, \psi \in \check{\mathcal{D}}_{e}^{\nu}$, 
set 
\begin{equation}
  \check{\mathcal{E}}^{\nu}(\varphi, \psi)
  \coloneqq
  \mathcal{E}(\check{h}_{B}(\varphi), \check{h}_{B}(\psi)).
\end{equation}
Note that $(\check{\cE}^{\nu}, \check{\cD}^{\nu})$
is the trace of $(\cE, \cF)$ on $B$ with respect to $\nu$
(see \cite[Section 6.2]{Fukushima_Oshima_Takeda_11_Dirichlet}).
Define a PCAF $A=(A_{t})_{t \geq 0}$ and its right-continuous inverse $\tau=(\tau(t))_{t \geq 0}$
by setting 
\begin{equation}
  A_{t}
  \coloneqq 
  \int_{F} L(x,t)\, \nu(dx),
  \qquad 
  \tau(t)
  \coloneqq
  \inf\{
    s>0 :
    A_{s} >t
  \},
\end{equation}
where $(L(x,t))_{x \in F, t \geq 0}$ is the jointly-measurable local time of $X$ satisfying the occupation density formula~\eqref{3. eq: the occupation density formula}.
The \textit{trace} $\check{X}^{\nu}=(\check{X}^{\nu}_{t})_{t \geq 0}$ of $X$ on $B$ (with respect to $\nu$)
is defined by setting $\check{X}^{\nu}_{t} \coloneqq X_{\tau(t)}$.
Note that $(\check{X}^{\nu}, (P_{x})_{x \in B})$ is a strong Markov process
(see \cite[Theorem~A.2.12]{Fukushima_Oshima_Takeda_11_Dirichlet}).

\begin{lem} [{\cite[Theorems~5.2.2 and~5.2.15]{Chen_Fukushima_12_Symmetric}}]
  \label{3. lem: the trace of Dirichlet form}
  The pair $(\check{\mathcal{E}}^{\nu}, \check{\mathcal{D}}^{\nu})$
  is a regular Dirichlet form on $L^{2}(B, \nu)$ 
  and $\check{X}^{\nu}$ is the associated process.
  The extended Dirichlet space is $\check{\mathcal{D}}_{e}^{\nu}$.
\end{lem}

The following is the main result of this section.
The assertion is almost the same as \cite[Theorem 2.5 and Lemma 2.6]{Croydon_Hambly_Kumagai_17_Time-changes},
but it is new
that we do not assume that $B$ is compact.

\begin{thm} \label{3. thm: trace and resistance metric space}
  Assume that the resistance form $(\cE, \cF)$ is recurrent.
  Set 
  \begin{gather}
    \mathcal{F}|_{B}
    \coloneqq
    \{ u|_{B} \mid u \in \mathcal{F} \}, \\
    \mathcal{E}|_{B}(u|_{B}, v|_{B})
    \coloneqq
    \mathcal{E}(h_{B}^{(1)}(u), h_{B}^{(1)}(v)), \qquad u,v \in \mathcal{F}.
  \end{gather}
  Then $(\cE|_{B}, \cF|_{B})$ is a recurrent resistance form 
  and the corresponding resistance metric is $R|_{B \times B}$.
  The process associated with $(B, R|_{B \times B}, \nu)$
  is $\check{X}^{\nu}$.
\end{thm}

\begin{proof}
  By Remark \ref{3. rem: difference of harmonic extensions}
  and \cite[Theorem 8.4]{Kigami_12_Resistance},
  we have that $(\mathcal{F}|_{B}, \mathcal{E}|_{B})$ is a regular resistance form 
  and the corresponding resistance metric is $R|_{B \times B}$.
  Moreover, using the recurrence of $(\cE, \cF)$,
  it is straightforward to verify the recurrence of $(\cE|_B, \cF|_B)$.
  This proves the first assertion.
  
  For the second assertion,
  note that Theorem~\ref{3. thm: recurrence characterization with F^1} yields that $\cF^{(1)} = \mathcal{F}$.
  Combining this with
  Lemmas~\ref{3. lem: the extended Dirichlet space for trace} and~\ref{3. lem: the trace of Dirichlet form},
  we deduce that 
  the extended Dirichlet space of $\check{X}^{\nu}$ is $\mathcal{F}|_{B}$.
  Moreover,
  by Theorem~\ref{3. thm: coincidence of two harmonic extensions},
  we have that $\mathcal{E}|_{B}=\check{\mathcal{E}}^{\nu}$ on $\mathcal{F}|_{B}$.
  On the other hand,
  from Theorem~\ref{3. thm: characterization of the extended Dirichlet space},
  we have that $(\cF|_{B})^{(1)}$ is the extended Dirichlet space 
  for the process associated with the tuple $(B, R|_{B \times B}, \nu)$.
  Therefore, the last assertion follows 
  by proving that $(\cF|_{B})^{(1)} = \mathcal{F}|_{B}$,
  as this implies the coincidence of the Dirichlet forms of the processes.

  By definition, we have $(\cF|_{B})^{(1)} \subseteq \mathcal{F}|_{B}$.
  Fix $\varphi \in \mathcal{F}|_{B}$
  and choose $u \in \mathcal{F}$ such that $u|_{B} = \varphi$.
  By Theorem~\ref{3. thm: recurrence characterization with F^1},
  there exists a sequence $(u_{n})_{n \geq 1}$ in $C_{c}(F, \RN) \cap \mathcal{F}$
  such that $u_{n} \to u$ with respect to $\mathcal{E}^{(1)}$.
  Set $\varphi_{n} \coloneqq u_{n}|_{B}$.
  It is then the case that $\varphi_{n} \in C_{c}(B) \cap \mathcal{F}|_{B}$.
  Obviously, $\varphi_{n} \to \varphi$ pointwise (on $B$).
  Moreover, we deduce that 
  \begin{align}
    \mathcal{E}|_{B}(\varphi_{n} - \varphi, \varphi_{n} - \varphi)
    &=
    \mathcal{E}(h_{B}^{(1)}(\varphi_{n} - \varphi), h_{B}^{(1)}(\varphi_{n} - \varphi))\\
    &=
    \inf\{
      \mathcal{E}(w,w) \mid w \in \mathcal{F}\ \text{such that}\ w|_{B}=\varphi_{n} - \varphi
    \}\\
    &\leq 
    \mathcal{E}(u_{n}- u, u_{n}-u),
  \end{align}
  which implies that $\varphi_{n} \to \varphi$ with respect to $\mathcal{E}|_{B}$.
  Therefore, it follows that $\varphi \in (\cF|_{B})^{(1)}$
  and hence $(\cF|_{B})^{(1)} = \mathcal{F}|_{B}$,
  which completes the proof.
\end{proof}

The resistance form $(\cE|_B, \cF|_B)$ defined above is called the \emph{trace} of $(\cE, \cF)$ onto $B$.

%%%%%%%%%%%%%%%%%%%%%%%%%%%%%%%%%%%%%%%%%%%%%%%%%%%%%%%%%%%%%%%%%%%%%%%%%%%%%%%%%%%%%%%%%%%%%%%%%%%%%%%%%%%%%%%%%%%%%%%%%%%%%%%%%%
%%%%%%%%%%%%%%%%%%%%%%%%%%%%%%%%%%%%%%%%%%%%%%%%%%%%%%%%%%%%%%%%%%%%%%%%%%%%%%%%%%%%%%%%%%%%%%%%%%%%%%%%%%%%%%%%%%%%%%%%%%%%%%%%%%
% Electrical networks
%%%%%%%%%%%%%%%%%%%%%%%%%%%%%%%%%%%%%%%%%%%%%%%%%%%%%%%%%%%%%%%%%%%%%%%%%%%%%%%%%%%%%%%%%%%%%%%%%%%%%%%%%%%%%%%%%%%%%%%%%%%%%%%%%%
%%%%%%%%%%%%%%%%%%%%%%%%%%%%%%%%%%%%%%%%%%%%%%%%%%%%%%%%%%%%%%%%%%%%%%%%%%%%%%%%%%%%%%%%%%%%%%%%%%%%%%%%%%%%%%%%%%%%%%%%%%%%%%%%%%

\section{Electrical networks} \label{sec: electrical networks}

In this section,
we study electrical networks from the point of view of resistance forms, 
using results obtained in the previous section.
In particular,
Section \ref{sec: resistance forms on electrical networks} presents some basic results 
about resistance forms associated with electrical networks,
and in Section~\ref{sec: trace of electrical networks}
we study traces of electrical networks. 
Section \ref{sec: measurability} provides technical conclusions on measurability 
that are needed for Theorem~\ref{1. thm: random, main result}.

%%%%%%%%%%%%%%%%%%%%%%%%%%%%%%%%%%%%%%%%%%%%%%%%%%%%%%%%%%%%%%%%%%%%%%%%%%%%%%%%%%%%%%%%%%%%%%%%%%%%%%%%%%%%%%%%%%%%%%%%%%%%%%%%%%
% Resistance forms of electrical networks
%%%%%%%%%%%%%%%%%%%%%%%%%%%%%%%%%%%%%%%%%%%%%%%%%%%%%%%%%%%%%%%%%%%%%%%%%%%%%%%%%%%%%%%%%%%%%%%%%%%%%%%%%%%%%%%%%%%%%%%%%%%%%%%%%%

\subsection{Resistance forms of electrical networks}  \label{sec: resistance forms on electrical networks}

In this subsection,
we study some basic properties of resistance forms associated with electrical networks.
Recall the definition of electrical networks $G$ and related bilinear forms $(\cE_G, \cF_G)$
from Definitions~\ref{dfn: electrical network} and~\ref{1. dfn: resistance forms associated with electrical networks}.

\begin{thm}
  Fix an electrical network $G$.
  Then the pair $(\cE_{G}, \cF_{G})$ is a regular resistance form.
  We denote the associated resistance metric by $R_{G}$.
  Then the topology on $V_{G}$ induced from $R_{G}$ is the discrete topology.
\end{thm}

\begin{proof}
  The conditions \ref{3. dfn item: the domain F}, \ref{3. dfn item: the quotient space of resistance form is Hilbert},
  and \ref{3. dfn item: Markov property of resistance forms} 
  can be checked similarly to the proofs of \cite[Proposition 1.21 and Lemma 1.27]{Barlow_17_Random}.
  Since we have that 
  \begin{equation}  \label{4. eq: energy of Dirac function}
    \cE_{G}(1_{\{x\}}, 1_{\{x\}}) 
    =
    \frac{1}{2}
    \sum_{z, w \in V_{G}}
    c_{G}(z, w) (1_{\{x\}}(z) - 1_{\{x\}}(w))^{2}
    =
    \sum_{z \in V_{G}} c_{G}(x, z) 
    =
    c_{G}(x)
    < \infty,
  \end{equation}
  we deduce that $1_{\{x\}} \in \cF_{G}$,
  which implies \ref{3. dfn item: separability of points}.
  Using that the electrical network is a connected graph,
  one can verify that $R(x,y) \leq d(x,y)$,
  where $d(x,y)$ is the shortest path distance 
  on a weighted graph $(V_{G}, E_{G}, (c_{G}(z,w)^{-1})_{\{z,w\} \in E_{G}})$,
  which implies \ref{3. dfn item: existence of resistance metric}.
  By \eqref{4. eq: energy of Dirac function} and \ref{3. dfn item: existence of resistance metric},
  we have that, for any $x \neq y$,
  \begin{equation}
    R_{G}(x, y) 
    \geq 
    \cE(1_{\{x\}}, 1_{\{x\}})^{-1}
    = 
    c_{G}(x)^{-1} 
    > 0.
  \end{equation}
  Therefore, 
  the topology on $F$ induced from $R_{G}$
  is the discrete topology.
  In particular,
  a subset $K$ of $V_{G}$ is compact if and only if $|K| < \infty$.
  Since $1_{\{x\}} \in \cF_{G}$ for each $x \in V_{G}$,
  it follows that $1_{K} \in \cF_{G}$ for any compact subset $K$ of $V_{G}$.
  Therefore,
  using \cite[Theorem 6.3]{Kigami_12_Resistance},
  we deduce that $(\cE_{G}, \cF_{G})$ is regular.
\end{proof}

Below, 
we verify that processes associated with resistance forms on electrical networks
coincide with naturally associated reversible Markov chains on them.

\begin{thm} \label{4. thm: the Hunt process associated with electrical network}
  Fix an electrical network $G$.
  Let $\nu$ be a Radon measure on $V_{G}$ of full support.
  Write $((X_{t})_{t \geq 0}, (P_{x})_{x \in V_{G}})$ for the Hunt process 
  associated with $(V_{G}, R_{G}, \nu)$.
  Then $(X_{t})_{t \geq 0}$ is the minimal continuous-time Markov chain on $V_{G}$
  with $Q$-matrix $(q_{xy})_{x, y \in V_{G}}$ given by $q_{xy} = c_{G}(x,y)/\nu(\{x\})$ 
  for $x \neq y$ and $q_{xx} = -c_{G}(x)/\nu(\{x\})$.
  (Recall that the minimality means that, after the explosion time, 
  the process stays at the cemetery point forever.)
\end{thm}

\begin{proof}
  Note that it is assumed in the definition of Hunt processes that $(X_{t})_{t \geq 0}$ is minimal.
  Let $Q=(q_{xy})_{x,y \in V_{G}}$ be the $Q$-matrix of the Hunt process.
  By direct calculations,
  we deduce that 
  \begin{equation} \label{eq: energy of indicators}
    \cE_{G}(1_{\{x\}}, 1_{\{y\}}) 
    =
    \begin{cases}
      -c_{G}(x,y), &x \neq y, \\
      c_{G}(x), &x = y.
    \end{cases}
  \end{equation}
  Let $(\cE_{G}, \cD_{G})$ be the Dirichlet form determined by 
  the resistance form $(\cE_{G}, \cF_{G})$ and the measure $\nu$.
  We write $(T_t)_{t > 0}$ for the associated semigroup,
  i.e.,
  $T_{t}f(x) \coloneqq E_{x} f(X_{t})$.
  Using the approximating forms of $(\cE_{G}, \cD_{G})$ 
  (see \cite[Lemma 1.3.4]{Fukushima_Oshima_Takeda_11_Dirichlet}),
  we obtain that 
  \begin{align}
    \cE_{G}(1_{\{x\}}, 1_{\{y\}}) 
    &= 
    \lim_{t \downarrow 0} 
    \frac{1}{t}
    \sum_{z \in V_{G}}
    (1_{\{x\}}(z) - T_{t}1_{\{x\}}(z)) \cdot 1_{\{y\}}(z) \nu(\{z\})\\
    &=
    \lim_{t \downarrow 0}
    \frac{1_{\{x\}}(y)-T_{t}1_{\{x\}}(y)}{t} \nu(\{y\})\\
    &=
    \lim_{t \downarrow 0}
    \frac{T_{0}1_{\{x\}}(y)-T_{t}1_{\{x\}}(y)}{t} \nu(\{y\})\\
    &=
    -q_{yx} \nu(\{y\}),
  \end{align}
  where we use the backward equation for $X$ at the last equality 
  (cf.\ \cite[Theorem~13.9]{Kallenberg_21_Foundations}).
  Now, the desired result is straightforward.
\end{proof}

Let $G$ be an electrical network.
Define $X_{G} = (X_{G}(t))_{t \geq 0}$ 
to be the Hunt process associated with $(V_{G}, R_{G}, \mu_{G})$,
where we recall that $\mu_G$ denotes the associated conductance measure.
By Theorem~\ref{4. thm: the Hunt process associated with electrical network}, $X_{G}$ is the constant speed random walk on $G$.
In this setting,
we have an explicit formula of the local time 
$L_{G} = (L_{G}(x,t))_{t \geq 0, x \in V_{G}}$ of $X_{G}$ as follows:
\begin{equation}  \label{4. eq: local time formula for electrical networks}
  L_{G}(x, t) 
  =
  \frac{1}{c_{G}(x)}
  \int_{0}^{t} 1_{ \{x\} } (X_G(s))\,  ds.
\end{equation}
Recall from Section~\ref{sec: introduction} that $Y_{G}=(Y_{G}(k))_{k \geq 0}$
denotes the discrete-time Markov chain on $G$.
For convenience,
in the following discussions,
we always suppose that $Y_{G}$ is defined on the same probability measure space as $X_{G}$
via the following relation:
\begin{equation}  \label{4. eq: coupling of X_G and Y_G}
  Y_{G}(k)
  =
  X_{G}(J_{G}^{(k)}),
\end{equation}
where $(J_{G}^{(k)})_{k \geq 0}$
is the sequence of jump times of $X_{G}$ with $J_{G}^{(0)} \coloneqq 0$.

%%%%%%%%%%%%%%%%%%%%%%%%%%%%%%%%%%%%%%%%%%%%%%%%%%%%%%%%%%%%%%%%%%%%%%%%%%%%%%%%%%%%%%%%%%%%%%%%%%%%%%%%%%%%%%%%%%%%%%%%%%%%%%%%%%
% Trace of electrical networks
%%%%%%%%%%%%%%%%%%%%%%%%%%%%%%%%%%%%%%%%%%%%%%%%%%%%%%%%%%%%%%%%%%%%%%%%%%%%%%%%%%%%%%%%%%%%%%%%%%%%%%%%%%%%%%%%%%%%%%%%%%%%%%%%%%

\subsection{Traces of electrical networks}  \label{sec: trace of electrical networks}
  
In this subsection,
we introduce the notion of traces of electrical networks,
which will be used in the proof of the main results of this paper 
to approximate Markov chains on infinite electrical networks 
by Markov chains on finite electrical networks.

Throughout this subsection,
we fix an electrical network $G$
such that the associated resistance form $(\cE_{G}, \cF_{G})$ is recurrent.
We define the associated Laplacian $\cL_{G}$ by setting 
\begin{equation}
  \cL_{G}f(x)
  \coloneqq 
  \sum_{y \in V_{G}} P_{G}(x,y)f(y) - f(x)
  =
  E_{x}^{G}f(Y_{G}(1)) - f(x)
\end{equation}
for each function $f \colon V_G \to \RN$ 
satisfying $E_{x}^{G}|f(Y_{G}(1))| < \infty$ for all $x \in V_{G}$.
In particular,
for any bounded function $f$,
$\cL_{G}f$ is defined.

\begin{prop} \label{4. prop: Gauss-Green formula in resistance forms}
  Fix $u, v \in \cF_{G}$
  such that $E_{x}^{G}|u(Y_{G}(1))| < \infty$ for all $x \in V_{G}$.
  If $\int_{V_{G}} |\mathcal{L}_{G}u| |v|\, d\mu_{G} < \infty$,
  then it holds that 
  \begin{equation}
    \cE_{G}(u, v)
    =
    - \int_{V_{G}} \cL_{G}u \cdot v\, d\mu_{G}. 
  \end{equation}
\end{prop}

\begin{proof}
  Note that if $v$ is compactly supported,
  then the assertion follows from \cite[Theorem 1.24]{Barlow_17_Random}.
  Assume that $\| v \|_{\infty} < \infty$.
  Since the resistance form $(\cE_G, \cF_G)$ is assumed to be recurrent,
  we have from Theorem~\ref{3. thm: recurrence characterization with F^1} that 
  $\cF_G^{(1)} = \cF_G$.
  Thus, we can use Corollary~\ref{3. cor: find a uniformly bounded sequence for a function in F^1}
  to find a sequence $v_{n}$ in $C_{c}(V_{G}, \RN) \cap \cF_{G}$ 
  such that $v_{n} \to v$ with respect to $\cE_{G}^{(1)}$
  and $|v_n(x) |\leq |v(x)|$ for all $x \in V_G$ and $n \geq 1$.
  Since each $v_n$ is compactly supported,
  we have that 
  \begin{equation}
    \cE_{G}(u, v_{n})
    =
    - \int_{V_{G}} \cL_{G}u \cdot v_{n} \, d\mu_{G}.
  \end{equation}
  Letting $n \to \infty$ in the above equality and using the dominated convergence theorem,
  we deduce the desired result.
  In the case that $\| v\|_{\infty}=\infty$,
  by Corollary~\ref{3. cor: find a bounded sequence for a function in F^1},
  the same result is verified.
\end{proof}

\begin{rem}
  Proposition~\ref{4. prop: Gauss-Green formula in resistance forms} is an analogue of \cite[Theorem~1.24]{Barlow_17_Random}.
  The difference is that
  the conditions on $u, v$ of Proposition~\ref{4. prop: Gauss-Green formula in resistance forms} 
  are weaker than those of \cite[Theorem 1.24]{Barlow_17_Random}.
  This is very important in the forthcoming results.
  However,
  one should note that the resistance form is assumed to be recurrent in our setting,
  while it is not in that book.
\end{rem}

Fix a non-empty subset $B \subseteq V_{G}$.
We now introduce the notion of trace of electrical networks.
Recall the definition of the resistance form $(\cE_{G}|_{B}, \cF_{G}|_{B})$
from Theorem~\ref{3. thm: trace and resistance metric space}.
Set $V(G|_{B}) \coloneqq B$ and, for $x,y \in V(G|_{B})$,
\begin{equation}
  c_{G|_{B}}(x,y) 
  \coloneqq 
  \begin{cases}
    -\cE_{G}|_{B}(1_{\{x\}}, 1_{\{y\}}), &x \neq y,\\
    0, &x = y.
  \end{cases}
\end{equation}
Following the argument in the proof of \cite[Theorem~A.33]{Barlow_17_Random},
one can readily verify that $c_{G|_B}(x,y)$ is non-negative.
We define the edge set $E(G|_{B})$ on $V(G|_{B})$
by declaring $\{x,y\}$ is an edge if and only if 
$c_{G|_{B}}(x,y) > 0$.
The following result is easily obtained 
from the fact that $(\cE_{G}|_{B}, \cF_{G}|_{B})$
is a resistance form,
and so we omit the proof.

\begin{thm} \label{thm: explicit expression of trace form}
  The tuple $(V(G|_{B}), E(G|_{B}), c_{G|_{B}})$ is an electrical network in the sense of Definition~\ref{dfn: electrical network}.
  Moreover,
  the associated resistance form $(\cE_{G|_B}, \cF_{G|_B})$ coincides with the trace $(\cE_{G}|_{B}, \cF_{G}|_{B})$
  defined in Theorem~\ref{3. thm: trace and resistance metric space}.
  In particular, the associated resistance metric is given by $R_{G}|_{B \times B}$.
\end{thm}

\begin{dfn} [{Trace of electrical networks}]
  We write $G|_{B} \coloneqq (V(G|_{B}), E(G|_{B}), c_{G|_{B}})$
  and call it the \textit{trace} of $G$ onto $B$.
\end{dfn}

To prove Theorem~\ref{thm: explicit expression of trace form},
we first establish an explicit formula for the conductances on the trace.
We define $T_B$ and $T_B^{+}$ 
to be the first hitting time and return time 
of $Y_{G}$ to $B$, i.e.,
\begin{equation} \label{4. eq: first hitting and return time of Y}
  T_{B}
  \coloneqq
  \{
    n \geq 0 \mid
    Y_{G}(n) \in B
  \},
  \quad
  T_{B}^{+}
  \coloneqq
  \{
    n > 0 \mid
    Y_{G}(n) \in B
   \}.
\end{equation}
Note that both of them are finite almost surely,
since we assume that the associated resistance form $(\cE_{G}, \cF_{G})$ is recurrent
and accordingly $Y_G$ is recurrent.

\begin{thm} \label{4. thm: expression of trace conductance}
  For any $x,y \in B$,
  \begin{equation}
    -\cE_G|_B(1_{\{x\}}, 1_{\{y\}}) 
    = 
    \begin{cases}
      c_{G}(x) P_{x}^{G} (Y_{G}(T_{B}^{+})=y), & x \neq y,\\
      -c_G(x) P_x^G(Y_G(T_B^+) \in B \setminus \{x\}), & x = y.
    \end{cases}
  \end{equation}
  In particular, if $x \neq y$, then 
  \begin{equation}
    c_{G|_B}(x,y)
    =
    c_{G}(x) P_{x}^{G} (Y_{G}(T_{B}^{+})=y).
  \end{equation}
\end{thm}

\begin{proof}
  Fix $x, y \in B$.
  Recall that $X_G$ denotes the constant speed random walk on $G$
  and $\sigma_B$ denotes the first hitting time of $B$ by $X_G$.
  Since we have that $X_{G}(\sigma_{B}) = Y_{G}(T_{B})$,
  by using Theorem~\ref{3. thm: coincidence of two harmonic extensions},
  we deduce that, for each $z \in B$,
  \begin{equation}
    h_{B}^{(1)}(1_{\{x\}})(z)
    =
    P_{z}^{G}( Y_{G}(T_{B}) = x )
    =
    \begin{cases}
      0, & z \in B\setminus \{x\}, \\
      1, & z =x.
    \end{cases}
  \end{equation}
  The Markov property yields that, for any $z \in V_{G}$,
  \begin{align}
    \sum_{w \in V_{G}} P_{G}(z,w) h_{B}^{(1)}(1_{\{y\}})(w)
    &=
    \sum_{w \in V_{G}} P_{G}(z,w) P_{w}^{G}(Y_{G}(T_{B}) = y) \\
    &=
    P_{z}^{G}(Y_{G}(T_{B}^{+}) = y),
  \end{align}
  which implies that
  \begin{equation}
    \cL_{G}(h_{B}^{(1)}(1_{\{y\}})) (z) 
    = 
    \begin{cases}
      0, & z \notin B, \\
      P_{z}^{G}(Y_{G}(T_{B}^{+}) = y), & z \in B \setminus \{y\},\\
      -P_y^G(Y_G(T_B^+) \in B \setminus \{y\}), & z = y.
    \end{cases}
  \end{equation}
  From Proposition~\ref{4. prop: Gauss-Green formula in resistance forms},
  it follows that 
  \begin{align}
    - \cE_{G}(h_{B}^{(1)}(1_{\{y\}}), h_{B}^{(1)}(1_{\{x\}}))
    =
    \mathcal{L}_{G}(h_{B}^{(1)}(1_{\{y\}}))(x)\, c_{G}(x).
  \end{align}
  Now, the result is immediate.
\end{proof}

\begin{rem}
  Theorem~\ref{4. thm: expression of trace conductance} is an extension of \cite[Proposition 2.48]{Barlow_17_Random}
  which assumes that the complement of $B$ is a finite set.
  In that book,
  this assumption is needed to prove \cite[Proposition 2.47]{Barlow_17_Random},
  which is the same as Theorem~\ref{3. thm: coincidence of two harmonic extensions}.
  The proof of \cite[Proposition 2.47]{Barlow_17_Random}
  in that book roughly goes as follows:
  if the complement of $B$ is finite,
  then $h_{B}^{(1)}(\varphi)$ is bounded;
  since a bounded solution of a Dirichlet problem is unique (in the current setting),
  we obtain $h_{B}^{(1)}(\varphi) = \check{h}_{B}(\varphi)$.
  Therefore, 
  in this approach,
  the assumption that the complement of $B$ is finite 
  is crucial,
  and
  our result that
  removes the assumption is non-trivial.
\end{rem}

\begin{cor} \label{4. cor: difference for trace measure at one point}
  For any $x \in B$,
  it holds that 
  \begin{equation}
    0
    \leq 
    c_{G}(x) - c_{G|_B}(x) 
    =
    c_{G}(x) P_{x}^{G} (Y_{G}(T_{B}^{+})=x)
    \leq
    \sum_{y \notin B} c_{G}(x,y).
  \end{equation}
\end{cor}

\begin{proof}
  By Theorem~\ref{4. thm: expression of trace conductance},
  we obtain that 
  \begin{align}
    0
    \leq 
    c_{G}(x) - c_{G|_B}(x) 
    &\leq
    c_{G}(x)
    \left(
      1 
      - 
      \sum_{y \in B \setminus \{x\}} 
      P_{x}^{G} (Y_{G}(T_{B}^{+})=y)
    \right) \\
    &=
    c_{G}(x) P_{x}^{G} (Y_{G}(T_{B}^{+})=x) \\
    &\leq 
    c_{G}(x) \sum_{y \notin B} 
    P_{G}(x,y) P_{y}^{G} (Y_{G}(T_{B})=x) \\
    &\leq 
    \sum_{y \notin B} c_{G}(x,y),
  \end{align}
  where we use the Markov property at the third inequality.
\end{proof}

\begin{proof} [Proof of Theorem~\ref{thm: explicit expression of trace form}]
  By Theorem~\ref{4. thm: expression of trace conductance} and Corollary~\ref{4. cor: difference for trace measure at one point},
  we deduce that the graph $(V(G|_B), E(G|_B))$ is connected
  and $c_{G|_B}(x) < \infty$ for all $x \in B$.
  Thus, $(V(G|_B), E(G|_B), c_{G|_B})$ is an electrical network.
  Write $(\cE_{G|_B}, \cF_{G|_B})$ for the associated resistance form,
  that is, for any functions $f, g \colon B \to \RN$,
  \begin{equation}
    \cE_{G|_B}(f, g) 
    \coloneqq 
    \frac{1}{2} 
    \sum_{x, y \in B} 
    c_{G|_B}(x,y) (f(x) - f(y)) (g(x) - g(y))
  \end{equation}
  (if the right-hand side exists),
  and 
  $\cF_{G|_B} \coloneqq \{f \mid \cE_{G|_B}(f,f) < \infty\}$.

  By Theorem~\ref{4. thm: expression of trace conductance},
  we have that 
  \begin{equation}
    \cE_G|_B(1_{\{x\}}, 1_{\{x\}})
    = 
    \sum_{y \in B} c_{G|_B}(x,y),
    \qquad 
    \forall x \in B.
  \end{equation}
  We deduce from the above identity that,
  for any function $f \colon B \to \RN$ supported on a finite subset $B_0 \subseteq B$,
  \begin{align}
    \cE_G|_B(f, f)
    &= 
    \sum_{x,y \in B_0} f(x) f(y)\, \cE_G|_B(1_{\{x\}}, 1_{\{y\}})\\
    &= 
    \sum_{x,y \in B} f(x) f(y)\, \cE_G|_B(1_{\{x\}}, 1_{\{y\}})\\
    &= 
    \sum_{x \in B} f(x)^2 \sum_{y \in B}  c_{G|_B}(x,y) 
    -
    \sum_{x,y \in B} f(x) f(y)\, c_{G|_B}(x, y)\\
    &= 
    \frac{1}{2} 
    \sum_{x, y \in B} 
    c_{G|_B}(x,y) (f(x) - f(y))^2\\
    &=
    \cE_{G|_B}(f, f).
    \label{pr eq: expansion for compactly supported function}
  \end{align}
  Thus,
  \begin{equation}
    \cE_G|_B(f, f) = \cE_{G|_B}(f, f)
  \end{equation}
  for all functions $f \in C_c(B, \RN)$.
  (NB.\ The space $B$ is equipped with the discrete topology, 
  and so every $f \in C_c(B, \RN)$ is supported on a finite subset.)
  Combining this with the recurrence of $(\cE_G|_B, \cF_G|_B)$,
  we obtain that $(\cE_{G|_B}, \cF_{G|_B})$ is also recurrent.
  Thus, by Theorem~\ref{3. thm: recurrence characterization with F^1},
  the domains $\cF_G|_B$ and $\cF_{G|_B}$ are the closures of $C_c(B, \RN)$ 
  with respect to $\cE_G|_B^{(1)}$ and $\cE_{G|_B}^{(1)}$, respectively.
  However, the coincidence of the forms on $C_c(B, \RN)$ implies that 
  $(\cE_G|_B, \cF_G|_B) = (\cE_{G|_B}, \cF_G|_B)$.
  This completes the proof. 
\end{proof}

Now, we provide a coupling 
of the discrete-time Markov chains $Y_{G|_B}$ on $G|_B$
and $Y_{G}$ on $G$.
Set $\tilde{T}_{B}^{(0)} \coloneqq 0$ and inductively for each $k \geq 0$
\begin{equation}
  \tilde{T}_{B}^{(k+1)}
  \coloneqq
  \inf
  \left\{
    n \geq \tilde{T}_{B}^{(k)} + 1\, \middle |\, Y_{G}(n) \in B \setminus \{ Y_{G}(\tilde{T}_{B}^{(k)})\}
  \right\}.
\end{equation}
If the infimum is taken over the empty set,
then we set $\tilde{T}_{B}^{(k+1)} \coloneqq \tilde{T}_{B}^{(k)}$.

\begin{dfn} [{Trace of discrete-time Markov chains}]  \label{4. dfn: trace}
  Define $\tr_{B}Y_{G} = (\tr_{B}Y_{G}(k))_{k \geq 0}$ by setting 
  $\tr_{B}Y_{G}(k) \coloneqq Y_{G} ( \tilde{T}_{B}^{(k)} )$.
  As before,
  we set $\tr_{B}Y_{G}(t) \coloneqq \tr_{B}Y_{G}( \lfloor t \rfloor)$ for $t \in \RNp$
  and regard $\tr_{B}Y_{G}$ as a random element of $D(\RNp, B)$.
  We call $\tr_{B}Y_{G}$ the \textit{trace} of $Y_{G}$ onto $B$.
\end{dfn}

\begin{thm} \label{4. thm: coupling by trace of Y_G}
  It holds that, for any $\rho \in B$,
  \begin{equation}
    P^{G}_{\rho}
    \left(
      \tr_{B} Y_{G} \in \cdot
    \right)
    =
    P^{G|_B}_{\rho}
    \left(
      Y_{G|_B} \in \cdot
    \right)
  \end{equation}
  as probability measures on $D(\RNp, B)$.
\end{thm}

\begin{proof}
  If $|B| =1$,
  then the assertion is obvious.
  We assume that there are at least two vertices in $B$.
  The strong Markov property of $Y_{G}$ yields that 
  $\left( \tr_{B} Y_{G}(k) \right)_{k \geq 0}$ is a Markov chain on $B$.
  Thus,
  it suffices to show that 
  \begin{equation}
    P_{\rho}^{G} 
    \left(
      Y_{G}(\tilde{T}_{B}^{(1)}) = x
    \right)
    =
    \frac{c_{G|_B}(\rho, x)}{c_{G|_B}(\rho)},
    \quad \forall x \in B \setminus \{\rho\}.
  \end{equation}
  Define $T_{B}^{(0)} \coloneqq 0$ and inductively for each $k \geq 0$
  \begin{equation}
    T_{B}^{(k+1)}
    \coloneqq
    \inf
    \left\{
      n \geq T_{B}^{(k)} + 1 \mid Y_{G}(n) \in B
    \right\}.
  \end{equation}
  Fix $x \in B \setminus \{\rho\}$.
  By the strong Markov property,
  we deduce that 
  \begin{align}
    P_{\rho}^{G} 
    \left(
      Y_{G}(\tilde{T}_{B}^{(1)}) = x
    \right)
    &=
    \sum_{k=1}^{\infty}
    P_{\rho}^{G} 
    \left(
      Y_{G}(T_{B}^{(1)}) = \rho, \ldots, Y_{G}(T_{B}^{(k-1)}) = \rho, Y_{G}(T_{B}^{(k)}) = x
    \right)\\
    &=
    \sum_{k=1}^{\infty}
    P_{\rho}^{G} 
    \left(
      Y_{G}(T_{B}^{+}) = \rho  
    \right)^{k-1}
    P_{\rho}^{G}
    \left(
      Y_{G}(T_{B}^{+}) = x
    \right)\\
    &=
    \frac{P_{\rho}^{G} \left( Y_{G}(T_{B}^{+}) = x \right)}
    { 1 - P_{\rho}^{G} \left( Y_{G}(T_{B}^{+}) = \rho  \right)} \\
    &=
    \frac{c_{G|_B}(\rho, x)}{c_{G|_B}(\rho)},
  \end{align}
  where we use Theorem~\ref{4. thm: expression of trace conductance}
  to obtain the last equality.
\end{proof}

%%%%%%%%%%%%%%%%%%%%%%%%%%%%%%%%%%%%%%%%%%%%%%%%%%%%%%%%%%%%%%%%%%%%%%%%%%%%%%%%%%%%%%%%%%%%%%%%%%%%%%%%%%%%%%%%%%%%%%%%%%%%%%%%%%
% Measurability
%%%%%%%%%%%%%%%%%%%%%%%%%%%%%%%%%%%%%%%%%%%%%%%%%%%%%%%%%%%%%%%%%%%%%%%%%%%%%%%%%%%%%%%%%%%%%%%%%%%%%%%%%%%%%%%%%%%%%%%%%%%%%%%%%%

\subsection{Measurability} \label{sec: measurability}

This subsection is devoted to a measurability problem
which we face when dealing with random electrical networks.
Recall the space $\mathbb{F}$ from Definition~\ref{1. dfn: the space F}.
Define a subspace of $\mathbb{F}$ consisting of 
rooted-and-measured resistance metric spaces associated with electrical networks by setting 
\begin{equation}
  \mathbb{F}^{E}
  \coloneqq
  \left\{
    (V_{G}, R_{G}, \rho_{G}, \mu_{G}) \in \mathbb{F} 
    \mid 
    G\ \text{is a rooted electrical network}
  \right\}.
\end{equation}
Given $(V_{G}, R_{G}, \rho_{G}, \mu_{G}) \in \mathbb{F}^{E}$,
we write 
\begin{equation}
  P_{G} (\cdot)
  \coloneqq 
  P_{\rho_{G}}^{G}
  \left(
    ( Y_{G}, \ell_{G} ) \in \cdot
  \right),
  \quad 
  \cY_{G}
  \coloneqq
  (V_{G}, R_{G}, \rho_{G}, \mu_{G}, P_{G}),
\end{equation}
where we recall from \eqref{eq: discrete local time} that $\ell_G$ denotes the local time of $Y_G$.
Note that $P_{G}$ is  a probability measure on $D(\RNp, V_{G}) \times C(V_{G} \times \RNp, \RN)$
and $\cY_{G}$ is an element of $\mbM_{L}$
(recall the space $\mbM_L$ from Section~\ref{sec: the space M_L}).
The aim of this subsection is to prove that $\cY_G$ is measurable with respect to $G$.
This is verified in Corollary~\ref{4. cor: Y_G is measurable},
through approximation of $\cY_G$ by traces.

We first recall that, on finite electrical networks,
convergence of resistance metrics and conductances are equivalent.

\begin{lem} [{\cite[Lemma 2.8]{Cao_23_Convergence}}]
\label{4. lem: continuity of resistance metric and conductance}
  Fix a non-empty finite set $V$.
  Define 
  \begin{align}
    \mathscr{R}(V)
    &\coloneqq 
    \left\{ (R_{G}(x,y))_{x,y \in V} \in \RN^{V \times V} \mid G\ \text{is an electrical network with}\ V_{G} = V \right\}, \\
    \mathscr{C}(V)
    &\coloneqq 
    \left\{ (c_{G}(x, y))_{x,y \in V} \in \RN^{V \times V} \mid G\ \text{is an electrical network with}\ V_{G} = V \right\}.
  \end{align}
  We equip both of $\mathscr{R}(V)$ and $\mathscr{C}(V)$ with the Euclidean topology induced from $\RN^{V \times V}$.
  Then the map $\mathscr{R}(V) \ni (R_{G}(x,y))_{x,y \in V} \mapsto (c_{G}(x,y))_{x, y \in V} \in \mathscr{C}(V)$ is a homeomorphism.
  (NB. The map is well-defined by Theorem~\ref{3. thm: one-to-one correspondence of fomrs and metrics}.)
\end{lem}

Using the above lemma,
we derive a simplified version of Theorem~\ref{1. thm: deterministic, main result},
that is,
we show that convergence of finite electrical networks 
implies the convergence of the associated discrete-time Markov chains and their local times,
under uniform finiteness of the conductances.

\begin{prop} \label{4. prop: prototype of main result for finite electrical networks}
  For each $n \in \mathbb{N}$,
  let $G_{n}$ be a rooted electrical network with a finite vertex set.
  Assume that $(V_{G_{n}}, R_{G_{n}}, \rho_{G_{n}})$ converges to $(V_{G}, R_{G}, \rho_{G})$
  in the pointed Gromov--Hausdorff topology for some rooted electrical network $G$
  with $|V_{G}| < \infty$.
  Furthermore, assume that 
  \begin{equation}
     \sup_{n} \sup_{x \in V_{G_{n}}} c_{G_{n}}(x) < \infty.
  \end{equation}
  Then it holds that $\mathcal{Y}_{G_{n}} \to \cY_{G}$ in $\mbM_{L}$.
\end{prop}

\begin{proof}
  We may assume that $(V_{G_{n}}, R_{G_{n}}, \rho_{G_{n}})$ and $(V_{G}, R_{G}, \rho_{G})$ 
  are embedded into a common compact metric space $(M, d^{M})$ 
  in such a way that $V_{G_{n}} \to V_{G}$ in the Hausdorff topology
  and $\rho_{G_{n}} \to \rho_{G}$
  (as objects embedded into $M$).
  Recall the metric entropy from Definition~\ref{1. dfn: metric entropy}.
  By \cite[Theorem~4.15]{Noda_pre_Metrization},
  for all but countably many $\delta$,
  we have that 
  \begin{equation}  \label{4. eq: net convergence for electrical networks}
    \lim_{n \to \infty} N_{R_{G_{n}}}(V_{G_{n}}, \delta) = N_{R_{G}}(V_{G}, \delta).
  \end{equation}
  We choose $\delta$ sufficiently small so that $N_{R_{G}}(V_{G}, \delta) = |V_{G}|$,
  which is possible by the finiteness of $V_G$.
  From \eqref{4. eq: net convergence for electrical networks},
  it follows that $N_{R_{G_{n}}}(V_{G_{n}}, \delta) = |V_{G}|$
  for all sufficiently large $n$,
  which implies that $\liminf_{n \to \infty} | V_{G_{n}} | \geq |V_{G}|$.
  Assume that 
  \begin{equation}
    \limsup_{n \to \infty} | V_{G_{n}} | > |V_{G}|.
  \end{equation}
  Then,
  by the Hausdorff convergence $V_{G_{n}} \to V_{G}$,
  it is possible to find a subsequence $(n_{k})_{k \geq 1}$
  and $x^{(n_{k})}, y^{(n_{k})} \in V_{G_{n_{k}}}$ 
  such that $x^{(n_k)} \neq y^{(n_k)}$ and both $x^{(n_{k})}$ and $y^{(n_{k})}$ 
  converge to a common point $x \in V_{G}$.
  This implies that $R_{G_{n_{k}}}(x^{(n_{k})}, y^{(n_{k})}) \to 0$.
  However,
  we have that 
  \begin{align}
    R_{G_{n}}(x^{(n_{k})}, y^{(n_{k})})^{-1} 
    &=
    \inf
    \{ \mathcal{E}_{G_{n_{k}}}(f, f) 
      \mid 
      f(x^{(n_{k})}) = 1,\, f(y^{(n_{k})})=0 
    \} \\
    &\leq 
    \mathcal{E}_{G_{n_{k}}}(1_{\{x^{(n_{k})} \}}, 1_{\{ x^{(n_{k})} \}}) \\
    &= 
    c_{G_{n_{k}}}(x^{(n_{k})}) \\
    &\leq 
    \sup_{x \in V_{n_{k}}} c_{G_{n_{k}}}(x),
  \end{align}
  from which we deduce that $\sup_{x \in V_{G_{n_{k}}}} c_{G_{n_{k}}}(x) \to \infty$.
  This contradicts the assumption.
  Therefore,
  it holds that $|V_{G_{n}}| = |V_{G}|$ for all sufficiently large $n$.

  By the above result and the Hausdorff convergence $V_{G_{n}} \to V_{G}$,
  we can write $V_{G_{n}} = (x_{i}^{(n)})_{i=1}^{N}$
  and $V_{G} = (x_{i})_{i=1}^{N}$
  in such a way that $x_{i}^{(n)} \to x_{i}$ in $M$ for each $i$.
  This yields that $R_{G_{n}}(x_{i}^{(n)}, x_{j}^{(n)}) \to R_{G}(x_{i}, x_{j})$ 
  for all $i,j$.
  From Lemma~\ref{4. lem: continuity of resistance metric and conductance},
  it follows that $c_{G_{n}}(x_{i}^{(n)}, x_{j}^{(n)}) \to c_{G}(x_{i}, x_{j})$,
  which implies the convergence of transition probabilities:
  \begin{equation}
    P_{G_{n}} ( x_{i}^{(n)}, x_{j}^{(n)} )
    = 
    \frac{c_{G_{n}}(x_{i}^{(n)}, x_{j}^{(n)})}{c_{G_{n}}(x_{i}^{(n)})}
    \xrightarrow{n \to \infty}
    \frac{c_{G}{(x_{i}, x_{j})}}{c_{G}(x_{i})}
    =
    P_{G} ( x_{i}, x_{j} ), \quad
    \forall i,j.
  \end{equation}
  Thus,
  we deduce that the convergence of finite-dimensional distributions of $Y_{G_{n}}$ to those of $Y_{G}$
  as processes in $M$.
  Since we consider the discrete-time processes,
  the tightness of $(Y_{G_{n}})_{n \geq 1}$ is obvious (cf.\ \cite[Theorem 23.11]{Kallenberg_21_Foundations}).
  Therefore, 
  $Y_{G_{n}}$ started at $\rho_{G_{n}}$ converges weakly to $Y_{G}$ started at $\rho_{G}$ in the usual $J_{1}$-Skorohod topology
  as processes in $M$.
  By the Skorohod representation theorem,
  we may assume that 
  $Y_{G_{n}}$ started at $\rho_{G_{n}}$ converges to $Y_{G}$ started at $\rho_{G}$
  almost surely on some probability space.
  Fix $T > 0$ arbitrarily.
  One can check that,
  with probability $1$,
  it holds that 
  \begin{equation}
    \forall t \in [0,T], \quad
    Y_{G_{n}}(t) = x_{i}^{(n)} 
    \Leftrightarrow
    Y_{G}(t) = x_{i}
  \end{equation}
  for all sufficiently large $n$.
  This yields that 
  \begin{equation}
    \int_{0}^{t} 1_{\{ x_{i}^{(n)} \}} (Y_{G_{n}}(t))\, ds
    = 
    \int_{0}^{t} 1_{\{ x_{i} \}} (Y_{G}(t))\, ds,
    \quad 
    \forall t \in [0,T],
  \end{equation}
  which, combined with the convergence $c_{G_{n}}(x_{i}^{(n)}) \to c_{G}(x_{i})$, 
  implies that
  \begin{equation}
    \sup_{i} 
    \sup_{0 \leq t \leq T} 
    \left|
      \ell_{G_{n}}(x_{i}^{(n)}, t) - \ell_{G}(x_{i}, t)
    \right|
    \to 0.
  \end{equation}
  Therefore,
  we obtain that $\ell_{G_{n}} \to \ell_{G}$
  in $\hatC(M \times \RNp, \RN)$ almost surely,
  which completes the proof.
\end{proof}

\begin{rem}
  In Proposition~\ref{4. prop: prototype of main result for finite electrical networks},
  one cannot drop the condition $\sup_{n}\sup_{x \in V_{G_{n}}} c_{G_{n}}(x) < \infty$.
  For example,
  set $V_{G_{n}} \coloneqq \{0, 1, 1+n^{-1}\} \subseteq \RN$,
  $R_{G_{n}} \coloneqq d^{\RN}|_{V_{G_{n}} \times V_{G_{n}}}$
  and $\rho_{G_{n}} \coloneqq 0$,
  where $d^{\RN}$ denotes the Euclidean metric.
  Then $(V_{G_{n}}, R_{G_{n}}, \rho_{G_{n}})$
  converges to $\{0, 1\}$ equipped with the Euclidean metric and the root $0$.
  However,
  the associated Markov chains $Y_{G_{n}}$ does not converge 
  to the Markov chain on $\{0,1\}$
  because $Y_{G_{n}}$ is at $1$ or $1+n^{-1}$ with high probability. 
\end{rem}

Given a rooted electrical network $G$ and $r > 0$,
we define 
$\tilde{G}^{(r)}$ 
to be the trace of $G$ onto $B_{R_{G}}(\rho_{G}, r)$
equipped with the root $\rho_{\tilde{G}^{(r)}} \coloneqq \rho_{G}$. 
In the following discussions,
note that
$\mu_{\tilde{G}^{(r)}}$, which is the conductance measure associated with $\tilde{G}^{(r)}$,
is in general different from $\mu_{G}^{(r)}$ defined in \eqref{1. eq: dfn of restriction operator},
which is the measure $\mu_{G}$ restricted to $B_{R_{G}}(\rho_{G}, r)$.
Using Proposition~\ref{4. prop: prototype of main result for finite electrical networks},
we verify the measurability of $\cY_{\tilde{G}^{(r)}}$ in Proposition~\ref{4. prop: measurability of Y_G^r} below.

\begin{lem} \label{4. lem: continuity of electrical networks wrt radius}
  Fix a rooted electrical network $G$.
  Then the map $(0, \infty) \ni r \mapsto \mathcal{Y}_{\tilde{G}^{(r)}} \in \mbM_{L}$
  is left-continuous.
\end{lem}

\begin{proof}
  Since $V_{\tilde{G}^{(r)}} = B_{R_{G}}(\rho_{G}, r)$ is a finite set,
  we can find $\delta>0$ such that $V_{\tilde{G}^{(r-\delta)}} = V_{\tilde{G}^{(r)}}$.
  Hence, we obtain the desired result.
\end{proof}

\begin{prop}  \label{4. prop: measurability of Y_G^r}
  The map 
  $(0, \infty) \times \mathbb{F}^{E} \ni (r, (V_{G}, R_{G}, \rho_{G}, \mu_{G})) \mapsto \mathcal{Y}_{\tilde{G}^{(r)}} \in \mbM_{L}$
  is measurable.
\end{prop}

\begin{proof}
  Fix a measurable function $f \colon \mbM_{L} \to \RN$ assumed to be bounded and continuous.
  Given $\varepsilon > 0$,
  define a map $F_{\varepsilon} \colon (0, \infty) \times \mbM_{L} \to \RN$ by setting 
  \begin{equation}
    F_{\varepsilon} (r, (V_{G}, R_{G}, \rho_{G}, \mu_{G})) 
    \coloneqq
    \frac{1}{\varepsilon \wedge (r / 2)}
    \int_{0}^{\varepsilon \wedge (r / 2)} 
    f ( \mathcal{Y}_{\tilde{G}^{(r-s)}} )\,
    ds.
  \end{equation}
  Note that the integral on the right-hand side of the above equation is well-defined
  by Lemma~\ref{4. lem: continuity of electrical networks wrt radius}.
  Suppose that $r_{n} \to r$ 
  and $(V_{G_{n}}, R_{G_{n}}, \rho_{G_{n}}, \mu_{G_{n}}) \to (V_{G}, R_{G}, \rho_{G}, \mu_{G})$.
  Then, 
  for Lebesgue almost-every $s \in (0, r/2)$,
  it holds that 
  \begin{equation}
    (V_{\tilde{G}_{n}^{(r_{n} - s)}}, R_{\tilde{G}_{n}^{(r_{n} - s)}}, \rho_{\tilde{G}_{n}^{(r_{n} - s)}}, \mu_{G_{n}}^{(r_{n} - s)}) 
    \to 
    (V_{\tilde{G}^{(r-s)}}, R_{\tilde{G}^{(r-s)}}, \rho_{\tilde{G}^{(r-s)}}, \mu_{G}^{(r-s)})
  \end{equation}
  in the Gromov--Hausdorff--Prohorov topology
  (cf.\ \cite[Lemma~2.15]{Noda_pre_Convergence}).
  We then have that 
  \begin{equation}
    \limsup_{n \to \infty}
    \mu_{\tilde{G}_{n}^{(r_{n} - s)}} (V_{\tilde{G}_{n}^{(r_{n} - s)}})
    \leq 
    \limsup_{n \to \infty}
    \mu_{G_{n}}^{(r_{n} - s)} (V_{\tilde{G}_{n}^{(r_{n} - s)}})
    =
    \mu_{G}^{(r-s)}(V_{\tilde{G}^{(r - s)}}) < \infty.
  \end{equation}
  From Proposition~\ref{4. prop: prototype of main result for finite electrical networks}, 
  it follows that $\mathcal{Y}_{\tilde{G}_{n}^{(r_{n} - s)}} \to \mathcal{Y}_{\tilde{G}^{(r - s)}}$ 
  for Lebesgue almost-every $s \in (0, r/2)$.
  This yields that the map $F_{\varepsilon}$ is continuous.
  Moreover,
  by Lemma~\ref{4. lem: continuity of electrical networks wrt radius},
  $F_{\varepsilon}$ converges to the desired map in the assertion as $\varepsilon \to 0$ pointwise,
  which completes the proof.
\end{proof}

Thanks to Proposition~\ref{4. prop: measurability of Y_G^r},
the measurability of $\cY_G$ with respect to $G$ is deduced 
by showing the convergence of $\cY_{\tilde{G}^{(r)}}$ to $\cY_G$ as $r \to \infty$.
The key result to verify this is the following,
which is a version of \cite[Lemma 4.2]{Croydon_18_Scaling}
modified for discrete-time Markov chains on electrical networks.
  
\begin{lem} \label{4. lem: exit time estimate for Y}
  Fix $(V_{G}, R_{G}, \rho_{G}, \mu_{G}) \in \mathbb{F}^{E}$.
  For every 
  $\delta \in (0, R_{G}(\rho_{G}, B_{R_{G}}(\rho_{G}, r)^{c})),\, t \geq 0$ 
  and $\lambda >0$, 
  it holds that 
  \begin{equation}
    P_{\rho_{G}}^{G}
    (T_{B_{R_{G}}(\rho_{G},r)^{c}} \leq t)
    \leq
    \frac{1}{\lambda}
    +
    \frac{4\delta}{R_{G}(\rho_{G},B_{R_{G}}(\rho_{G}, r)^{c})} 
    + 
    \frac{4t \lambda}
    {\mu_{G}(B_{R_{G}}(\rho_{G},\delta))(R_{G}(\rho_{G}, B_{R_{G}}(\rho_{G}, r)^{c})-\delta)},
  \end{equation}
  where we recall from \eqref{4. eq: first hitting and return time of Y} that 
  $T_{B_{R_G}(\rho_G, r)^c}$ denotes the first hitting time of $B_{R_G}(\rho_B, r)^c$ by $Y_G$.
\end{lem}
  
\begin{proof}
  Write $\sigma_{r}^{X}$ and $\sigma_{r}^{Y}$ 
  for the first hitting time of $X_{G}$ and $Y_{G}$
  to $B_{R_{G}}(\rho_{G}, r)^{c}$ respectively.
  We then have that $\sigma_{r}^{X} = \sum_{k =1}^{\sigma_{r}^{Y}} S_{k}$,
  where $(S_{k})_{k \geq 1}$ denotes the sequence of the holding times of $X_{G}$.
  Note that $(S_k)_{k \geq 1}$ is independent identically distributed (i.i.d.) 
  and each $S_k$ has the exponential distribution with rate $1$.
  Since $\sigma_r^Y$ is independent of $(S_k)_{k \geq 1}$,
  the Markov inequality yields that 
  \begin{equation}
    P_{\rho}^{G} (\sigma_{r}^{X} > \lambda \sigma_{r}^{Y})
    \leq
    \frac{1}{\lambda} 
    E^G_\rho\!\left[ \frac{\sum_{k=1}^{\sigma_r^Y} S_k}{\sigma_r^Y}  \right]
    =
    \frac{1}{\lambda}.
  \end{equation}
  Thus, we deduce that 
  \begin{equation}  \label{4. eq: exit time for Y, bound X by Y}
    P_{\rho}^{G}
    ( \sigma_{r}^{Y} \leq t)
    \leq 
    P_{\rho}^{G} (\sigma_{r}^{X} > \lambda \sigma_{r}^{Y})
    +
    P_{\rho}^{G} ( \sigma_{r}^{X} \leq \lambda t)
    \leq 
    \lambda^{-1}
    +
    P_{\rho}^{G} ( \sigma_{r}^{X} \leq \lambda t).
  \end{equation}
  Now, \cite[Lemma 4.2]{Croydon_18_Scaling} yields the desired result.
\end{proof}

Using the above exit time estimate,
we verify that $\cY_{\tilde{G}^{(r)}}$ approximates $\cY_G$ as $r \to \infty$.
  
\begin{prop}  \label{4. prop: convergence of Y_G^r to Y_G}
  Fix $(V_{G}, R_{G}, \rho_{G}, \mu_{G}) \in \mathbb{F}^{E}$.
  It holds that $\mathcal{Y}_{\tilde{G}^{(r)}} \to \cY_{G}$ in $\mbM_{L}$ as $r \to \infty$.
\end{prop}
  
\begin{proof}
  Since $(V_{G}, R_{G})$ is boundedly-compact 
  and has the discrete topology,
  $B_{R_{G}(\rho_{G}, r_{0})}$ is a finite set for each $r_{0} > 0$.
  This fact, 
  combined with Lemmas~\ref{3. lem: characterization of recurrence by resistance} and~\ref{4. lem: exit time estimate for Y},
  yields that 
  \begin{equation}  \label{4. eq: trac convergence, uniform exit time estimate}
    \lim_{r \to \infty}
    \sup_{x \in B_{R_{G}(\rho_{G}, r_{0})}} 
    P_{x}^{G}
    \left(
      T_{B_{R_{G}}(x, r)^{c}} \leq 1
    \right)
    = 0,
    \quad 
    \forall r_{0} > 0.
  \end{equation}
  We regard $(V_{\tilde{G}^{(r)}}, R_{\tilde{G}^{(r)}})$
  as a subspace of $(V_{G}, R_{G})$.
  To prove the desired result,
  it is enough to show that $\mu_{\tilde{G}^{(r)}} \to \mu_{G}$ 
  and $P_{\tilde{G}^{(r)}} \to P_{G}$ as $r \to \infty$. 

  We first prove the former convergence. 
  It is easy to check that $\mu_{G}^{(r)} \to \mu_{G}$ 
  in the vague topology.
  Therefore, 
  it suffices to show that 
  $d_{V, \rho_{G}}(\mu_{\tilde{G}^{(r)}}, \mu_{G}^{(r)}) \to 0$,
  where $d_{V, \rho_{G}}$ denotes the vague metric, as recalled from Definition~\ref{dfn: vague metric}.
  Fix $\varepsilon > 0$
  and choose $r_{0}>0$ satisfying $e^{-r_{0}} < \varepsilon$.
  For all $r>r_{0}$,
  we have that 
  \begin{equation}
    d_{V, \rho_{G}}(\mu_{\tilde{G}^{(r)}}, \mu_{G}^{(r)}) 
    \leq 
    \varepsilon 
    +
    \int_{0}^{r_{0}} e^{-s} 
    \left(
      1 
      \wedge
      \dP \left( \mu_{\tilde{G}^{(r)}}^{(s)}, \mu_{G}^{(s)} \right)
    \right)
    ds,
  \end{equation}
  where $\dP$ denotes the Prohorov metric on $\cMfin(V_{G})$.
  For $s < r_{0} < r$ and $A \subseteq B_{R_{G}}(\rho_{G}, s)$,
  by using Corollary~\ref{4. cor: difference for trace measure at one point},
  we deduce that 
  \begin{align}
    \mu_{G}^{(s)}(A) 
    &= 
    \sum_{x \in A} c_{G}(x) \\
    &\leq 
    \sum_{x \in A} c_{\tilde{G}^{(r)}}(x) 
    + 
    \sum_{x \in A} c_{G}(x) P_{x}^{G}\!\left( Y_{G}(T^+_{B_{R_{G}}(\rho_{G}, r)}) = x \right) \\
    &\leq 
    \mu_{\tilde{G}^{(r)}}^{(s)}(A) 
    + 
    \mu_{G}(B_{R_{G}}(\rho_{G}, r_{0})) 
    \sup_{x \in B_{R_{G}}(\rho_{G}, r_{0})} 
    P_{x}^{G}\!\left( T_{B_{R_{G}}(\rho_{G}, r)^{c}} \leq 1 \right),
  \end{align} 
  and $\mu_{\tilde{G}^{(r)}}^{(s)}(A) \leq \mu_G^{(s)}(A)$.
  It follows that 
  \begin{equation}
    \dP \left( \mu_{\tilde{G}^{(r)}}^{(s)}, \mu_{G}^{(s)} \right)
    \leq 
    \mu_{G}(B_{R_{G}}(\rho_{G}, r_{0})) 
    \sup_{x \in B_{R_{G}}(\rho_{G}, r_{0})} 
    P_{x}^{G}\!\left( T_{B_{R_{G}}(\rho_{G}, r)^{c}} \leq 1 \right).
  \end{equation}
  This yields that 
  \begin{equation}
    d_{V, \rho_{G}}(\mu_{\tilde{G}^{(r)}}, \mu_{G}^{(r)})  
    \leq 
    \varepsilon 
    + 
    \mu_{G}(B_{R_{G}}(\rho_{G}, r_{0})) 
    \sup_{x \in B_{R_{G}}(\rho_{G}, r_{0})} 
    P_{x}^{G}\!\left( T_{B_{R_{G}}(\rho_{G}, r)^{c}} \leq 1 \right).
  \end{equation}
  For any $r > 2r_{0}$ and $x \in B_{R_{G}}(\rho_{G}, r_{0})$,
  we have that $B_{R_{G}}(x, r/2) \subseteq B_{R_{G}}(\rho_{G}, r)$,
  which implies that 
  \begin{equation}
    \sup_{x \in B_{R_{G}}(\rho_{G}, r_{0})} 
    P_{x}^{G}\!\left( T_{B_{R_{G}}(\rho_{G}, r)} \leq 1 \right) 
    \leq 
    \sup_{x \in B_{R_{G}}(\rho_{G}, r_{0})}
    P_{x}^{G}\!\left( T_{B_{R_{G}}(x, r/2)^{c}} \leq 1 \right).
  \end{equation}
  Hence, using \eqref{4. eq: trac convergence, uniform exit time estimate},
  we obtain that $\mu_{\tilde{G}^{(r)}} \to \mu_{G}$ vaguely 
  as measures on $V_{G}$.

  We next show that $P_{\tilde{G}^{(r)}} \to P_{G}$
  as probability measures on $D(\RNp, V_{G}) \times \hatC(V_{G} \times \RNp, \RN)$.
  We write $\tr_{r} Y_{G} \coloneqq \tr_{B_{R_{G}}(\rho_{G}, r)} Y_{G}$,
  where we recall this notation from Definition~\ref{4. dfn: trace}.
  We then define a random element $\tr_{r} \ell_{G}$ of $C(V_{\tilde{G}^{(r)}} \times \RNp, \RN)$
  by setting 
  \begin{equation}
    \tr_{r} \ell_{G}(x, t)
    \coloneqq
    \frac{1}{c_{\tilde{G}^{(r)}}(x)}
    \int_{0}^{t} 
    1_{\{x\}}(\tr_{r}Y_{G}(s))\,
    ds.
  \end{equation}
  By Theorem~\ref{4. thm: coupling by trace of Y_G},
  we have that $(Y_{\tilde{G}^{(r)}}, \ell_{\tilde{G}^{(r)}}) \stackrel{\mathrm{d}}{=} (\tr_{r}Y_{G}, \tr_{r}\ell_{G})$.
  Since $\tr_{r}Y_{G}(s) = Y_{G}(s)$ for all $0 \leq s \leq t$ on the event $\{T_{B_{R_{G}}(\rho_{G}, r)^{c}} > t\}$,
  it holds that, for all $\varepsilon > 0$,
  \begin{equation}
    P_{\rho_{G}}^{G}
    \left(
      \sup_{0 \leq s \leq t} R_{G}(\tr_{r}Y_{G}(s), Y_{G}(s)) > \varepsilon
    \right)
    \leq 
    P_{\rho_{G}}^{G}
    \left(
      T_{B_{R_{G}}(\rho_{G}, r)^{c}} > t
    \right).
  \end{equation}
  From \eqref{4. eq: trac convergence, uniform exit time estimate},
  we deduce that $\tr_{r}Y_{G}$ converges to $Y_{G}$ in probability
  as random elements of $D(\RNp, V_{G})$.
  Since $\tr_{r} \ell_{G}(x, s) = \frac{c_{G}(x)}{c_{\tilde{G}^{(r)}}(x)} \ell_{G}(x, s)$ 
  for all $0 \leq s \leq t$ on the event $\{T_{B_{R_{G}}(\rho_{G}, r)^{c}} > t\}$,
  we obtain that, for all $r> r_{0}$ and $t >0$,
  \begin{align} 
    &P_{\rho_{G}}^{G}
    \left(
      \sup_{x \in V_{G}^{(r_{0})}}
      \sup_{0 \leq s \leq t}
      \left|
        \tr_{r}\ell_{G}(x, s) - \ell_{G}(x, s)
      \right|
      >
      \varepsilon
    \right) \\
    \leq 
    &P_{\rho_{G}}^{G}
    \left(
      T_{B_{R_{G}}(\rho_{G}, r)^{c}} \leq t
    \right)
    +
    P_{\rho_{G}}^{G} 
    \left(
      \sup_{x \in V_{G}^{(r_{0})}} 
      \left|
        1 - \frac{c_{G}(x)}{c_{\tilde{G}^{(r)}}(x)}
      \right|
      \ell_{G}(x, t)
      >
      \varepsilon/2
    \right).
    \label{4. eq: trac convergence, dis of local times}
  \end{align}
  By Corollary~\ref{4. cor: difference for trace measure at one point},
  we deduce that 
  \begin{align}
    \sup_{x \in V_{G}^{(r_{0})}}
    \left|
        1 - \frac{c_{G}(x)}{c_{\tilde{G}^{(r)}}(x)}
    \right|
    =
    \sup_{x \in V_{G}^{(r_{0})}}
    \left|
        \frac{ P_{x}^{G}\bigl( Y_{G}(T_{V_{G}^{(r)}}) = x \bigr) }{1 - P_{x}^{G}\bigl( Y_{G}(T_{V_{G}^{(r)}}) = x \bigr) }
    \right|.
  \end{align}
  By the same argument as before,
  since it holds that 
  \begin{equation}
    \limsup_{r \to \infty}
    \sup_{x \in V_{G}^{(r_{0})}}
    P_{x}^{G}\bigl( Y_{G}(T_{V_{G}^{(r)}}) = x \bigr)
    =
    0,
  \end{equation}
  we obtain that 
  \begin{equation}
    \lim_{r \to \infty}
    \sup_{x \in V_{G}^{(r_{0})}}
    \left|
        1 - \frac{c_{G}(x)}{c_{\tilde{G}^{(r)}}(x)}
    \right|
    =
    0.
  \end{equation}
  This and the tightness of $(\ell_{G}(x, t))_{x \in V_{G}^{(r_{0})}}$ implies that 
  \begin{equation}  \label{4. eq: trac convergence, conductance part}
    \lim_{r \to \infty}
    P_{\rho_{G}}^{G} 
    \left(
      \sup_{x \in V_{G}^{(r_{0})}} 
      \left|
        1 - \frac{c_{G}(x)}{c_{\tilde{G}^{(r)}}(x)}
      \right|
      \ell_{G}(x, t)
      >
      \varepsilon/2
    \right)
    =
    0.
  \end{equation}
  From \eqref{4. eq: trac convergence, uniform exit time estimate},
  \eqref{4. eq: trac convergence, dis of local times}
  and \eqref{4. eq: trac convergence, conductance part},
  we deduce that 
  \begin{equation}
    \lim_{r \to \infty}
    P_{\rho_{G}}^{G}
    \left(
      \sup_{x \in V_{G}^{(r_{0})}}
      \sup_{0 \leq s \leq t}
      \left|
        \tr_{r}\ell_{G}(x, s) - \ell_{G}(x, s)
      \right|
      >
      \varepsilon
    \right)
    =0,
    \quad 
    \forall r_{0}, t >0,
  \end{equation}
  which implies that $\tr_{r} \ell_{G} \to \ell_{G}$ in probability
  as random elements of $\hatC(V_{G} \times \RNp, \RN)$.
  Therefore,
  we deduce that $(Y_{\tilde{G}^{(r)}}, \ell_{\tilde{G}^{(r)}})$
  converges to $(Y_{G}, \ell_{G})$
  in probability
  as random elements of $D(\RNp, V_{G}) \times \hatC(V_{G} \times \RNp, \RN)$,
  which completes the proof.
\end{proof}

Combining Proposition~\ref{4. prop: measurability of Y_G^r} with Proposition~\ref{4. prop: convergence of Y_G^r to Y_G},
we obtain the measurability of $\cY_{G}$ with respect to $G$.

\begin{cor} \label{4. cor: Y_G is measurable}
  The map $\mathbb{F}^{E} \ni (V_{G}, R_{G}, \rho_{G}, \mu_{G}) \mapsto \cY_{G} \in \mbM_{L}$ is measurable.
\end{cor}

%%%%%%%%%%%%%%%%%%%%%%%%%%%%%%%%%%%%%%%%%%%%%%%%%%%%%%%%%%%%%%%%%%%%%%%%%%%%%%%%%%%%%%%%%%%%%%%%%%%%%%%%%%%%%%%%%%%%%%%%%%%%%%%%%%
%%%%%%%%%%%%%%%%%%%%%%%%%%%%%%%%%%%%%%%%%%%%%%%%%%%%%%%%%%%%%%%%%%%%%%%%%%%%%%%%%%%%%%%%%%%%%%%%%%%%%%%%%%%%%%%%%%%%%%%%%%%%%%%%%%
% Proof of main result 1
%%%%%%%%%%%%%%%%%%%%%%%%%%%%%%%%%%%%%%%%%%%%%%%%%%%%%%%%%%%%%%%%%%%%%%%%%%%%%%%%%%%%%%%%%%%%%%%%%%%%%%%%%%%%%%%%%%%%%%%%%%%%%%%%%%
%%%%%%%%%%%%%%%%%%%%%%%%%%%%%%%%%%%%%%%%%%%%%%%%%%%%%%%%%%%%%%%%%%%%%%%%%%%%%%%%%%%%%%%%%%%%%%%%%%%%%%%%%%%%%%%%%%%%%%%%%%%%%%%%%%

\section{Proof of Theorem~\ref{1. thm: deterministic, main result}} \label{sec: proof of main result 1}

In this section, we prove Theorem~\ref{1. thm: deterministic, main result}.
The first part of the assertion of Theorem~\ref{1. thm: deterministic, main result} follows 
from \cite[Proposition 5.1]{Noda_pre_Convergence} 
and Theorem~\ref{5. thm: recurrence is preserved by the non-explosion condition} below.
Recall from Section~\ref{sec: the local GHV} the space $\mathbb{D}$ consisting of (equivalence classes of)
rooted boundedly-compact metric spaces 
equipped with the local Gromov--Hausdorff topology.

\begin{thm} \label{5. thm: recurrence is preserved by the non-explosion condition}
  For each $n \geq 1$,
  let $(F_{n}, R_{n}, \rho_{n})$
  be an element of $\mathbb{D}$ such that $R_{n}$ is a resistance metric.
  Suppose that $(F_{n}, R_{n}, \rho_{n})$ converges to
  some $(F, R, \rho) \in \mathbb{D}$
  in the local Gromov--Hausdorff topology.
  Then $R$ is a resistance metric
  and it holds that,
  for all but countably many $r > 0$, 
  \begin{equation} \label{5. thm eq: preserving recurrence, effective resistance comparison}
  R(\rho, B_{R}(\rho, r)^{c})
  \geq 
  \limsup_{n \to \infty}
  R_{n}(\rho_{n}, B_{R_{n}}(\rho_{n}, r)^{c}).
  \end{equation}
  In particular,
  if it holds that 
  \begin{equation}
  \lim_{r \to \infty}
  \limsup_{n \to \infty}
  R_{n}(\rho_{n}, B_{R_{n}}(\rho_{n}, r)^{c})
  = 
  \infty,
  \end{equation}
  then the resistance form associated with $(F, R)$ is recurrent.
\end{thm}

\begin{proof}
  By \cite[Theorem~4.9]{Noda_pre_Metrization},
  we may assume that 
  $(F_{n}, R_{n})$ and $(F,R)$
  are isometrically embedded into some rooted boundedly-compact metric space $(M, d^{M}, \rho_{M})$
  in such a way that $\rho_{n} = \rho = \rho_{M}$ as elements in $M$ and
  $F_{n} \to F$ in the Fell topology in $M$.
  Fix a finite subset $S = \{x_{i}\}_{i = 1}^{N} \subseteq F$.
  Using the convergence $F_{n} \to F$,
  we can find $\{y_{i}^{(n)}\}_{i=1}^{N} \subseteq F_{n}$
  such that $y_{i}^{(n)} \to x_{i}$ in $M$ for each $i$.
  It is then from \cite[Lemma 2.8]{Cao_23_Convergence} that 
  $R|_{S \times S}$ is a resistance metric on $S$.
  Hence $R$ is a resistance metric.

  Write $(\cE_{n}, \cF_{n})$ and $(\cE, \cF)$ 
  for the resistance forms associated with $(F_{n}, R_{n})$ and $(F, R)$,
  respectively.
  Let $r > 0$ be such that 
  \begin{equation}  \label{5. pr eq: 1, recurrence is preserved by the non-explosion condition}
    \cl(D_R(\rho, r)^c) = B_R(\rho, r)^c.
  \end{equation}
  (Since $(F, R)$ is boundedly compact, 
  all but countably many $r$ satisfy the above.)
  If the right-hand side of \eqref{5. thm eq: preserving recurrence, effective resistance comparison} is equal to $0$,
  then \eqref{5. thm eq: preserving recurrence, effective resistance comparison} is obvious.
  We consider the other case:
  \begin{equation}
    C_r \coloneqq \limsup_{n \to \infty} R_{n}(\rho_{n}, B_{R_{n}}(\rho_{n}, r)^{c}) > 0.
  \end{equation}
  Choose a subsequence $(n_{k})_{k \geq 1}$ satisfying
  \begin{equation}
    \lim_{k \to \infty} R_{n_{k}}(\rho_{n_{k}}, B_{R_{n_{k}}}(\rho_{n_{k}}, r)^{c}) = C_r.
  \end{equation}
  Let $f_{n_k}$ be a unique function in $\cF_{n_k}$ 
  such that $f_{n_k}(\rho_{n_{k}}) = 1, \, f_{n_k}|_{B_{R_{n_{k}}}(\rho_{n_{k}}, r)^{c}} \equiv 0$
  and $\cE_{n_{k}}(f_{n_k}, f_{n_k})^{-1} =  R_{n_{k}}(\rho_{n_{k}}, B_{R_{n_{k}}}(\rho_{n_{k}}, r)^{c})$
  (see \cite[Section 4]{Kigami_12_Resistance}).
  Note that $0 \leq f_{n_k} \leq 1$.
  Using \eqref{3. eq: the basic inequality for resistance forms},
  we obtain that,
  for all $r_{0}, \delta >0$,
  \begin{equation}
    \limsup_{k \to \infty}
    \sup_{\substack{x, y \in F_{n_k}^{(r_{0})},\\ R_{n_{k}}(x, y) < \delta}}
    |f_{n}(x) - f_{n}(y)|
    \leq 
    \limsup_{k \to \infty}
    \sup_{\substack{x, y \in F_{n_k}^{(r_{0})},\\ R_{n_{k}}(x, y) < \delta}}
    \sqrt{\cE_{n_{k}}(f_{n_k}, f_{n_k}) R_{n_{k}}(x, y)}
    \leq 
    \sqrt{C_r^{-1} \delta}.
  \end{equation}
  Hence,
  by \cite[Theorem~3.31]{Noda_pre_Metrization}, 
  there exist a further subsequence $(n_{k(l)})_{l \geq 1}$ and a function $f \in C(F, \mathbb{R})$ such that 
  \begin{equation}  \label{5. eq: preserving recurrence, convergence of green functions}
    \lim_{\delta \to 0}
    \limsup_{l \to \infty}
    \sup_{\substack{x \in F_{n_{k(l)}}^{(r_{0})}, y \in F^{(r_{0})},\\ d^{M}(x, y) < \delta}}
    |f_{n_{k(l)}}(x) - f(y)|
    =
    0,
    \quad 
    \forall r_{0} >0.
  \end{equation}
  To simplify index, from now on we will assume that
  the above convergence holds for the original sequence $(f_n)_{n \geq 1}$.
  This immediately yields that $f(\rho) = 1$.
  Moreover, it holds that $f|_{B_{R}(\rho, r)^{c}} \equiv 0$.
  To check this,
  fix $x \in D_R(\rho, r)^c$.
  From the convergence of $F_n$ to $F$,
  we can choose elements $x_n \in F_n$, $n \geq 1$, so that $x_n \to x$ in $M$.
  Since $R(\rho, x) > r$,
  we have that $R_n(\rho_n, x_n) > r$ 
  for all sufficiently large $n$.
  Recalling that $f_n|_{B_{R_n}(\rho_n, r)^c} \equiv 0$,
  we obtain that $f(x) = 0$.
  This proves that $f|_{D_R(\rho, r)^c} \equiv 0$.
  The continuity of $f$ and \eqref{5. pr eq: 1, recurrence is preserved by the non-explosion condition}
  then yield that $f|_{B_{R}(\rho, r)^{c}} \equiv 0$.
  We now let $F^{*} = \{x_{i}\}_{i=1}^{\infty}$ be a countable dense subset of $F$.
  We then define $\{c^{(N)}(x_{i}, x_{j}) \mid 1 \leq i, j \leq N\}$
  to be the conductances on $F^{(N)} \coloneqq \{x_{i}\}_{i=1}^{N}$
  such that the associated effective resistance coincides with $R|_{F^{(N)} \times F^{(N)}}$.
  For each $N$, 
  let $F_{n}^{(N)} \coloneqq \{x_{n, i}^{(N)} \mid 1 \leq i \leq N\}$ be a subset of $F_{n}$
  such that $x_{n, i}^{(N)} \to x_{i}$ as $n \to \infty$ in $M$ for each $i$.
  Define $\{c_{n}^{(N)}(x_{n, i}^{(N)}, x_{n, j}^{(N)}) \mid 1 \leq i, j \leq N \}$
  be the conductances on $F_{n}^{(N)}$ such that 
  the associated effective resistance coincides with $R_{n}|_{F_{n}^{(N)} \times F_{n}^{(N)}}$.
  By Lemma~\ref{4. lem: continuity of resistance metric and conductance},
  we have that $c_{n}^{(N)}(x_{n, i}^{(N)}, x_{n, j}^{(N)}) \to c^{(N)}(x_{i}, x_{j})$,
  and by \eqref{5. eq: preserving recurrence, convergence of green functions},
  we have that $f_n(x_{n, i}^{(N)}) \to f(x_{i})$ as $n \to \infty$.
  This yields that 
  \begin{align}
    \cE_n|_{F_n^{(N)}} (f_n|_{F_n^{(N)}}, f_n|_{F_n^{(N)}})
    &=
    \frac{1}{2}
    \sum_{1 \leq i, j \leq N}
    c_n^{(N)}(x_{n, i}^{(N)}, x_{n, j}^{(N)})
    (f_n(x_{n_, i}^{(l)}) - f_n(x_{n, j}^{(l)}))^{2} \\
    &\xrightarrow{l \to \infty} 
    \frac{1}{2}
    \sum_{1 \leq i, j \leq N}
    c^{(N)}(x_{i}, x_{j})
    (f(x_{i}) - f(x_{j}))^{2},
  \end{align}
  where recall the trace of resistance forms from Theorem~\ref{3. thm: trace and resistance metric space}.
  Therefore,
  we deduce that 
  \begin{align}
    \limsup_{N \to \infty} 
    \frac{1}{2}
    \sum_{1 \leq i, j \leq N}
    c^{(l)}(x_{i}, x_{j})
    (f(x_{i}) - f(x_{j}))^{2}
    &=
    \limsup_{N \to \infty}
    \lim_{n \to \infty}
    \cE_n|_{F_n}^{(N)} (f_n|_{F_n^{(N)}}, f_n|_{F_n^{(N)}})\\
    &\leq 
    \lim_{l \to \infty}
    \cE_n(f_n, f_n)\\
    &=
    C_{r}^{-1}
  \end{align}
  This implies that $f \in \cF$ and $\cE(f, f) \leq C_{r}^{-1}$ (see \cite[Theorem 3.13]{Kigami_12_Resistance}).
  Recalling that $f(\rho) = 1$ and $f|_{B_{R}(\rho, r)^{c}} \equiv 0$,
  we obtain that $R(\rho, B_{R}(\rho, r)^{c}) \geq C_{r}$,
  which completes the proof.
\end{proof}

The most important ingredient for showing convergence of local times is a quantitative estimate of 
the equicontinuity of local times, 
which was proved by Croydon in \cite{Croydon_15_Moduli}.
For an electrical network $G$,
set 
\begin{equation} \label{eq: diameter and total mass}
  r(G) \coloneqq \max_{x,y \in V_{G}} R_{G}(x,y), 
  \quad
  m(G) \coloneqq \mu_{G}(V_{G}).
\end{equation}

\begin{lem} [{\cite[Theorem 1.1]{Croydon_15_Moduli}}]
\label{5. lem: weak version of equicontinuity estimate by Croydon}
  Let $G$ be a rooted electrical network with a finite vertex set.
  Fix $T > 0$.
  There exist a constant $c_{1}(T)$ depending only on $T$ and a universal constant $c_{2} > 0$
  such that, 
  for any $\lambda > 0$,
  \begin{equation}
    \max_{x, y \in V_{G}}
    P^{G}_{\rho_{G}}
    \left(
      \max_{0 \leq t \leq T m(G) r(G)} 
      r(G)^{-1}
      \left|
        \ell_{G}(x, t) - \ell_{G}(y, t)
      \right|
      \geq 
      \lambda 
      \sqrt{ r(G)^{-1} R_{G} (x, y)}
    \right)
    \leq 
    c_{1}(T) e^{-c_{2} \lambda}.
  \end{equation}
\end{lem}

Following the proof of \cite[Theorem~4.19]{Noda_pre_Convergence},
we obtain a quantitative estimate of equicontinuity of the local time of a discrete-time Markov chain on an electrical network.

\begin{thm} \label{5. thm: uniform continuity estimate of discrete local times}
  Fix $\alpha \in (0,1/2)$ and $T>0$.
  There exist a constant $c_{3}(\alpha) \in (0, \infty)$ depending only on $\alpha$ 
  and a constant $c_{4}(T) > 0$ depending only on $T$ such that,
  for any rooted electrical network $G$ with finite vertex set and $N \in \mathbb{N}$,
  \begin{align}
    &P_{\rho_{G}}^G\!
    \left(
      \sup_{\substack{ x,y \in V_{G} \\ r(G)^{-1} R_{G}(x,y) < 2^{-N+1}}}
      \sup_{0 \leq t \leq T r(G) m(G)}
      r(G)^{-1}
      | \ell_{G}(x, t)- \ell_{G}(y, t)|
      >
      c_{3}(\alpha) \, 2^{-(\frac{1}{2}-\alpha)N}
    \right)\\
    \leq
    &c_{4}(T)
    \sum_{k \geq N} 
    (k+1)^{2} 
    N_{r(G)^{-1} R_{G}}(V_{G},2^{-k})^{2} 
    \exp \left(-2^{\alpha(k-3)}\right).
    \end{align}
\end{thm}

We will also use the following result regarding the upper bound of local times,
which is established in the proof of \cite[Theorem~1.1]{Croydon_15_Moduli}.

\begin{lem} [{Proof of \cite[Theorem 1.1]{Croydon_15_Moduli}}]
  \label{5. lem: weak version of upper bound estimate by Croydon}
  Let $G$ be a rooted electrical network with a finite vertex set.
  Fix $T > 0$.
  There exist constants $c_{1}(T)$ and $c_{2}$
  such that, 
  for any $\lambda > 0$,
  \begin{equation}
    \max_{x \in V_{G}}
    P^{G}_{\rho_{G}}
    \bigl(
      r(G)^{-1}
      \ell_{G}(x, Tm(G)r(G)) 
      \geq 
      \lambda 
    \bigr)
    \leq 
    c_{1}(T) e^{-c_{2} \lambda}.
  \end{equation}
  The constants $c_{1}$ and $c_{2}$ are the same 
  as the constants of Lemma~\ref{5. lem: weak version of equicontinuity estimate by Croydon}.
\end{lem}

Using the metric entropy introduced in Definition~\ref{1. dfn: metric entropy}
and the continuity estimate established in Theorem~\ref{5. thm: uniform continuity estimate of discrete local times},
we extend the estimate in Lemma~\ref{5. lem: weak version of upper bound estimate by Croydon}
to a uniform estimate, as follows.

\begin{prop}  \label{5. prop: uniform bound estimate of discrete local times}
  Fix $\alpha \in (0,1/2)$ and $T>0$.
  There exist constants $c_{5}(\alpha) > 0$ depending only on $\alpha$ 
  and $c_{6}(T) > 0$ depending only on $T$ such that,
  for any rooted electrical network $G$ with finite vertex set and $N \in \mathbb{N}$,
  \begin{align}
    &P^G_{\rho_{G}}
    \left(
      \sup_{x \in V_{G}}
      r(G)^{-1}
      \ell_{G}(x, Tm(G)r(G))
      >
      2^{-c_{5}(\alpha)N}
    \right)\\
    \leq
    &c_{6}(T)
    \sum_{k \geq N-1} 
    (k+1)^{2} 
    N_{r(G)^{-1} R_{G}}(V_{G},2^{-k})^{2} 
    \exp \left(-2^{\alpha(k-3)}\right)
    \end{align}
\end{prop}

\begin{proof}
  Set $\lambda_{N} \coloneqq c_{2}^{-1}2^{\alpha (N-4)} + c_{3}(\alpha) 2^{-(\frac{1}{2} - \alpha)N}$.
  Assume that 
  \begin{equation}
    \sup_{\substack{ x,y \in V_{G} \\ r(G)^{-1} R_{G}(x,y) < 2^{-N+1}}}
    \sup_{0 \leq t \leq T r(G) m(G)}
    r(G)^{-1}
    | \ell_{G}(x, t)- \ell_{G}(y, t)|
    \leq
    c_{3}(\alpha) \, 2^{-(\frac{1}{2}-\alpha)N}.
  \end{equation}
  Let $\{x_{i}\}_{i=1}^{K}$ be a minimal $2^{-N+1}$-covering of the metric space $(V_{G}, r(G)^{-1}R_{G})$.
  Note that by definition $K = N_{r(G)^{-1}R_{G}}(V_{G}, 2^{-N+1})$.
  If we have that, for all $i$,
  \begin{equation}
    r(G)^{-1} \ell_{G}(x_{i}, T m(G) r(G)) < c_{2}^{-1} 2^{\alpha (N-4)}, 
  \end{equation}
  then it follows that, for all $x \in V_{G}$,
  \begin{equation}
    r(G)^{-1} \ell_{G}(x, Tm(G)r(G)) < \lambda_{N}.
  \end{equation}
  Hence, by Theorem~\ref{5. thm: uniform continuity estimate of discrete local times}
  and Lemma~\ref{5. lem: weak version of upper bound estimate by Croydon}
  together with a union bound applied across the points $\{x_{i}\}_{i=1}^{K}$,
  we deduce that 
  \begin{align}
    &P^G_{\rho_{G}}
    \left(
      \sup_{x \in V_{G}}
      r(G)^{-1}
      \ell_{G}(x, Tm(G)r(G))
      >
      \lambda_{N}
    \right)\\
    \leq
    &c_{4}(T)
    \sum_{k \geq N} 
    (k+1)^{2} 
    N_{r(G)^{-1} R_{G}}(V_{G},2^{-k})^{2} 
    \exp \left(-2^{\alpha(k-3)}\right) \\
    &\quad 
    +
    c_{1}(T) 
    N_{r(G)^{-1}R_{G}}(V_{G}, 2^{-N+1}) 
    \exp(-2^{\alpha (N -4)})\\
    \leq 
    & 
    \left(c_{4}(T) \vee c_{1}(T) \right)
    \sum_{k \geq N-1} 
    (k+1)^{2} 
    N_{r(G)^{-1} R_{G}}(V_{G},2^{-k})^{2} 
    \exp \left(-2^{\alpha(k-3)}\right).
  \end{align}
  Choosing $c_{5}(\alpha)$ large enough so that $\lambda_{N} \leq 2^{c_{5}(\alpha) N}$ for all $N$,
  we obtain the desired result.
\end{proof}

Now, we are ready to start proving Theorem~\ref{1. thm: deterministic, main result}.
First, we provide a result regarding the exit times of the Markov chains,
which roughly asserts that, up to each fixed time, the chains are contained in a common large ball 
(with high probability).
This is an analogue of \cite[Lemma~5.3]{Noda_pre_Convergence} in the discrete setting.

\begin{lem} \label{5. lem: precompactness for exit times}
  Under Assumption~\ref{1. assum: deterministic version}\ref{1. assum item: deterministic, convergence of spaces} 
  and \ref{1. assum item: deterministic, the non-explosion condition},
  it holds that 
  \begin{equation}  \label{5. lem eq: precompactness for discrete exit times}
    \lim_{r \to \infty}
    \limsup_{n \to \infty} 
    P_{\rho_{n}}^{G_{n}}\!
    \left(
      T_{B_{\hat{R}_{n}}(\hat{\rho}_{n}, r)^{c}} \leq a_{n}b_{n}t
    \right)
    = 0, 
    \quad 
    \forall t > 0,
  \end{equation}
  which is equivalent to 
  \begin{equation}  \label{5. lem eq: precompactness for continuous exit times}
    \lim_{r \to \infty}
    \limsup_{n \to \infty} 
    P_{\rho_{n}}^{G_{n}}\!
    \left(
      \sigma_{B_{\hat{R}_{n}}(\hat{\rho}_{n}, r)^{c}} \leq a_{n}b_{n}t
    \right)
    = 0, 
    \quad 
    \forall t > 0,
  \end{equation}
  where we recall from \eqref{3. eq: dfn of hitting time} and \eqref{4. eq: first hitting and return time of Y} 
  that, given an electrical network $G$, 
  $\sigma_{B}$ denotes the hitting time of $B$ by the associated constant speed random walk $X_{G}$
  and $T_{B}$ denotes the first hitting time of $B$ by the associated discrete-time Markov chain $Y_{G}$.
\end{lem}

\begin{proof}
  By Lemma~\ref{4. lem: exit time estimate for Y},
  we obtain \eqref{5. lem eq: precompactness for discrete exit times}.
  Write $(S_n^{(k)})_{k \geq 1}$ for the holding times of $X_{G_{n}}$. 
  Note that $(S_n^{(k)})_{k \geq 1}$ is a sequence of i.i.d.\ random variables 
  from the exponential distribution with mean $1$.
  Set $C_{n}(r) \coloneqq B_{\hat{R}_{n}}(\hat{\rho}_{n}, r)^{c}$.
  It is then the case that 
  $\sigma_{C_{n}(r)} = \sum_{k=1}^{T_{C_{n}(r)}} S_{n}^{(k)}$.
  Suppose that \eqref{5. lem eq: precompactness for discrete exit times} holds.
  Since we have that 
  \begin{equation}
    P_{\rho_{n}}^{G_{n}}\!
    \left( 
      \sigma_{C_{n}(r)} \leq a_{n}b_{n}t
    \right)
    \leq 
    P_{\rho_{n}}^{G_{n}}\!
    \left(
      a_{n}^{-1} b_{n}^{-1} \sum_{k=1}^{\lfloor 2a_{n}b_{n}t \rfloor} S_{n}^{(k)} \leq t
    \right)
    +
    P_{\rho_{n}}^{G_{n}}\!
    \left(
      T_{C_{n}(r)} \leq 2a_{n}b_{n}t
    \right),
  \end{equation}
  the weak law of large numbers and \eqref{5. lem eq: precompactness for discrete exit times} 
  yield \eqref{5. lem eq: precompactness for continuous exit times}.
  Using \eqref{4. eq: exit time for Y, bound X by Y},
  one can check that \eqref{5. lem eq: precompactness for continuous exit times} implies 
  \eqref{5. lem eq: precompactness for discrete exit times}.
\end{proof}

In \cite[Proposition~5.4]{Noda_pre_Convergence},
it is proven by the author that 
convergence of (scaled) electrical networks and the exit time estimate as in Lemma~\ref{4. lem: exit time estimate for Y}
imply the convergence of (scaled) discrete-time Markov chains.
Using this and the time-change relation given in \ref{4. eq: coupling of X_G and Y_G},
we establish convergence of (scaled) discrete-time Markov chains, as follows.

\begin{prop}  \label{5. prop: convergence of processes}
  If Assumption~\ref{1. assum: deterministic version}\ref{1. assum item: deterministic, convergence of spaces} 
  and \eqref{5. lem eq: precompactness for discrete exit times} are satisfied,
  then it holds that 
  \begin{equation}
    \left(
      \hat{V}_{n}, \hat{R}_{n}, \hat{\rho}_{n}, \hat{\mu}_{n}, P_{\rho_{n}}^{G_{n}}(\hat{Y}_{n} \in \cdot)
    \right)
    \to 
    \left(
      F, R, \rho, \mu, P_{\rho}^{G}(X_{G} \in \cdot)
    \right)
  \end{equation}
  as elements in $\mbM$ (recall this space from Section~\ref{sec: the space M_L}).
\end{prop}

 \begin{proof}
    Set $\hat{X}_{n}(t) \coloneqq X_{G_{n}}(a_{n} b_{n} t)$.
    Using Theorem~\ref{4. thm: the Hunt process associated with electrical network},
    one can check that $\hat{X}_{n}$ is the process associated with $(\hat{V}_{n}, \hat{R}_{n}, \hat{\mu}_{n})$.
    By \cite[Proposition 5.4]{Noda_pre_Convergence} and Lemma~\ref{5. lem: precompactness for exit times},
    we have that 
    \begin{equation}
    \left(
      \hat{V}_{n}, \hat{R}_{n}, \hat{\rho}_{n}, \hat{\mu}_{n}, P_{\rho_{n}}^{G_{n}}(\hat{X}_{n} \in \cdot)
    \right)
    \to 
    \left(
      F, R, \rho, \mu, P_{\rho}^{G}(X_{G} \in \cdot)
    \right)
  \end{equation}
  as elements in $\mbM$.
  From Theorem~\ref{2. thm: space M, convergence},
  we may assume that
  $(\hat{V}_{n}, \hat{R}_{n}, \hat{\rho}_{n})$ and $(F, R, \rho)$ are embedded 
  into a rooted boundedly-compact metric space $(M, d^{M}, \rho_{M})$
  in such a way that
  $\hat{\rho}_{n} = \rho = \rho_{M}$ as elements of $M$,
  $\hat{V}_{n} \to F$ in the Fell topology in $M$,
  $\hat{\mu}_{n} \to \mu$ vaguely as measures on $M$
  and $P_{\rho_{n}}^{G_{n}}(\hat{X}_{n} \in \cdot) \to P_{\rho}^{G}(X_{G} \in \cdot)$
  as probability measures on $D(\RNp, M)$.
  Using the Skorohod representation theorem,
  we may further assume that $\hat{X}_{n}$ started at $\rho_{n}$ and $X_{G}$ started at $\rho$
  are coupled in such a way that 
  $\hat{X}_{n}$ converges to $X_{G}$ almost surely 
  in $D(\RNp, M)$.
  We denote the underlying probability measure by $P$.
    
  Let $(S_{n}^{(k)})_{k \geq 1}$ be the sequence of the holding times of $X_{G_{n}}$
  and let $(J_{n}^{(k)})_{k \geq 0}$ be the sequence of the jump times of $X_{G_{n}}$
  with $J_{n}^{(0)} \coloneqq 0$.
  Note that $J_{n}^{(k)} = \sum_{i=1}^{k} S_{n}^{(k)}$.
  Recall from \eqref{4. eq: coupling of X_G and Y_G} that we have that 
  $\hat{Y}_{n}(t) = X_{G_{n}}(J_{n}^{( \lfloor a_{n}b_{n}t \rfloor )})$.
  Define an increasing cadlag function $\lambda:\RNp \to \RNp$ by setting
  \begin{equation}
    \lambda_{n}(t)
    \coloneqq 
    \frac{J_{n}^{\lfloor a_{n} b_{n} t \rfloor}}{a_{n} b_{n}}.
  \end{equation}
  It is then the case that 
  \begin{equation}  \label{5. eq: process convergence, relation of Yhat and Xhat with time-change}
    \hat{Y}_{n}(t)
    =
    X_{G_{n}}(J_{n}^{ \lfloor a_{n} b_{n} t \rfloor})
    =
    X_{G_{n}} (a_{n} b_{n} \lambda_{n}(t))
    =
    \hat{X}_{n} \circ \lambda_{n} (t).
  \end{equation}
  Since $(S_{n}^{(k)})_{k \geq 1}$ is i.i.d.\ of the exponential distribution with mean $1$,
  using the strong law of large numbers for triangular arrays (cf.\ \cite{Taylor_Hu_87_Strong}),
  we deduce that,
  for each $t>0$,
  \begin{equation}
    \sup_{0 \leq s \leq t}
    | \lambda_{n}(s) - s |
    \xrightarrow{\mathrm{a.s.}}
    0.
  \end{equation}
  In particular,
  $\lambda_{n} \to \mathrm{id}_{\RNp}$ in $D(\RNp, \RNp)$.
  Combining this with $\hat{X}_{n} \to X_{G}$ in $D(\RNp, M)$ 
  and \eqref{5. eq: process convergence, relation of Yhat and Xhat with time-change},
  we obtain $\hat{Y}_{n} \to X_{G}$ almost surely in $D(\RNp, M)$
  (cf.\ \cite[Theorem 13.2.2]{Whitt_02_Stochastic}).
  By Theorem~\ref{2. thm: space M, convergence}, we obtain the desired result.
\end{proof}

To establish convergence of (scaled) local times,
we approximate electrical networks by their traces introduced in Section~\ref{sec: trace of electrical networks}.
Let $(G_{n})_{n \geq 1}$ be the rooted electrical networks,
and $(a_{n})_{n \geq 1}, (b_{n})_{n \geq 1}$ be the scaling factors,
appearing in Assumption~\ref{1. assum: deterministic version}.
We introduce notation for scaled traces of electrical networks.
For each $r > 0$, we let $\tilde{G}_n^{(a_nr)}$ be the trace of $G_n$ onto $B_{R_n}(\rho_n, a_n r)$.
Write 
\begin{equation}
  \tilde{V}_{n}^{(r)} \coloneqq V_{\tilde{G}_{n}^{(a_{n}r)}}, 
  \quad 
  \tilde{R}_{n}^{(r)} \coloneqq a_{n}^{-1} R_{\tilde{G}_{n}^{(a_{n}r)}},
  \quad 
  \tilde{\rho}_{n}^{(r)} \coloneqq \rho_{n},
  \quad
  \tilde{\mu}_{n}^{(r)} \coloneqq b_{n}^{-1} \mu_{\tilde{G}_{n}^{(a_{n}r)}},
\end{equation}
and 
\begin{equation}
  \tilde{\ell}_{n}^{(r)} (x,t)
  \coloneqq
  a_{n}^{-1} \ell_{\tilde{G}_{n}^{(a_{n}r)}} (a_{n} b_{n}t),
\end{equation}
where we recall the local time $\ell_{\tilde{G}_n^{(a_n r)}}$ from \eqref{eq: discrete local time}.
Recall the scaled space $(\hat{V}_{n}, \hat{R}_{n}, \hat{\rho}_{n}, \hat{\mu}_{n})$ from \eqref{1. eq: dfn of scaled spaces}
and the restriction operator $\cdot^{(r)}$ from \eqref{1. eq: dfn of restriction operator}.
We note that 
\begin{equation}
  \tilde{V}_{n}^{(r)} = \hat{V}_{n}^{(r)},
  \quad 
  \tilde{R}_{n}^{(r)} = \hat{R}_{n}^{(r)}, 
  \quad 
  \tilde{\rho}_{n}^{(r)} = \hat{\rho}_{n}^{(r)},
\end{equation}
but $\tilde{\mu}_{n}^{(r)} \neq \hat{\mu}_{n}^{(r)}$ in general.
Below, we prove some technical results.

\begin{lem} \label{5. lem: uniform non-explosion condition}
  If Assumption~\ref{1. assum: deterministic version}\ref{1. assum item: deterministic, convergence of spaces} 
  and \ref{1. assum item: deterministic, the non-explosion condition} are satisfied,
  then it holds that 
  \begin{gather}
    \lim_{r \to \infty}
    \inf_{n \geq 1}
    \inf_{x \in \hat{V}_{n}^{(r_{0})}}
    \hat{R}_{n}(x, B_{\hat{R}_{n}}(x, r)^{c})
    =
    \infty, 
    \quad 
    \forall r_{0} > 0, 
    \label{5. lem eq: uniform non-explosion condition}\\
    \lim_{r \to \infty}
    \sup_{n \geq 1}
    \sup_{x \in \hat{V}_{n}^{(r_{0})}}
    P_{x}^{G_{n}}
    \left(
      T_{B_{\hat{R}_{n}}(x, r)^{c}} \leq a_{n}b_{n}T
    \right)
    =0,
    \quad 
    \forall 
    r_{0}, T >0,
    \label{5. lem eq: uniform non-explosion condition, exit time}\\
    \lim_{r \to \infty} 
    \inf_{n \geq 1}
    \tilde{\mu}_{n}^{(r)}(\tilde{V}_{n}^{(r)})
    > 0.
    \label{5. lem eq: uniform non-explosion condition, measures}
  \end{gather}
\end{lem}

\begin{proof}
  For $x \in \hat{V}_{n}^{(r_{0})}$ and $r > 2 r_{0}$,
  we have that $B_{\hat{R}_{n}}(x, r)^{c} \subseteq B_{\hat{R}_{n}}(\hat{\rho}_{n}, r/2)^{c}$,
  which implies that 
  \begin{equation}  \label{5. eq: uniform non-explosion condition, set comparison}
    \hat{R}_{n}(x, B_{\hat{R}_{n}}(x, r)^{c}) 
    \geq 
    \hat{R}_{n}(x, B_{\hat{R}_{n}}(\hat{\rho}_{n}, r/2)^{c})
  \end{equation} 
  Write $C_{n}(r) \coloneqq B_{\hat{R}_{n}}(\hat{\rho}_{n}, r/2)^{c}$.
  Let $\hat{R}_{n}^{f}$ be the resistance metric 
  obtained by fusing $C_{n}(r)$ into a single vertex
  (see \cite[Theorem 4.3]{Kigami_12_Resistance}).
  In particular, it is a metric on $(\hat{V}_{n} \setminus C_{n}(r)) \cup \{C_{n}(r)\}$
  such that 
  $\hat{R}_{n}^{f}(y, C_{n}(r)) = \hat{R}_{n}(y, C_{n}(r))$ 
  and $\hat{R}_{n}^{f}(y, z) \leq \hat{R}_{n}(y, z)$ for $y,z \in \hat{V}_{n} \setminus C_{n}(r)$.
  Using this metric and \eqref{5. eq: uniform non-explosion condition, set comparison},
  we deduce that 
  \begin{align}
    \hat{R}_{n}(x, B_{\hat{R}_{n}}(x, r)^{c}) 
    &\geq 
    \hat{R}_{n}(x, C_{n}(r)) \\
    &\geq 
    \hat{R}_{n}^{f}(\rho_{n}, C_{n}(r))
    -
    \hat{R}_{n}^{f}(\rho_{n}, x) \\
    &\geq 
    \hat{R}_{n}(\rho_{n}, B_{\hat{R}_{n}}(\hat{\rho}_{n}, r/2)^{c})
    -
    r_{0}.
    \label{5. eq: uniform non-explosion condition, changing root}
  \end{align}
  This, combined with Assumption~\ref{1. assum: deterministic version}\ref{1. assum item: deterministic, the non-explosion condition},
  yields that 
  \begin{equation}  \label{5. eq: uniform non-explosion condition, proto resistance divergence}
    \lim_{r \to \infty}
    \liminf_{n \to \infty}
    \inf_{x \in \hat{V}_{n}^{(r_{0})}}
    \hat{R}_{n}(x, B_{\hat{R}_{n}}(x, r)^{c})
    =
    \infty.
  \end{equation}
  By \eqref{5. eq: uniform non-explosion condition, changing root},
  one can check that, for each $n$,
  it holds that 
  \begin{equation}
    \lim_{r \to \infty}
    \inf_{x \in \hat{V}_{n}^{(r_{0})}}
    \hat{R}_{n}(x, B_{\hat{R}_{n}}(x, r)^{c})
    = 
    \infty.
  \end{equation}
  Hence, it is possible to replace $\liminf_{n \to \infty}$ 
  of \eqref{5. eq: uniform non-explosion condition, proto resistance divergence} 
  by $\inf_{n \geq 1}$,
  and we obtain   
  \eqref{5. lem eq: uniform non-explosion condition}.
  By Lemma~\ref{4. lem: exit time estimate for Y}
  and \eqref{5. lem eq: uniform non-explosion condition},
  we establish \eqref{5. lem eq: uniform non-explosion condition, exit time}.
  From Theorem~\ref{4. thm: expression of trace conductance}, 
  we deduce that, for $r > 1$,
  \begin{align} 
    \tilde{\mu}_{n}^{(r)}(\tilde{V}_{n}^{(r)})
    &\geq
    b_{n}^{-1}
    \mu_{\tilde{G}_{n}^{(a_{n}r)}}(\hat{V}_{n}^{(1)})\\
    &\geq
    b_{n}^{-1}
    \mu_{n}(\hat{V}_{n}^{(1)})
    -
    \sum_{x \in \hat{V}_{n}^{(1)}} 
    b_{n}^{-1}
    \mu_{n}(\{x\}) 
    P_{x}^{G_{n}} 
    \left( 
      Y_{n}(T_{B_{\hat{R}_{n}}(\hat{\rho}_{n}, r)}) = x 
    \right) \\
    &\geq 
    \hat{\mu}_{n}^{(1)}(\hat{V}_{n}^{(1)}) 
    - 
    \hat{\mu}_{n}^{(1)}( \hat{V}_{n}^{(1)})
    \sup_{x \in \hat{V}_{n}^{(1)}}
    P_{x}^{G_{n}} 
    \left(
      T_{B_{\hat{R}_{n}}(\hat{\rho}_{n}, r)^{c}} \leq 1
    \right).
    \label{5. eq: uniform non-explosion condition, volume estimate}
  \end{align}
  Since Assumption~\ref{1. assum: deterministic version}\ref{1. assum item: deterministic, convergence of spaces}
  implies that $0 < \inf_{n \geq 1} \hat{\mu}_{n}^{(1)}(\hat{V}_{n}^{(1)}) 
  \leq \sup_{n \geq 1} \hat{\mu}_{n}^{(1)}(\hat{V}_{n}^{(1)}) < \infty$,
  we obtain \eqref{5. lem eq: uniform non-explosion condition, measures} 
  by using \eqref{5. lem eq: uniform non-explosion condition, exit time}
  and that $a_{n} \wedge b_{n} \to \infty$.
\end{proof}

Recall from \eqref{eq: total conductance at vertex} 
that $c_{\tilde{G}_{n}^{(a_{n}r)}}(x)$ denotes the total conductance at $x$ in the trace $\tilde{G}_n^{(a_n r)}$.
Below, we prove that it converges to the total conductance $c_G(x)$ in the original network as $r \to \infty$
(locally uniformly).

\begin{lem} \label{5. lem: uniform convergence of conductances}
  If Assumption~\ref{1. assum: deterministic version}\ref{1. assum item: deterministic, convergence of spaces} 
  and \ref{1. assum item: deterministic, the non-explosion condition} are satisfied,
  then it holds that 
  \begin{equation}
    \lim_{r \to \infty}
    \limsup_{n \to \infty}
    \sup_{x \in \hat{V}_{n}^{(r_{0})}}
    \left|
      \frac{c_{\tilde{G}_{n}^{(a_{n}r)}}(x)}{c_{G_{n}}(x)} - 1
    \right|
    =0,
    \quad 
    \forall 
    r_{0}>0.
  \end{equation} 
\end{lem}

\begin{proof}
  By Corollary~\ref{4. cor: difference for trace measure at one point},
  we have that 
  \begin{equation}
    \left|
      \frac{c_{\tilde{G}_{n}^{(a_{n}r)}}(x)}{c_{G_{n}}(x)} - 1
    \right|
    =
    P_{x}^{G_{n}}
    \left(
      Y_{n}(T_{B_{\hat{R}_{n}}(\rho_{n}, r)}) = x
    \right)
    \leq 
    P_{x}^{G_{n}}
    \left(
      T_{B_{\hat{R}_{n}}(\rho_{n}, r)^{c}} 
      \leq 
      1
    \right).
  \end{equation}
  For all $r > 2r_{0}$ and $x \in \hat{V}_{n}^{(r_{0})}$,
  it holds that $B_{\hat{R}_{n}}(x, r/2) \subseteq B_{\hat{R}_{n}}(\rho_{n}, r)$.
  Hence, we deduce that, for all $r > 2r_{0}$,
  \begin{equation}  \label{5. eq: conductance convergence, bound by exit times}
    \sup_{x \in \hat{V}_{n}^{(r_{0})}}
    \left|
      \frac{c_{\tilde{G}_{n}^{(a_{n}r)}}(x)}{c_{G_{n}}(x)} - 1
    \right|
    \leq 
    \sup_{x \in \hat{V}_{n}^{(r_{0})}}
    P_{x}^{G_{n}}
    \left(
      T_{B_{\hat{R}_{n}}(x, r/2)^{c}} \leq 1
    \right).
  \end{equation} 
  Now, Lemma~\ref{5. lem eq: uniform non-explosion condition, exit time} immediately yields the desired result.
\end{proof}

Since the traces $\tilde{G}_n^{(a_n r)}$ have finite vertex sets,
we can apply Theorem~\ref{5. thm: uniform continuity estimate of discrete local times} 
and Proposition~\ref{5. prop: uniform bound estimate of discrete local times} 
to obtain the following.
  
\begin{lem} \label{5. lem: precpt of trace local times at same times}
  Under Assumption~\ref{1. assum: deterministic version},
  for all $\varepsilon, T>0$,
  \begin{gather}
    \lim_{r \to \infty}
    \limsup_{\delta \to 0}
    \limsup_{n \to \infty}
    P_{\tilde{\rho}_{n}^{(r)}}^{\tilde{G}_{n}^{(a_{n}r)}}
    \left(
      \sup_{ \stackrel{x, y \in \tilde{V}_{n}^{(r)}}{\tilde{R}_{n}^{(r)}(x,y) < \delta}} 
      \sup_{0 \leq t \leq T}
      \left|
        \tilde{\ell}_{n}^{(r)} (x, t) - \tilde{\ell}_{n}^{(r)} (y, t)
      \right|
      >
      \varepsilon
    \right)
    =0,
    \label{5. lem eq: precpt of trace local times at same times, equicontinuity}\\
    \lim_{r \to \infty}
    \limsup_{M \to \infty}
    \limsup_{n \to \infty}
    P_{\tilde{\rho}_{n}^{(r)}}^{\tilde{G}_{n}^{(a_{n}r)}}
    \left(
      \sup_{x \in \tilde{V}_{n}^{(r)}}
      \tilde{\ell}_{n}^{(r)} (x, T) 
      >
      M
    \right)
    =0.
    \label{5. lem eq: precpt of trace local times at same times, bound}
  \end{gather}
\end{lem}

\begin{proof}
  Under Assumption~\ref{1. assum: deterministic version}\ref{1. assum item: deterministic, convergence of spaces},
  it is an immediate consequence of the convergence in the local Gromov--Hausdorff-vague topology
  that, for all but countably many $r>0$,
  \begin{gather}  \label{5. eq: precompact for local times, convergence of metrics and measures}
    \max_{x, y \in \hat{V}_{n}^{(r)}}
    \hat{R}_{n}^{(r)}(x, y)
    \to 
    \max_{x, y \in F^{(r)}} R^{(r)}(x,y),
    \quad
    \hat{\mu}_{n}^{(r)}(\hat{V}_{n}^{(r)}) \to \mu^{(r)}(F^{(r)}).
  \end{gather}
  This, combined with \eqref{5. lem eq: uniform non-explosion condition, measures},
  yields that, for all sufficiently large $r$,
  there exists a positive constant $L = L(r)$ such that 
  \begin{equation}  \label{5. eq: precompact for local times, same order}
    \begin{split}
    L^{-1}
    < 
    \inf_{n \geq 1} \frac{r(\tilde{G}_{n}^{(a_{n}r)})}{a_{n}}
    \leq 
    \sup_{n \geq 1} \frac{r(\tilde{G}_{n}^{(a_{n}r)})}{a_{n}}
    < 
    L, \\
    \quad 
    L^{-1}
    < 
    \inf_{n \geq 1} \frac{m(\tilde{G}_{n}^{(a_{n}r)})}{b_{n}}
    \leq 
    \sup_{n \geq 1} \frac{m(\tilde{G}_{n}^{(a_{n}r)})}{b_{n}} 
    < 
    L,
    \end{split}
  \end{equation}
  where we recall the notation $r(\cdot)$ and $m(\cdot)$ from \eqref{eq: diameter and total mass}.
  We fix such a radius $r$, and write $m_{n} \coloneqq m (\tilde{G}_{n}^{(a_{n}r)})$ and $r_{n} \coloneqq r (\tilde{G}_{n}^{(a_{n}r)})$.
  By applying Theorem~\ref{5. thm: uniform continuity estimate of discrete local times} 
  to the electrical networks $\tilde{G}_{n}^{(a_{n}r)}$,
  we obtain that
  \begin{align}
    &P_{\tilde{\rho}_{n}^{(r)}}^{\tilde{G}_{n}^{(a_{n}r)}}
    \left(
      \sup_{ \stackrel{x, y \in \tilde{V}_{n}^{(r)}}{\tilde{R}_{n}^{(r)}(x,y) < \delta}} 
      \sup_{0 \leq t \leq T}
      \left|
        \tilde{\ell}_{n}^{(r)} (x, t) - \tilde{\ell}_{n}^{(r)} (y, t)
      \right|
      >
      \varepsilon
    \right) \\
    \leq 
    &P_{\tilde{\rho}_{n}^{(r)}}^{\tilde{G}_{n}^{(a_{n}r)}}
    \left(
      \sup_{ \stackrel{x, y \in \hat{V}_{n}^{(r)}}{R_{n}(x,y) < L r_{n} \delta}} 
      \sup_{0 \leq t \leq L^{2}T m_{n} r_{n} }
      L\,r_{n}^{-1}
      \left|
        \ell_{\tilde{G}_{n}^{(a_{n}r)}} (x, t) - \ell_{\tilde{G}_{n}^{(a_{n}r)}} (y, t)
      \right|
      >
      \varepsilon
    \right) \\
    \leq 
    &c_{4}(L^{2}T)
    \sum_{k \geq N(L, \delta)} 
    (k+1)^{2} 
    N_{\hat{R}_{n}^{(r)}}(\hat{V}_{n}^{(r)}, L^{-1} 2^{-k})^{2} 
    \exp \left(-2^{\alpha(k-3)}\right)
    +
    1_{[\varepsilon, \infty)} (c_{3}(\alpha) 2^{-(\frac{1}{2} - \alpha) N}),
  \end{align}
  where we set $N(L, \delta)$ to be the maximum $N$ satisfying $2^{-N+1} > L \delta$.
  Then,
  Assumption~\ref{1. assum: deterministic version}\ref{1. assum item: deterministic, the metric-entropy condition}
  yields \eqref{5. lem eq: precpt of trace local times at same times, equicontinuity}.
  Similarly,
  using Proposition~\ref{5. prop: uniform bound estimate of discrete local times},
  we obtain \eqref{5. lem eq: precpt of trace local times at same times, bound}.
\end{proof}

In the following two lemmas,
we transfer the results in Lemma~\ref{5. lem: precpt of trace local times at same times}
to the local times of the original Markov chains (rather than the traces).

\begin{lem} \label{5. lem: uniform bound of local times}
  Under Assumption~\ref{1. assum: deterministic version},
  it holds that 
  \begin{equation}
    \lim_{M \to \infty}
    \limsup_{n \to \infty}
    P_{\rho_{n}}^{G_{n}}
    \left(
      \sup_{x \in \hat{V}_{n}^{(r)}}
      \hat{\ell}_{n} (x, T) 
      >
      M
    \right)
    =0,
    \quad 
    \forall r>0.
  \end{equation}
\end{lem}

\begin{proof}
  Fix $r' > r$ arbitrarily.
  By Theorem~\ref{4. thm: coupling by trace of Y_G},
  we may assume that $Y_{\tilde{G}_{n}^{(a_{n}r')}}$ is defined on the same probability space
  as $Y_{G_{n}}$ by setting $Y_{\tilde{G}_{n}^{(a_{n}r')}} \coloneqq \tr_{B_{\hat{R}_{n}}(\hat{\rho}_{n}, r')} Y_{n}$.
  Since we have that 
  $Y_{n}(s) = \tr_{B_{\hat{R}_{n}}(\hat{\rho}_{n}, r')} Y_{n}(s)$
  for all $0 \leq s \leq a_{n}b_{n}T$
  on the event $\{T_{B_{\hat{R}_{n}}(\hat{\rho}_{n}, r')^{c}} > a_{n}b_{n}T\}$,
  we deduce that 
  \begin{align}
    \hat{\ell}_{n}(x, t) 
    &=
    \frac{1}{a_{n}c_{G_{n}}(x)} 
    \int_{0}^{a_{n}b_{n}t} 
    1_{\{x\}}(Y_{n}(s))\,
    ds\\
    &=
    \frac{1}{a_{n}c_{G_{n}}(x)} 
    \int_{0}^{a_{n}b_{n}t} 
    1_{\{x\}}( \tr_{B_{\hat{R}_{n}}(\hat{\rho}_{n}, r')} Y_{n}(s))\,
    ds\\
    &=
    \frac{1}{a_{n}c_{G_{n}}(x)} 
    \int_{0}^{a_{n}b_{n}t} 
    1_{\{x\}}( Y_{\tilde{G}_{n}^{(a_{n}r')}}(s))\,
    ds\\
    &=
    \frac{c_{\tilde{G}_{n}^{(a_{n}r')}}(x)}{c_{G_{n}}(x)} 
    \tilde{\ell}_{n}^{(r')}(x, t).
  \end{align}
  This yields that 
  \begin{align}
    &P_{\rho_{n}}^{G_{n}}
    \left(
      \sup_{x \in \hat{V}_{n}^{(r)}}
      \hat{\ell}_{n} (x, T) 
      >
      M
    \right) \\
    \leq 
    &P_{\rho_{n}}^{G_{n}}
    \left(
      T_{B_{\hat{R}_{n}}(\hat{\rho}_{n}, r')^{c}} \leq a_{n}b_{n}T
    \right)
    +
    P_{\rho_{n}}^{G_{n}}
    \left(
      \sup_{x \in \hat{V}_{n}^{(r)}}
      \frac{c_{\tilde{G}_{n}^{(a_{n}r')}}(x)}{c_{G_{n}}(x)} 
      \tilde{\ell}_{n}^{(r')} (x, T) 
      >
      M
    \right) \\
    \leq 
    &P_{\rho_{n}}^{G_{n}}
    \left(
      T_{B_{\hat{R}_{n}}(\hat{\rho}_{n}, r')^{c}} \leq a_{n}b_{n}T
    \right)
    +
    P_{\rho_{n}}^{G_{n}}
    \left(
      2
      \sup_{x \in \hat{V}_{n}^{(r)}}
      \tilde{\ell}_{n}^{(r')} (x, T) 
      >
      M
    \right) \\
    &\quad 
    +
    1_{[2, \infty)} 
    \left(
      \sup_{x \in \hat{V}_{n}^{(r)}}
      \frac{c_{\tilde{G}_{n}^{(a_{n}r')}}(x)}{c_{G_{n}}(x)} 
    \right).
    \label{5. eq: uniform bound of local times, decomposition}
  \end{align}
  Therefore,
  from Lemma~\ref{5. lem: precompactness for exit times},
  \ref{5. lem: uniform convergence of conductances} 
  and \ref{5. lem: precpt of trace local times at same times},
  we obtain the desired result.
\end{proof}

\begin{lem} \label{5. lem: precpt of local times at same times}
  Under Assumption~\ref{1. assum: deterministic version},
  it holds that 
  \begin{equation}
    \lim_{\delta \to 0}
    \limsup_{n \to \infty}
    P_{\rho_{n}}^{G_{n}}
    \left(
      \sup_{ \stackrel{x, y \in \hat{V}_{n}^{(r)}}{\hat{R}_{n}(x,y) < \delta}} 
      \sup_{0 \leq t \leq T}
      \left|
        \hat{\ell}_{n}(x, t) - \hat{\ell}_{n} (y, t)
      \right|
      >
      \varepsilon
    \right)
    =0,
    \quad 
    \forall 
    r, \varepsilon, T >0.
  \end{equation}
\end{lem}

\begin{proof}
  By the same argument as the proof of Lemma~\ref{5. lem: uniform bound of local times},
  we deduce that 
  \begin{align}
    &P_{\rho_{n}}^{G_{n}}
    \left(
      \sup_{\substack{x, y \in \hat{V}_{n}^{(r)},\\ \hat{R}_{n}(x,y) < \delta}}
      \sup_{0 \leq t \leq T}
      |\hat{\ell}_{n}(x,t) - \hat{\ell}_{n}(y,t)|
      > 
      \varepsilon
    \right)\\
    \leq 
    &P_{\rho_{n}}^{G_{n}}
    \left(
      \sup_{\substack{x, y \in \hat{V}_{n}^{(r)},\\ \hat{R}_{n}(x,y) < \delta}}
      \sup_{0 \leq t \leq T}
      \left|
        \hat{\ell}_{n}(x,t) \wedge \left( \frac{c_{\tilde{G}_{n}^{(a_{n}r')}}(x)}{c_{G_{n}}(x)} M\right)  
        - 
        \hat{\ell}_{n}(y,t) \wedge \left( \frac{c_{\tilde{G}_{n}^{(a_{n}r')}}(y)}{c_{G_{n}}(y)} M\right)
      \right|
      > 
      \varepsilon
    \right) \\
    &\quad 
    +
    P_{\rho_{n}}^{G_{n}}
    \left(
      \sup_{x \in \hat{V}_{n}^{(r)}}
      \frac{c_{G_{n}}(x)}{c_{\tilde{G}_{n}^{(a_{n}r')}}(x)}
      \hat{\ell}_{n}(x, T)
      > M
    \right)\\
    \leq 
    &P_{\rho_{n}}^{G_{n}}
    \left(
      T_{B_{\hat{R}_{n}}(\hat{\rho}_{n}, r')^{c}} \leq a_{n} b_{n} T
    \right)
    +
    P_{\rho_{n}}^{G_{n}}
    \left(
      \sup_{x \in \hat{V}_{n}^{(r)}}
      \frac{c_{G_{n}}(x)}{c_{\tilde{G}_{n}^{(a_{n}r')}}(x)}
      \hat{\ell}_{n}(x, T)
      > M
    \right)\\
    &\quad 
    +
    P_{\tilde{\rho}_{n}^{(r')}}^{\tilde{G}_{n}^{(a_{n}r')}}
    \left(
      \sup_{\substack{x, y \in \hat{V}_{n}^{(r)},\\ \hat{R}_{n}(x,y) < \delta}}
      \sup_{0 \leq t \leq T}
      \left|
        \frac{c_{\tilde{G}_{n}^{(a_{n}r')}}(x)}{c_{G_{n}}(x)} 
        \bigl( \tilde{\ell}_{n}^{(r')}(x, t) \wedge M \bigr)
        -
        \frac{c_{\tilde{G}_{n}^{(a_{n}r')}}(y)}{c_{G_{n}}(y)} 
        \bigl( \tilde{\ell}_{n}^{(r')}(y, t) \wedge M \bigr)
      \right|
      > 
      \varepsilon
    \right).
  \end{align}
  We have that 
  \begin{align}
    &\sup_{\substack{x, y \in \hat{V}_{n}^{(r)},\\ \hat{R}_{n}(x,y) < \delta}}
    \sup_{0 \leq t \leq T}
    \left|
      \frac{c_{\tilde{G}_{n}^{(a_{n}r')}}(x)}{c_{G_{n}}(x)} 
      \bigl( \tilde{\ell}_{n}^{(r')}(x, t) \wedge M \bigr)
      -
      \frac{c_{\tilde{G}_{n}^{(a_{n}r')}}(y)}{c_{G_{n}}(y)} 
      \bigl( \tilde{\ell}_{n}^{(r')}(y, t) \wedge M \bigr)
    \right|\\
    \leq 
    &2 M \sup_{x \in \hat{V}_{n}^{(r)}}
    \left|
      \frac{c_{\tilde{G}_{n}^{(a_{n}r')}}(x)}{c_{G_{n}}(x)}  - 1
    \right|
    +
    \sup_{\substack{x, y \in \hat{V}_{n}^{(r)},\\ \hat{R}_{n}(x,y) < \delta}}
    \sup_{0 \leq t \leq T}
    \left|
      \tilde{\ell}_{n}^{(r')}(x, t)
      -
      \tilde{\ell}_{n}^{(r')}(y, t)
    \right|
  \end{align}
  and 
  \begin{align}
    &P_{\rho_{n}}^{G_{n}}
    \left(
      \sup_{x \in \hat{V}_{n}^{(r)}}
      \frac{c_{G_{n}}(x)}{c_{\tilde{G}_{n}^{(a_{n}r')}}(x)}
      \hat{\ell}_{n}(x, T)
      > M
    \right) \\
    \leq 
    &P_{\rho_{n}}^{G_{n}}
    \left(
      2
      \sup_{x \in \hat{V}_{n}^{(r)}}
      \hat{\ell}_{n}(x, T)
      > M
    \right) 
    +
    1_{[2, \infty)}
    \left(
      \sup_{x \in \hat{V}_{n}^{(r)}}
      \frac{c_{G_{n}}(x)}{c_{\tilde{G}_{n}^{(a_{n}r')}}(x)}
    \right).
  \end{align}
  Hence, it follows that 
  \begin{align}
    &P_{\rho_{n}}^{G_{n}}
    \left(
      \sup_{\substack{x, y \in \hat{V}_{n}^{(r)},\\ \hat{R}_{n}(x,y) < \delta}}
      \sup_{0 \leq t \leq T}
      |\hat{\ell}_{n}(x,t) - \hat{\ell}_{n}(y,t)|
      > 
      \varepsilon
    \right) \\
    \leq
    &P_{\rho_{n}}^{G_{n}}
    \left(
      T_{B_{\hat{R}_{n}}(\hat{\rho}_{n}, r')^{c}} \leq a_{n} b_{n} T
    \right)
    +
    P_{\rho_{n}}^{G_{n}}
    \left(
      2
      \sup_{x \in \hat{V}_{n}^{(r)}}
      \hat{\ell}_{n}(x, T)
      > M
    \right) \\
    &\quad
    +
    1_{[2, \infty)}
    \left(
      \sup_{x \in \hat{V}_{n}^{(r)}}
      \frac{c_{G_{n}}(x)}{c_{\tilde{G}_{n}^{(a_{n}r')}}(x)}
    \right)
    +
    1_{[\varepsilon/2, \infty)} 
    \left( 
      2 M \sup_{x \in \hat{V}_{n}^{(r)}}
      \left|
        \frac{c_{\tilde{G}_{n}^{(a_{n}r')}}(x)}{c_{G_{n}}(x)}  - 1
      \right| 
    \right)\\
    &\quad 
    +
    P_{\tilde{\rho}_{n}^{(r')}}^{\tilde{G}_{n}^{(a_{n}r')}}
    \left(
      \sup_{\substack{x, y \in \hat{V}_{n}^{(r)},\\ \hat{R}_{n}(x,y) < \delta}}
      \sup_{0 \leq t \leq T}
      \left|
        \tilde{\ell}_{n}^{(r')}(x, t)
        -
        \tilde{\ell}_{n}^{(r')}(y, t)
      \right|
      >
      \frac{\varepsilon}{2}
    \right).
    \label{5. eq: precpt at same times, decomposition}
  \end{align}
  By Lemma~\ref{5. lem: precompactness for exit times},
  \ref{5. lem: uniform convergence of conductances},
  \ref{5. lem: precpt of trace local times at same times} 
  and \ref{5. lem: uniform bound of local times},
  we obtain the desired result.
\end{proof}
  
Note that Lemma~\ref{5. lem: precpt of local times at same times}
focuses on the continuity of local times at the same times.
Our next task is to extend the result to different (but close) times.
To this end,
we approximate the local times by mollified ones defined as follows.
Set, for each $\eta > 0$,
\begin{gather}
  f_{\eta}^{\hat{V}_{n}}(x,y)
  \coloneqq
  (\eta - \hat{R}_{n}(x,y)) \vee 0,
  \quad
  x, y \in \hat{V}_{n}.
\end{gather}
We then define 
\begin{gather}
  \hat{\ell}_{n}^{\eta} (x, t)
  \coloneqq
  \frac{ \int_{0}^{t} f_{\eta}^{\hat{V}_{n}}(x, \hat{Y}_{n}(s))\, ds}
  {\int_{\hat{V}_{n}} f_{\eta} (x,y)\, \hat{\mu}_{n}(dy)},
  \quad 
  x \in \hat{V}_{n}.
\end{gather}
which approximate $\hat{\ell}_{n}$.
The convergence of the processes given in Proposition~\ref{5. prop: convergence of processes}
implies the tightness of $(\hat{\ell}_{n}^{\eta})_{n \geq 1}$ as follows.
(Moreover, one can prove that it is a convergent sequence.)

\begin{lem} \label{5. lem: continuity of delta approximation}
  If Assumption~\ref{1. assum: deterministic version} \ref{1. assum item: deterministic, convergence of spaces} 
  is satisfied,
  then it holds that
  \begin{equation}
    \lim_{\delta \to 0}
    \limsup_{n \to \infty}
    P_{\rho_{n}}^{G_{n}}
    \left(
      \sup_{\substack{x, y \in \hat{V}_{n}^{(r)}, \\ \hat{R}_{n}(x,y) < \delta}}
      \sup_{\substack{0 \leq s, t \leq T,\\ |t -s| < \delta}}
      \left|
        \hat{\ell}_{n}^{\eta}(x,t) - \hat{\ell}_{n}^{\eta}(y,s)
      \right|
      > 
      \varepsilon
    \right) 
    =
    0,
    \quad 
    \forall 
    r, \eta>0.
  \end{equation}
\end{lem}

\begin{proof}
  Assume that, for some $m$ and $M$, it holds that 
  \begin{equation}
    m
    <
    \inf_{x \in \hat{V}_{n}^{(r)}}
    \int_{\hat{V}_{n}}
    f_{\eta}^{\hat{V}_{n}}(x, y) \, \hat{\mu}_{n}(dy),
    \quad 
    \hat{\mu}_{n}(\hat{V}_{n}^{(r+1)}) 
    < 
    M
  \end{equation}
  Then, for $x,y \in \hat{V}_{n}$ with $\hat{R}_{n}(x, y) < \delta\ (< 1)$ 
  and $0 \leq s \leq t \leq T$ with $|t -s| < \eta$,
  we have that
  \begin{align}
    | \hat{\ell}_{n}^{\eta}(x, t) - \hat{\ell}_{n}^{\eta}(y,s)|
    &\leq 
    m^{-1} 
    \left|
      \int_{0}^{t} f_{\eta}^{\hat{V}_{n}}(x, \hat{Y}_{n}(u))\, du 
      - 
      \int_{0}^{s} f_{\eta}^{\hat{V}_{n}}(y, \hat{Y}_{n}(u))\, du
    \right|\\
    &\quad 
    + 
    T \eta
    \left|
      \frac{1}{\int_{\hat{V}_{n}} f_{\eta}^{\hat{V}_{n}}(x, z)\, \hat{\mu}_{n}(dz)}
      -
      \frac{1}{\int_{\hat{V}_{n}} f_{\eta}^{\hat{V}_{n}}(y, z)\, \hat{\mu}_{n}(dz)}
    \right| \\
    &\leq 
    m^{-1} (\delta T + \delta \eta)
    +
    T \delta m^{-2} \eta M \\
    &\eqqcolon
    c(\delta, m, M).
  \end{align}
  Hence, 
  we deduce that 
  \begin{align}
    &P_{\rho_{n}}^{G_{n}}
    \left(
      \sup_{\substack{x, y \in \hat{V}_{n}^{(r)}, \\ \hat{R}_{n}(x,y) < \delta}}
      \sup_{\substack{0 \leq s, t \leq T,\\ |t -s| < \delta}}
      \left|
        \hat{\ell}_{n}^{\eta}(x,t) - \hat{\ell}_{n}^{\eta}(y,s)
      \right|
      > 
      \varepsilon
    \right) \\
    \leq 
    &
    1_{(\varepsilon, \infty)} ( c(\delta, m, M) )
    +
    1_{[0, m]}
    \left(
      \inf_{x \in \hat{V}_{n}^{(r)}}
      \int_{\hat{V}_{n}}
      f_{\eta}^{\hat{V}_{n}}(x, y) \, \hat{\mu}_{n}(dy)
    \right) 
    + 
    1_{[M, \infty)}
    \left(
      \hat{\mu}_{n}(\hat{V}_{n}^{(r+1)})
    \right).
    \label{5. eq: continuity of delta approximation, decomposition}
  \end{align}
  Since the convergence of measures $\hat{\mu}_n$ in Assumption~\ref{1. assum: deterministic version}\ref{1. assum item: deterministic, convergence of spaces}
  and the fact that the limiting measure $\mu$ is of full support 
  imply that 
  \begin{gather}
    \liminf_{n \to \infty}
    \inf_{x \in \hat{V}_{n}^{(r)}}
    \int_{\hat{V}_{n}}
    f_{\eta}^{\hat{V}_{n}}(x, y) \, \hat{\mu}_{n}(dy)
    > 0, \\ 
    \limsup_{n \to \infty}
    \hat{\mu}_{n}(\hat{V}_{n}^{(r+1)})
    < \infty,
  \end{gather}
  we obtain the desired result.
\end{proof}

Using the results above, 
we show the tightness of the local times.
Note that, in Lemma~\ref{5. lem: precpt of local times} below,
we consider the equicontinuity of $(\tilde{\ell}_{n}(x, t))_{n \geq 1}$ with respect to $(x,t)$,
while in Lemma~\ref{5. lem: precpt of local times at same times} 
we consider the equicontinuity with respect to $x$ at the same time $t$.

\begin{lem} \label{5. lem: precpt of local times}
  Under Assumption~\ref{1. assum: deterministic version},
  it holds that,
  for all $\varepsilon, r, T>0$,
  \begin{equation}
    \lim_{\delta \to 0}
    \limsup_{n \to \infty}
    P_{\rho_{n}}^{G_{n}}
    \left(
      \sup_{\substack{x, y \in \hat{V}_{n}^{(r)}, \\ \hat{R}_{n}(x,y) < \delta}}
      \sup_{\substack{0 \leq s, t \leq T,\\ |t -s| < \delta}}
      \left|
        \hat{\ell}_{n}(x,t) - \hat{\ell}_{n}(y,s)
      \right|
      > 
      \varepsilon
    \right)
    =
    0.
  \end{equation}
\end{lem}

\begin{proof}
  Fix $\eta \in (0,1)$ arbitrarily.
  By the occupation density formula (see \eqref{1. eq: discrete-time occupation density formula}),
  we have that
  \begin{equation}
    \int_{0}^{t} 
    f_{\eta}^{\hat{V}_{n}} (x, \hat{Y}_{n}(s))\, ds 
    =
    \int_{\hat{V}_{n}} 
    f_{\eta}^{\hat{V}_{n}} (x, y)\,
    \hat{\ell}_{n}(y, t)\,
    \hat{\mu}_{n}(dy).
  \end{equation}
  Using this, we obtain that 
  \begin{equation}
    \sup_{x \in \hat{V}_{n}^{(r)}}
    \sup_{0 \leq t \leq T}
    \left|
      \hat{\ell}_{n}(x, t) - \hat{\ell}_{n}^{\eta}(x,t)
    \right|
    \leq 
    \sup_{\substack{x, y \in \hat{V}_{n}^{(r + 1)}, \\ \hat{R}_{n}(x, y) < \eta}}
    \sup_{0 \leq t \leq T}
    \left|
      \hat{\ell}_{n}(x, t) - \hat{\ell}_{n}(y, t)
    \right|.
  \end{equation}
  We have that
  \begin{align} 
    &P_{\rho_{n}}^{G_{n}}
    \left(
      \sup_{\substack{x, y \in \hat{V}_{n}^{(r)}, \\ \hat{R}_{n}(x,y) < \delta}}
      \sup_{\substack{0 \leq s, t \leq T,\\ |t -s| < \delta}}
      \left|
        \hat{\ell}_{n}(x,t) - \hat{\ell}_{n}(y,s)
      \right|
      > 
      \varepsilon
    \right) \\
    \leq 
    &P_{\rho_{n}}^{G_{n}}
    \left(
      \sup_{\substack{x, y \in \hat{V}_{n}^{(r)}, \\ \hat{R}_{n}(x,y) < \delta}}
      \sup_{\substack{0 \leq s, t \leq T,\\ |t -s| < \delta}}
      \left|
        \hat{\ell}_{n}^{\eta}(x,t) - \hat{\ell}_{n}^{\eta}(y,s)
      \right|
      > 
      \varepsilon/2
    \right) \\
    &\quad
    +
    P_{\rho_{n}}^{G_{n}}
    \left(
      2
      \sup_{x \in \hat{V}_{n}^{(r)}}
      \sup_{0 \leq t \leq T}
      \left|
        \hat{\ell}_{n}^{\eta}(x,t) - \hat{\ell}_{n}(x,t)
      \right|
      > 
      \varepsilon/2
    \right)\\
    \leq 
    &P_{\rho_{n}}^{G_{n}}
    \left(
      \sup_{\substack{x, y \in \hat{V}_{n}^{(r)}, \\ \hat{R}_{n}(x,y) < \delta}}
      \sup_{\substack{0 \leq s, t \leq T,\\ |t -s| < \delta}}
      \left|
        \hat{\ell}_{n}^{\eta}(x,t) - \hat{\ell}_{n}^{\eta}(y,s)
      \right|
      > 
      \varepsilon/2
    \right) \\
    &\quad 
    + 
    P_{\rho_{n}}^{G_{n}}
    \left(
      2
      \sup_{\substack{x, y \in \hat{V}_{n}^{(r + 1)}, \\ \hat{R}_{n}(x, y) < \eta}}
      \sup_{0 \leq t \leq T}
      \left|
        \hat{\ell}_{n}(x, t) - \hat{\ell}_{n}(y, t)
      \right|
      > 
      \varepsilon/2
    \right).
    \label{5. eq: precpt of local times, decomposition}
  \end{align}
  Hence,
  by \eqref{5. eq: precpt of local times, decomposition},
  Lemmas~\ref{5. lem: precpt of local times at same times} and~\ref{5. lem: continuity of delta approximation},
  we obtain the desired result.
\end{proof}

Now it is straightforward to complete the proof of Theorem~\ref{1. thm: deterministic, main result}.
Indeed, since Lemma~\ref{5. lem: precpt of local times} implies the tightness of the local times,
it remains to show the uniqueness of the limit of any convergent subsequence.
This can be done exactly along the same lines as the proof of \cite[Theorem~1.9]{Noda_pre_Convergence}.

\begin{proof} [Proof of the second part of Theorem~\ref{1. thm: deterministic, main result}]
  Since we can now follow the proof of \cite[Theorem~1.9]{Noda_pre_Convergence},
  we only give a sketch.
  We first check that $(\cY_{n})_{n \geq 1}$ satisfies all the conditions 
  \ref{2. thm item: precpt in M_L, spaces are precompact}-\ref{2. thm item: precpt in M_L, equicontinuity of local times}
  of Theorem~\ref{2. thm: precpt in M_L}.
  By Theorem~\ref{2. thm: precompactness in M} 
  and Proposition~\ref{5. prop: convergence of processes},
  we obtain 
  \ref{2. thm item: precpt in M_L, spaces are precompact},
  \ref{2. thm item: precpt in M_L, values of processes are precompact} 
  and \ref{2. thm item: precpt in M_L, equicontinuity of processes}.
  The condition \ref{2. thm item: precpt in M_L, equicontinuity of local times}
  follows from Lemma~\ref{5. lem: precpt of local times}.
  Since we have $\hat{\ell}_{n}(x,0)=0$ for all $x \in \hat{V}_{n}$,
  we obtain \ref{2. thm item: precpt in M_L, boundedness of local times}.
  Hence, the sequence $(\cY_{n})_{n \geq 1}$ is precompact in $\mbM_{L}$.
  Suppose that a subsequence $(\cY_{n_{k}})_{k \geq 1}$ converges to 
  some $\cX$ in $\mbM_{L}$.
  Then, by Assumption~\ref{1. assum: deterministic version}\ref{1. assum item: deterministic, convergence of spaces},
  we can write $\cX = (F, R, \rho, \mu, \pi)$.
  Let $(X, L)$ be a random element of $D(\RNp, F) \times \hatC(F \times \RNp, \RN)$
  whose law coincides with $\pi$.
  Proposition~\ref{5. prop: convergence of processes} implies that $X \stackrel{\mathrm{d}}{=} X_{G}$.
  Since the discrete local times satisfy the occupation density formula~\eqref{1. eq: discrete-time occupation density formula},
  we deduce that $L$ also satisfies the occupation density formula~\eqref{3. eq: the occupation density formula},
  which implies that $(X, L) \stackrel{\mathrm{d}}{=} (X_{G}, L_{G})$,
  as shown in the proof of \cite[Theorem~1.9]{Noda_pre_Convergence}.
  This completes the proof.
\end{proof}

%%%%%%%%%%%%%%%%%%%%%%%%%%%%%%%%%%%%%%%%%%%%%%%%%%%%%%%%%%%%%%%%%%%%%%%%%%%%%%%%%%%%%%%%%%%%%%%%%%%%%%%%%%%%%%%%%%%%%%%%%%%%%%%%%%
%%%%%%%%%%%%%%%%%%%%%%%%%%%%%%%%%%%%%%%%%%%%%%%%%%%%%%%%%%%%%%%%%%%%%%%%%%%%%%%%%%%%%%%%%%%%%%%%%%%%%%%%%%%%%%%%%%%%%%%%%%%%%%%%%%
% Proof of main result 2
%%%%%%%%%%%%%%%%%%%%%%%%%%%%%%%%%%%%%%%%%%%%%%%%%%%%%%%%%%%%%%%%%%%%%%%%%%%%%%%%%%%%%%%%%%%%%%%%%%%%%%%%%%%%%%%%%%%%%%%%%%%%%%%%%%
%%%%%%%%%%%%%%%%%%%%%%%%%%%%%%%%%%%%%%%%%%%%%%%%%%%%%%%%%%%%%%%%%%%%%%%%%%%%%%%%%%%%%%%%%%%%%%%%%%%%%%%%%%%%%%%%%%%%%%%%%%%%%%%%%%

\section{Proof of Theorem~\ref{1. thm: random, main result}}  \label{sec: proof of the main results 2}

In this section, we prove Theorem~\ref{1. thm: random, main result},
which is a version of Theorem~\ref{1. thm: deterministic, main result} for random electrical networks.
The first part of Theorem~\ref{1. thm: random, main result} is immediately obtained from the following result,
which is an analogue of Theorem~\ref{5. thm: recurrence is preserved by the non-explosion condition} 
for random resistance metric spaces.

\begin{prop} \label{6. prop: random recurrence is preserved by the non-explosion condition}
  For each $n \geq 1$,
  let $(F_{n}, R_{n}, \rho_{n})$
  be a random element of $\mathbb{D}$ with underlying complete probability measure $P_{n}$
  such that $R_{n}$ is a resistance metric with probability $1$.
  Assume that there exists a random element $(F, R, \rho)$ of $\mathbb{D}$
  with underlying complete probability measure $P$
  such that 
  $(F_{n}, R_{n}, \rho_{n})$ converges to $(F, R, \rho)$ 
  in distribution
  in the local Gromov--Hausdorff topology.
  Then $R$ is a resistance metric with probability $1$.
  Moreover,
  if it holds that 
  \begin{equation} \label{6. prop eq: random preserving recurrence, effective resistance comparison}
  \lim_{r \to \infty}
  \limsup_{n \to \infty}
  P_{n}
  \left(
    R_{n}(\rho_{n}, B_{R_{n}}(\rho_{n}, r)^{c}) > \lambda
  \right)
  =1,
  \quad
  \forall 
  \lambda > 0,
  \end{equation}
  then $(F, R)$ is recurrent with probability $1$.
\end{prop}

\begin{proof}
  Since the space $\mathbb{D}$ is separable,
  we can use the Skorohod representation theorem 
  and assume that $(F_{n}, R_{n}, \rho_{n}) \to (F, R, \rho)$
  almost surely on some complete probability space.
  We denote the underlying probability measure by $Q$.
  It follows from Theorem~\ref{5. thm: recurrence is preserved by the non-explosion condition}
  that $R$ is a resistance metric almost surely.
  Moreover, 
  by the theorem and (reverse) Fatou's lemma,
  we deduce that 
  \begin{align}
    P 
    \left(
      \lim_{r \to \infty} 
      R(\rho, B_{R}(\rho, r)^{c}) 
      =
      \infty
    \right)
    &=
    \lim_{\lambda \to \infty}
    P 
    \left(
      \lim_{r \to \infty} 
      R(\rho, B_{R}(\rho, r)^{c}) 
      >
      \lambda
    \right) \\ 
    &\geq 
    \lim_{\lambda \to \infty}
    Q 
    \left(
      \lim_{r \to \infty} 
      \limsup_{n \to \infty}
      R_{n}(\rho_{n}, B_{R_{n}}(\rho_{n}, r)^{c})
      >
      \lambda
    \right) \\
    &\geq 
    \lim_{\lambda \to \infty}
    \limsup_{r \to \infty}
    \limsup_{n \to \infty}
    P
    \left(
      R_{n}(\rho_{n}, B_{R_{n}}(\rho_{n}, r)^{c})
      >
      \lambda
    \right) \\
    &=
    1.
  \end{align}
\end{proof}

The convergence of (scaled) discrete-time Markov chains and their local times on random electrical networks
is proven in a similar way to the deterministic setting.
Henceforth, we use the same notation introduced in Section~\ref{sec: proof of main result 1},
and verify analogues of several results in that section for random electrical networks.
The following corresponds to Proposition~\ref{5. prop: convergence of processes}.

\begin{prop}  \label{6. prop: convergence of processes}
  Under Assumption~\ref{1. assum: random version}\ref{1. assum item: random, convergence of spaces} 
  and \ref{1. assum item: random, the non-explosion condition},
  it holds that 
  \begin{equation}
    \left(
      \hat{V}_{n}, \hat{R}_{n}, \hat{\rho}_{n}, \hat{\mu}_{n}, P_{\rho_{n}}^{G_{n}}(\hat{Y}_{n} \in \cdot)
    \right)
    \xrightarrow{\mathrm{d}}
    \left(
      F, R, \rho, \mu, P_{\rho}^{G}(X_{G} \in \cdot)
    \right)
  \end{equation}
  as random elements of $\mbM$.
\end{prop}

\begin{proof}
  Set $\hat{X}_{n}(t) \coloneqq X_{G_{n}}(a_{n}b_{n}t)$,
  which is the continuos-time Markov chain associated with $(\hat{V}_{n}, \hat{R}_{n}, \hat{\mu}_{n})$.
  By \cite[Proposition 6.7]{Noda_pre_Convergence},
  we have that 
  \begin{equation}
    \left(
      \hat{V}_{n}, \hat{R}_{n}, \hat{\rho}_{n}, \hat{\mu}_{n}, P_{\rho_{n}}^{G_{n}}(\hat{X}_{n} \in \cdot)
    \right)
    \xrightarrow{\mathrm{d}}
    \left(
      F, R, \rho, \mu, P_{\rho}^{G}(X_{G} \in \cdot)
    \right)
  \end{equation}
  as random elements of $\mbM$.
  Using the Skorohod representation theorem,
  we may assume that the above convergence holds almost surely 
  on some probability space.
  Then, by the same argument as the proof of Proposition~\ref{5. prop: convergence of processes},
  we deduce that it holds that, with probability $1$,
  \begin{equation}
    \left(
      \hat{V}_{n}, \hat{R}_{n}, \hat{\rho}_{n}, \hat{\mu}_{n}, P_{\rho_{n}}^{G_{n}}(\hat{Y}_{n} \in \cdot)
    \right)
    \to
    \left(
      F, R, \rho, \mu, P_{\rho}^{G}(X_{G} \in \cdot)
    \right),
  \end{equation}
  which completes the proof.
\end{proof}

Below,
we prove a version of Lemma~\ref{5. lem: uniform non-explosion condition} for random electrical networks.

\begin{lem} \label{6. lem: uniform non-explosion condition}
  Under Assumption~\ref{1. assum: random version}\ref{1. assum item: random, convergence of spaces} 
  and \ref{1. assum item: random, the non-explosion condition},
  it holds that 
  \begin{gather}
    \lim_{r \to \infty}
    \liminf_{n \to \infty}
    \mathbf{P}_{n}\!
    \left(
      \inf_{x \in \hat{V}_{n}^{(r_{0})}}
      \hat{R}_{n}(x, B_{\hat{R}_{n}}(x, r)^{c})
      >
      \lambda
    \right)
    =1
    , 
    \quad 
    \forall r_{0}, \lambda > 0, 
    \label{6. lem eq: uniform non-explosion condition}\\
    \lim_{r \to \infty}
    \limsup_{n \to \infty}
    \mathbf{P}_{n}\!
    \left(
      \sup_{x \in \hat{V}_{n}^{(r_{0})}}
      P_{x}^{G_{n}}
      \left(
        T_{B_{\hat{R}_{n}}(x, r)^{c}} \leq a_{n}b_{n}T
      \right)
      >
      \varepsilon
    \right)
    =0,
    \quad 
    \forall 
    r_{0}, T, \varepsilon >0,
    \label{6. lem eq: uniform non-explosion condition, exit time}\\
    \lim_{r \to \infty}
    \limsup_{\varepsilon \to 0}
    \limsup_{n \to \infty}
    \mathbf{P}_{n}\!
    \left(
      \tilde{\mu}_{n}^{(r)}(\tilde{V}_{n}^{(r)}) < \varepsilon
    \right)
    =
    0.
    \label{6. lem eq: uniform non-explosion condition, lower bound measures}
  \end{gather}
\end{lem}

\begin{proof}
  By \eqref{5. eq: uniform non-explosion condition, changing root},
  we have that, for all $r > 2r_{0}$,
  \begin{equation}
    \mathbf{P}_{n}\!
    \left(
      \inf_{x \in \hat{V}_{n}^{(r_{0})}}
      \hat{R}_{n}(x, B_{\hat{R}_{n}}(x, r)^{c})
      >
      \lambda
    \right)
    \geq 
    \mathbf{P}_{n}\!
    \left(
      \hat{R}_{n}(\hat{\rho}_{n}, B_{\hat{R}_{n}}(\hat{\rho}_{n}, r/2)^{c})
      >
      \lambda + r_{0}
    \right).
  \end{equation}
  This, combined with Assumption~\ref{1. assum: random version}\ref{1. assum item: random, the non-explosion condition},
  immediately yields \eqref{6. lem eq: uniform non-explosion condition}.
  From Lemma~\ref{4. lem: exit time estimate for Y} and \eqref{6. lem eq: uniform non-explosion condition},
  we deduce \eqref{6. lem eq: uniform non-explosion condition, exit time}.
  Using \eqref{5. eq: uniform non-explosion condition, volume estimate}, 
  we deduce that, for all $r>1$,
  \begin{align} 
    \mathbf{P}_{n}\!
    \left(
      \tilde{\mu}_{n}^{(r)}(\tilde{V}_{n}^{(r)}) < \varepsilon
    \right)
    &\leq 
    \mathbf{P}_{n}\!
    \left(
      \hat{\mu}_{n}^{(1)}(\hat{V}_{n}^{(1)})
      -
      \hat{\mu}_{n}^{(1)}(\hat{V}_{n}^{(1)})
      \sup_{x \in \hat{V}_{n}^{(1)}}
      P_{x}^{G_{n}}
      \left(
        T_{B_{\hat{R}_{n}}(\hat{\rho}_{n}, r)^{c}} \leq 1
      \right)
      < \varepsilon
    \right) \\
    &\leq 
    \mathbf{P}_{n}\!
    \left(
      \frac{1}{2}
      <
      \sup_{x \in \hat{V}_{n}^{(1)}}
      P_{x}^{G_{n}}
      \left(
        T_{B_{\hat{R}_{n}}(\hat{\rho}_{n}, r)^{c}} \leq 1
      \right)
    \right)
    +
    \mathbf{P}_{n}\!
    \left(
      \hat{\mu}_{n}^{(1)}(\hat{V}_{n}^{(1)}) < 2 \varepsilon
    \right)
    \label{6. eq: uniform non-explosion condition, measure estimate}
  \end{align}
  Since Assumption~\ref{1. assum: random version}\ref{1. assum item: random, convergence of spaces} implies that 
  \begin{equation}  \label{6. eq: uniform non-explosion condition, meaures with radius 1}
    \lim_{\varepsilon \to 0}
    \limsup_{n \to \infty}
    \mathbf{P}_{n}\!
    \left(
      \hat{\mu}_{n}^{(1)}(\hat{V}_{n}^{(1)}) < 2 \varepsilon
    \right)
    =
    0,
  \end{equation}
  we obtain \eqref{6. lem eq: uniform non-explosion condition, lower bound measures}
  by \eqref{6. lem eq: uniform non-explosion condition, exit time}, 
  \eqref{6. eq: uniform non-explosion condition, measure estimate}
  and \eqref{6. eq: uniform non-explosion condition, meaures with radius 1}.
\end{proof}

The following two lemmas correspond to Lemmas~\ref{5. lem: uniform convergence of conductances} and \ref{5. lem: precpt of trace local times at same times}.

\begin{lem} \label{6. lem: uniform convergence of conductances}
  If Assumption~\ref{1. assum: random version}\ref{1. assum item: random, convergence of spaces} 
  and \ref{1. assum item: random, the non-explosion condition} are satisfied,
  then it holds that 
  \begin{equation}
    \lim_{r \to \infty}
    \limsup_{n \to \infty}
    \mathbf{P}_{n}
    \left(
      \sup_{x \in \hat{V}_{n}^{(r_{0})}}
      \left|
        \frac{c_{\tilde{G}_{n}^{(a_{n}r)}}(x)}{c_{G_{n}}(x)} - 1
      \right|
      > 
      \varepsilon
    \right)
    =0,
    \quad 
    \forall 
    r_{0}, \varepsilon>0.
  \end{equation} 
\end{lem}

\begin{proof}
  This is an immediate consequence of \eqref{5. eq: conductance convergence, bound by exit times}
  and \eqref{6. lem eq: uniform non-explosion condition, exit time}.
\end{proof}

\begin{lem} \label{6. lem: precpt of trace local times at same times}
  Under Assumption~\ref{1. assum: deterministic version},
  for all $\varepsilon, \varepsilon_{1}, \varepsilon_{2}, T, \eta>0$,
  \begin{gather}
    \lim_{r \to \infty}
    \limsup_{\delta \to 0}
    \limsup_{n \to \infty}
    \mathbf{P}_{n}\!
    \left(
      P_{\tilde{\rho}_{n}^{(r)}}^{\tilde{G}_{n}^{(a_{n}r)}}\!
      \left(
        \sup_{ \stackrel{x, y \in \tilde{V}_{n}^{(r)}}{\tilde{R}_{n}^{(r)}(x,y) < \delta}} 
        \sup_{0 \leq t \leq T}
        \left|
          \tilde{\ell}_{n}^{(r)} (x, t) - \tilde{\ell}_{n}^{(r)} (y, t)
        \right|
        >
        \varepsilon_{1}
      \right)
      > 
      \varepsilon_{2}
    \right)
    =0,
    \label{6. lem eq: precpt of trace local times at same times, equicontinuity}\\
    \lim_{r \to \infty}
    \limsup_{M \to \infty}
    \limsup_{n \to \infty}
    \mathbf{P}_{n}\!
    \left(
      P_{\tilde{\rho}_{n}^{(r)}}^{\tilde{G}_{n}^{(a_{n}r)}}\!
      \left(
        \sup_{x \in \tilde{V}_{n}^{(r)}}
        \tilde{\ell}_{n}^{(r)} (x, T) 
        >
        M
      \right)
      >
      \varepsilon
    \right)
    =0.
    \label{6. lem eq: precpt of trace local times at same times, bound}
  \end{gather}
\end{lem}

\begin{proof}
  By \eqref{6. lem eq: uniform non-explosion condition, lower bound measures} and 
  Assumption~\ref{1. assum: random version}\ref{1. assum item: random, convergence of spaces},
  we have that 
  \begin{gather}
    \lim_{r \to \infty}
    \liminf_{L \to \infty}
    \liminf_{n \to \infty}
    \mathbf{P}_{n}\!
    \left(
      L^{-1} 
      < 
      \frac{r(\tilde{G}_{n}^{(a_{n}r)})}{a_{n}} 
      < 
      L
    \right) 
    =1, \\
    \lim_{r \to \infty}
    \liminf_{L \to \infty}
    \liminf_{n \to \infty}
    \mathbf{P}_{n}\!
    \left(
      L^{-1} 
      < 
      \frac{m(\tilde{G}_{n}^{(a_{n}r)})}{b_{n}} 
      < 
      L
    \right) 
    =1.
  \end{gather}
  Therefore,
  by using Assumption~\ref{1. assum: random version}\ref{1. assum item: random, the metric-entropy condition} 
  and following the proof of Lemma~\ref{5. lem: precpt of trace local times at same times},
  we obtain the desired result.
\end{proof}

From the above results,
we establish an analogue of Lemmas~\ref{5. lem: uniform bound of local times} and \ref{5. lem: precpt of local times at same times}.

\begin{lem} \label{6. lem: precpt of local times at same times}
  Under Assumption~\ref{1. assum: deterministic version},
  it holds that, for all $r, \varepsilon, \varepsilon_{1}, \varepsilon_{2}, \eta >0$,
  \begin{gather}
    \lim_{M \to \infty}
    \limsup_{n \to \infty}
    \mathbf{P}_{n}\!
    \left(
      P_{\rho_{n}}^{G_{n}}\!
      \left(
        \sup_{x \in \hat{V}_{n}^{(r)}}
        \hat{\ell}_{n} (x, T) 
        >
        M
      \right)
      >
      \varepsilon
    \right)
    =0, 
    \label{6. lem eq: precpt of local times, upper bound} \\
    \lim_{\delta \to 0}
    \limsup_{n \to \infty}
    \mathbf{P}_{n}\!
    \left(
      P_{\rho_{n}}^{G_{n}}
      \left(
        \sup_{\substack{x, y \in \hat{V}_{n}^{(r)}, \\ \hat{R}_{n}(x,y) < \delta}}
        \sup_{\substack{0 \leq t \leq T}}
        \left|
          \hat{\ell}_{n}(x,t) - \hat{\ell}_{n}(y,t)
        \right|
        > 
        \varepsilon_{1}
      \right)
      > \varepsilon_{2}
    \right)
    =
    0.
    \label{6. lem eq: precpt of local times at same times, equicontinuity}
  \end{gather}
\end{lem}

\begin{proof}
  It is easy to establish \eqref{6. lem eq: precpt of local times, upper bound}
  by using \eqref{5. eq: uniform bound of local times, decomposition},
  \eqref{6. lem eq: uniform non-explosion condition, exit time},
  \eqref{6. lem eq: precpt of trace local times at same times, bound}
  and 
  Lemma~\ref{6. lem: uniform convergence of conductances}.
  We then obtain \eqref{6. lem eq: precpt of local times at same times, equicontinuity} 
  from \eqref{5. eq: precpt at same times, decomposition},
  \eqref{6. lem eq: uniform non-explosion condition, exit time},
  \eqref{6. lem eq: precpt of trace local times at same times, equicontinuity},
  \eqref{6. lem eq: precpt of local times, upper bound}
  and Lemma~\ref{6. lem: uniform convergence of conductances}.
\end{proof}

The following two lemmas correspond to Lemmas~\ref{5. lem: continuity of delta approximation} and~\ref{5. lem: precpt of local times}.

\begin{lem} \label{6. lem: continuity of delta approximation}
  If Assumption~\ref{1. assum: random version} \ref{1. assum item: random, convergence of spaces} 
  is satisfied,
  then it holds that,
  for all $r, \varepsilon_{1}, \varepsilon_{2}, \eta >0$,
  \begin{equation}
    \lim_{\delta \to 0}
    \limsup_{n \to \infty}
    \mathbf{P}_{n}\!
    \left(
      P_{\rho_{n}}^{G_{n}}
      \left(
        \sup_{\substack{x, y \in \hat{V}_{n}^{(r)}, \\ \hat{R}_{n}(x,y) < \delta}}
        \sup_{\substack{0 \leq s, t \leq T,\\ |t -s| < \delta}}
        \left|
          \hat{\ell}_{n}^{\eta}(x,t) - \hat{\ell}_{n}^{\eta}(y,s)
        \right|
        > 
        \varepsilon_{1}
      \right) 
      >
      \varepsilon_{2}
    \right)
    =
    0.
  \end{equation}
\end{lem}

\begin{proof}
  Assumption~\ref{1. assum: random version}\ref{1. assum item: random, convergence of spaces} implies that 
  \begin{gather}
    \lim_{m \to 0}
    \liminf_{n \to \infty}
    \mathbf{P}_{n}\!
    \left(
      \inf_{x \in \hat{V}_{n}^{(r)}} 
      \int_{\hat{V}_{n}^{(r)}} 
      f_{\eta}^{\hat{V}_{n}}(x, y)\,
      \hat{\mu}_{n}(dy)
      >
      m
    \right)
    = 
    1,\\
    \lim_{M \to \infty}
    \limsup_{n \to \infty}
    \mathbf{P}_{n}\!
    \left(
      \hat{\mu}_{n}(\hat{V}_{n}^{(r+1)}) > M
    \right)
    =
    0.
  \end{gather}
  These, combined with \eqref{5. eq: continuity of delta approximation, decomposition}, yields the desired result.
\end{proof}

\begin{lem} \label{6. lem: precpt of local times}
  Under Assumption~\ref{1. assum: random version},
  it holds that,
  for all $\varepsilon_{1}, \varepsilon_{2}, r, T>0$,
  \begin{equation}
    \lim_{\delta \to 0}
    \limsup_{n \to \infty}
    \mathbf{P}_{n}\!
    \left(
      P_{\rho_{n}}^{G_{n}}
      \left(
        \sup_{\substack{x, y \in \hat{V}_{n}^{(r)}, \\ \hat{R}_{n}(x,y) < \delta}}
        \sup_{\substack{0 \leq s, t \leq T,\\ |t -s| < \delta}}
        \left|
          \hat{\ell}_{n}(x,t) - \hat{\ell}_{n}(y,s)
        \right|
        > 
        \varepsilon_{1}
      \right)
      >
      \varepsilon_{2}
    \right)
    =
    0.
  \end{equation}
\end{lem}

\begin{proof}
  This is an immediate consequence of \eqref{5. eq: precpt of local times, decomposition},
  Lemma~\ref{6. lem: precpt of local times at same times} 
  and Lemma~\ref{6. lem: continuity of delta approximation}.
\end{proof}

Now, it is possible to complete the proof of Theorem~\ref{1. thm: random, main result}.

\begin{proof} [{Proof of the second part of Theorem~\ref{1. thm: random, main result}}]
  Since we can now follow the proof of \cite[Theorem~1.11]{Noda_pre_Convergence},
  we only give a sketch.
  By Theorem~\ref{2. thm: tightness in M_L}, 
  \ref{6. prop: convergence of processes}
  and 
  \ref{6. lem: precpt of local times},
  we deduce that the sequence of the random elements $(\hat{\cY}_{n} )_{n \geq 1}$ of $\mbM_{L}$ is tight.
  Thus, 
  it remains to show that the limit of any weakly-convergent subsequence of 
  $(\hat{\cY}_{n})_{n \geq 1}$ is $\cX_{G}$.
  To simplify subscripts,
  we suppose that $\hat{\cY}_{n}$ converges to $\cX$ 
  as random elements of $\mbM_{L}$,
  and show that $\cX \stackrel{\mathrm{d}}{=} \cX_{G}$.
  Write $\cX = (F', R', \rho', \mu', \pi')$.
  Note that $G' \coloneqq (F', R', \rho', \mu') \stackrel{\mathrm{d}}{=} G$.
  By the Skorohod representation theorem,
  we may assume that $\cX_{G_{n}}$ converges to $\cX$ 
  almost surely on some probability space.
  By the same argument as the proof of Theorem~\ref{1. thm: deterministic, main result},
  we obtain that $\pi' = \hat{P}_{n}$, which completes the proof.
\end{proof}

%%%%%%%%%%%%%%%%%%%%%%%%%%%%%%%%%%%%%%%%%%%%%%%%%%%%%%%%%%%%%%%%%%%%%%%%%%%%%%%%%%%%%%%%%%%%%%%%%%%%%%%%%%%%%%%%%%%%%%%%%%%%%%%%%%
%%%%%%%%%%%%%%%%%%%%%%%%%%%%%%%%%%%%%%%%%%%%%%%%%%%%%%%%%%%%%%%%%%%%%%%%%%%%%%%%%%%%%%%%%%%%%%%%%%%%%%%%%%%%%%%%%%%%%%%%%%%%%%%%%%
% Appendix
%%%%%%%%%%%%%%%%%%%%%%%%%%%%%%%%%%%%%%%%%%%%%%%%%%%%%%%%%%%%%%%%%%%%%%%%%%%%%%%%%%%%%%%%%%%%%%%%%%%%%%%%%%%%%%%%%%%%%%%%%%%%%%%%%%
%%%%%%%%%%%%%%%%%%%%%%%%%%%%%%%%%%%%%%%%%%%%%%%%%%%%%%%%%%%%%%%%%%%%%%%%%%%%%%%%%%%%%%%%%%%%%%%%%%%%%%%%%%%%%%%%%%%%%%%%%%%%%%%%%%

\appendix

%%%%%%%%%%%%%%%%%%%%%%%%%%%%%%%%%%%%%%%%%%%%%%%%%%%%%%%%%%%%%%%%%%%%%%%%%%%%%%%%%%%%%%%%%%%%%%%%%%%%%%%%%%%%%%%%%%%%%%%%%%%%%%%%%%
% Applications
%%%%%%%%%%%%%%%%%%%%%%%%%%%%%%%%%%%%%%%%%%%%%%%%%%%%%%%%%%%%%%%%%%%%%%%%%%%%%%%%%%%%%%%%%%%%%%%%%%%%%%%%%%%%%%%%%%%%%%%%%%%%%%%%%%

\section{Applications}
\label{A. sec: application}

%%%%%%%%%%%%%%%%%%%%%%%%%%%%%%%%%%%%%%%%%%%%%%%%%%%%%%%%%%%%%%%%%%%%%%%%%%%%%%%%%%%%%%%%%%%%%%%%%%%%%%%%%%%%%%%%%%%%%%%%%%%%%%%%%%
% Applications
%%%%%%%%%%%%%%%%%%%%%%%%%%%%%%%%%%%%%%%%%%%%%%%%%%%%%%%%%%%%%%%%%%%%%%%%%%%%%%%%%%%%%%%%%%%%%%%%%%%%%%%%%%%%%%%%%%%%%%%%%%%%%%%%%%

\subsection{Examples in the previous paper}

In this appendix,
we explain why we can apply our main results (Theorem~\ref{1. thm: deterministic, main result} and \ref{1. thm: random, main result})
to the examples of \cite{Noda_pre_Convergence}.
For details of those examples,
we refer to \cite[Section 8]{Noda_pre_Convergence}.

We assume convergence of conductance measures for our main results 
(Assumption~\ref{1. assum: deterministic version}\ref{1. assum item: deterministic, convergence of spaces} and 
\ref{1. assum: random version}\ref{1. assum item: random, convergence of spaces}).
However, 
in \cite[Section~8]{Noda_pre_Convergence},
it is convergence of counting measures 
that is proven for critical Galton--Watson trees conditioned on size, 
uniform spanning trees
on $\ZN^d$ with $d = 2, 3$ and on high-dimensional tori, 
the critical Erd\H{o}s--R\'{e}nyi random graph,
and the configuration model.
To deduce the convergence of conductance measures for these examples from that of counting measures,
we use the following result regarding the Prohorov distance between counting measures and conductance measures on trees. 

\begin{prop}  \label{A. prop: distance between counting measures and conductance measures}
  Fix a graph tree $(V, E)$.
  Regard this as an electrical network by placing conductance $1$ on each edge.
  Equip $V$ with the associated resistance metric $d$.
  (Note that $d$ coincides with the graph metric).
  Write $\mu^{\#}$ for the counting measure and $\mu$ for the conductance measure.
  Then it holds that 
  \begin{equation}
    d_{P}(2\mu^{\#}, \mu) \leq 2,
  \end{equation}
  where $d_{P}$ denotes the Prohorov metric.
\end{prop}

\begin{proof}
  Choose a vertex $\rho \in V$ and regard the tree $(V, E)$ as a rooted tree with the root $\rho$.
  For each vertex $v$,
  we write $d_{v}$ for the degree of $v$ and $\mathsf{C}(v)$ for the set of children of $v$.
  Fix a subset $A$ of $V$
  and define $A^{(1)}$ to be the subset of $V$ consisting of $w$ such that $d(v, w) \leq 1$ for some $v \in A$.
  Noting that any non-root vertex has a unique parent,
  we deduce that 
  \begin{align}
    \mu(A) 
    &= 
    \sum_{v \in A} d_{v} \\
    &\leq
    \sum_{v \in A} 
    \left(
      1 + \sum_{w \in \mathsf{C}(v)} 1
    \right) \\
    &=
    \sum_{v \in A} 1 + \sum_{v \in A} \sum_{w \in \mathsf{C}(v)} 1 \\
    &\leq 
    \sum_{v \in A} 1 + \sum_{w \in A^{(1)}} 1 \\
    &\leq 
    2 \mu^{\#}(A^{(1)}),
  \end{align}
  and 
  \begin{align}
    2 \mu^{\#}(A) 
    &= 
    \sum_{v \in A} 2 \\
    &= 
    \sum_{v \in A \setminus \{\rho\}} 1 
    + 
    \sum_{w \in V} \sum_{v \in \mathsf{C}(w) \cap A} 1
    +
    2 \cdot 1_{\{\rho \in A\}} \\
    &\leq 
    \sum_{v \in A \setminus \{\rho\}} 1 
    + 
    \sum_{w \in A^{(1)} \setminus \{\rho\}} (d_{w} - 1)
    +
    d_{\rho} \cdot 1_{\{\rho \in A^{(1)}\}}
    +
    2 \cdot 1_{\{\rho \in A\}} \\
    &\leq 
    \mu(A^{(1)}) + 2.
  \end{align}
  Hence, the desired result follows.
\end{proof}

By Proposition~\ref{A. prop: distance between counting measures and conductance measures},
when scaled counting measures of trees converge to a measure $\mu$,
conductance measures scaled by the same scaling factors converge to $2\mu$
(assuming the relevant scaling factors diverge).
Hence, we can apply our results to the critical Galton--Watson trees and uniform spanning trees considered in 
\cite[Sections~8.2, 8.3, and~8.4]{Noda_pre_Convergence}.
In particular, the convergence results of \cite[Corollaries~8.11, 8.16, 8.21, and~8.28]{Noda_pre_Convergence}
still hold if we replace the scaled continuous-time Markov chains and their local times appearing there
with the discrete-time Markov chains and their local times.

As for the Erd\H{o}s--R\'{e}nyi random graph and the configuration model,
though they are not trees,
it is known that 
sequences of these graphs have the same asymptotic behavior (in terms of measured metric spaces) 
as certain sequences of random graphs obtained by fusing random vertices 
in random tilted trees.
Indeed, such an observation was used in \cite{Berry_Broutin_Goldschmidt_12_The_continuum,Bhamidi_Sen_20_Geometry}
to establish the convergence of Erd\H{o}s--R\'{e}nyi random graphs and configuration models equipped with the graph metrics.
See also \cite{Noda_pre_Metrization}, which proved the convergence with respect to the resistance metrics.
By Proposition~\ref{A. prop: distance between counting measures and conductance measures}
and the convergence of the counting measures of the tilted trees (see \cite[Proofs of Lemma 8.42 and Proposition 8.69]{Noda_pre_Convergence}),
one obtains the convergence of the conductance measures of the tilted trees.
Thus,
a slight modification of \cite[Proofs of Theorem 8.41 and 8.51]{Noda_pre_Convergence} 
enables us to deduce the convergence of the conductance measures 
for critical the Erd\H{o}s--R\'{e}nyi random graph and the configuration model.
Namely, the convergence results of \cite[Corollary~8.48 and Theorem~8.51]{Noda_pre_Convergence}
still hold if we replace the scaled continuous-time Markov chains and their local times appearing there
with the discrete-time Markov chains and their local times.

In \cite[Section~8.5]{Noda_pre_Convergence},
a random recursive Sierpi\'{n}ski gasket $G$ was also considered
and it was proved that 
if $G_{n}$ is the $n$-level electrical network approximating $G$,
then $G_{n}$ converges to $G$ as measured resistance metric spaces.
There,
the conductance measures were normalized so that the total mass is equal to $1$,
and 
the existence of deterministic scaling factors for the (non-normalized) conductance measures 
(i.e., the sequence $(b_{n})_{n \geq 1}$ 
satisfying Assumption~\ref{1. assum: deterministic version}\ref{A. assum item: deterministic, space convergence})
was not proved.
However, 
in \cite{Hambly_97_Brownian},
where a random recursive Sierpi\'{n}ski gasket was first constructed,
the existence of such scaling factors was implicitly proved; 
alternatively, one can deduce it
as a consequence of standard limit theorems for general branching processes 
(see \cite[Theorem~5.4]{Nerman_81_On} for example).
In \cite[Sections~2 and 3]{Hambly_97_Brownian}, 
the random recursive Sierpi\'{n}ski gasket $G$ and appropriating graphs $G_n$
are associated with a certain general branching process.
Using the Malthusian parameter $\kappa$ for this branching process
(see \cite[Equation~(1,4)]{Nerman_81_On} and \cite[Section~3]{Hambly_97_Brownian}),
the scaling factor for the conductance measure of $G_{n}$ is given by $b_{n} \coloneqq e^{(\kappa+1)n}$.
A brief heuristic explanation of $b_{n}$ is as follows.
At time $n$,
there are of order $e^{\kappa n}$ individuals alive in the branching process,
which means that $G_{n}$ consists of order $e^{\kappa n}$ triangles.
Since the order of the conductances on each triangle on $G_{n}$ is $e^{n}$,
the order of the total mass given by the conductance measure on $G_{n}$ is $e^{(\kappa+1)n}$.

%%%%%%%%%%%%%%%%%%%%%%%%%%%%%%%%%%%%%%%%%%%%%%%%%%%%%%%%%%%%%%%%%%%%%%%%%%%%%%%%%%%%%%%%%%%%%%%%%%%%%%%%%%%%%%%%%%%%%%%%%%%%%%%%%%
% Application to the random conductance model
%%%%%%%%%%%%%%%%%%%%%%%%%%%%%%%%%%%%%%%%%%%%%%%%%%%%%%%%%%%%%%%%%%%%%%%%%%%%%%%%%%%%%%%%%%%%%%%%%%%%%%%%%%%%%%%%%%%%%%%%%%%%%%%%%%

\subsection{The random conductance model}

In this appendix, we apply our main results to the random conductance model on unbounded fractals.
In the study of scaling limits of random graphs, 
graph metrics are often employed. 
In our assumptions,
we consider convergence with respect to resistance metrics, 
which is harder to check in general. 
(If a graph is a tree, then the graph metric and the resistance metric coincide, 
but otherwise, 
the resistance metric is smaller than the graph metric.) 
Before going into applications, 
we first establish a general method 
for obtaining the local Gromov--Hausdorff convergence 
for a sequence of graphs approximating a non-compact fractal (Theorem~\ref{A. thm: convergence of fractals}). 
We then briefly describe an application of it to the random conductance model on the unbounded Sierpi\'{n}ski gasket.

We first clarify the setting for the main result.
For a specific example,
see Example \ref{A. exm: unbounded Sierpinski gasket}. 
Fix a set $W$ and an element $\rho \in W$,
which serves as the root.
Let $(V_{n})_{n \geq 1}$ and $(K^{(N)})_{N \geq 1}$ be increasing sequences of subsets of $W$
such that each $V_{n}$ is finite or countable,
$\rho \in V_{n}$ and $\rho \in K^{(N)}$ for each $n$ and $N$,
and $V \coloneqq \bigcup_{n \geq 1} V_{n} \subseteq \bigcup_{N \geq 1} K^{(N)}$.
The set $V_{n}$ will be the vertex set of a level $n$ graph
and $(K_{N})_{N \geq 1}$ will be an alternative for closed balls with radius $N$
in the resistance metric spaces we will consider.
Write $V_{n}^{(N)} \coloneqq V_{n} \cap K^{(N)}$.
We assume that $V_{n}^{(N)}$ is a finite set for each $n$ and $N$.
Suppose that we have an electrical network $G_{n} = (V_{n}, E_{n}, c_{n})$ with root $\rho_{n} \coloneqq \rho$.
Let $R_{n}$ and $(\cE_{n}, \cF_{n})$ 
be the associated resistance metric and resistance form, respectively. 
For $n \geq m$,
we simply write $R_{n}|_{m} \coloneqq R_{n}|_{V_{m} \times V_{m}}$,
which is the resistance metric associated with the trace $G_{n}|_{V_{m}}$
(recall the trace from Section~\ref{sec: trace of electrical networks}).
Let $\{c_{n}|_{m}(x, y) \mid x, y \in V_{m}\}$
be the conductance set of $G_{n}|_{V_{m}}$.
In particular, 
$R_{n}|_{m}$ is the resistance metric on $V_{m}$ determined by the conductance $c_{n}|_{m}$.
We then define 
\begin{equation}
  R_{n}|_{m}^{(N)}(x, y)
  \coloneqq 
  \sup\{
    \cE_{n}(u, u)^{-1} \mid u \in \cF_{n},\, u(x)=1,\, u(y)=0,\, 
    u \ \text{is constant on}\  V_{m} \setminus K^{(N)}
  \},
\end{equation}
which is the fused resistance metric on $V_{m}^{(N)} \cup \{V_{m} \setminus K^{(N)}\}$
(here, $V_{m} \setminus K^{(N)}$ is regarded as a single vertex).
See, \cite[Theorem~4.3]{Kigami_12_Resistance} for details on fused resistance metrics.
Note that we first take a trace onto $V_{m}$ and then fuse the outside of $K^{(N)}$,
which is different from fusing first and then taking a trace.
It is easy to check that 
if we define conductance $c_{n}|_{m}^{(N)}$ on $V_{m}^{(N)} \cup \{V_{m} \setminus K^{(N)}\}$
by setting 
\begin{gather}
  c_{n}|_{m}^{(N)}(x, y) 
  \coloneqq 
  c_{n}(x,y),
  \quad x, y \in V_{m}^{(N)},\\
  c_{n}|_{m}^{(N)}(x, V_{m} \setminus K^{(N)})
  \coloneqq 
  \sum_{y \in V_{m} \setminus K^{(N)}} 
  c_{n}|_{m}(x, y),
  \quad 
  x \in V_{m}^{(N)},
\end{gather}
then $R_{n}|_{m}^{(N)}$ is the resistance metric determined by the conductance $c_{n}|_{m}^{(N)}$.

Let $G|_{n} = (V_{n}, E|_{n}, c|_{n})$ be another electrical network, 
which does not necessarily coincides with $G_{n}$ introduced above,
and $R|_{n}$ be the associated resistance metric on $V_{n}$.
If the sequence $(G|_{n})_{n \geq 1}$ is compatible, 
i.e., $R|_{m} = R|_{n}|_{V_{m} \times V_{m}}$ for any $n \geq m$,
then we obtain a resistance metric $R$ on $V = \bigcup_{n \geq 1} V_{n}$
by setting $R|_{V_{m} \times V_{m}} = R|_{m}$.
We then define $(F, R)$ to be the completion of $(V, R)$.
We note that $(F, R)$ is also a resistance metric space (see \cite[Theorem 3.13]{Kigami_12_Resistance}).
Let $(\cE, \cF)$ be the resistance form on $F$ corresponding to $R$.
Set 
\begin{equation}
  R|_{m}^{(N)}(x, y)
  \coloneqq 
  \sup\{
    \cE(u, u)^{-1} \mid u \in \cF,\, u(x)=1,\, u(y)=0,\, 
    u\ \text{is constant on}\ V_{m} \setminus K^{(N)} 
  \}.
\end{equation}
Again,
we note that the conductance $c|_{m}^{(N)}$ on $V_{m}^{(N)} \cup \{V_{m} \setminus K^{(N)}\}$
given by 
\begin{gather}
  c|_{m}^{(N)}(x, y) 
  \coloneqq 
  c|_{m}(x,y),
  \quad x, y \in V_{m}^{(N)},\\
  c|_{m}^{(N)}(x, V_{m} \setminus K^{(N)})
  \coloneqq 
  \sum_{y \in V_{m} \setminus K^{(N)}} 
  c|_{m}(x, y),
  \quad 
  x \in V_{m}^{(N)}
\end{gather}
yields the resistance metric $R|_{m}^{(N)}$ on $V_{m}^{(N)} \cup \{V_{m} \setminus K^{(N)}\}$.
We then consider the following conditions.
Note that 
we define $V^{(N)} \coloneqq V \cap K^{(N)}$, similarly to $V_{n}^{(N)}$.

\begin{assum} \label{A. assum: general condition for approximation}
  \leavevmode
  \begin{enumerate} [label = \textup{(\roman*)}]
    \item \label{A. assum item: approximation, neighborhood function}
      For each $m \geq 1$, there exists a map $g_{m} \colon V \to V_{m}$ satisfying
      \begin{gather}
        g_{m}(V^{(N)}) \subseteq V_{m}^{(N)}, 
        \quad 
        g_{m}(V \setminus K^{(N)}) \subseteq V_{m} \setminus K^{(N)},\\
        \lim_{m \to \infty}
        \limsup_{n \to \infty} 
        \sup_{x \in V_{n}}
        R_{n}(x, g_{m}(x))
        =0, 
        \quad
        \lim_{m \to \infty}
        \sup_{x \in V}
        R(x, g_{m}(x)) 
        =0.
      \end{gather}
    \item \label{A. assum item: approximation, bound for V_m} 
      For each $m \geq 1$ and $N \geq 1$, 
      $\displaystyle \limsup_{n \to \infty} \sup_{x \in V_{m}^{(N)}} R_{n}(\rho_{n}, x) < \infty$.
    \item \label{A. assum item: approximation, non-explosion}
      It holds that 
      $\displaystyle \lim_{N \to \infty} \limsup_{n \to \infty} R_{n}(\rho_{n}, V_{n} \setminus K^{(N)}) = \infty$.
    \item \label{A. assum item: approximation, conductance convergence}
      For each $x, y \in V_{m}^{(N)} \cup \{V_{m} \setminus K^{(N)}\}$,
      $c_{n}|_{m}^{(N)}(x,y) \to c|_{m}^{(N)}(x, y)$ as $n \to \infty$.
  \end{enumerate}
\end{assum}

In Assumption~\ref{A. assum: general condition for approximation},
condition~\ref{A. assum item: approximation, neighborhood function} means 
that the sequence $(V_{m})_{m \geq 1}$ converges to $V$ uniformly
in a suitable sense,
condition~\ref{A. assum item: approximation, bound for V_m}, 
combined with \ref{A. assum item: approximation, neighborhood function},
implies that the diameters of $(V_{n}^{(N)})_{n \geq 1}$ are uniformly bounded for each $N$
(see Lemma~\ref{A. lem: the diameters are uniformly bounded}),
condition~\ref{A. assum item: approximation, non-explosion} is the non-explosion condition 
with respect to increasing subsets $K_N$ rather than balls,
and condition~\ref{A. assum item: approximation, conductance convergence} is equivalent to 
that the condition that the fused resistance metric $R_{n}|_{m}^{(N)}$ on $V_{m}^{(N)} \cup \{V_{m} \setminus K^{(N)}\}$ 
converges to $R|_{m}^{(N)}$
(by Lemma~\ref{4. lem: continuity of resistance metric and conductance}).
This assumption leads to the convergence of $(V_{n}, R_{n}, \rho_{n})$ to $(F, R, \rho)$
in the local Gromov--Hausdorff topology.

\begin{thm} \label{A. thm: convergence of fractals}
  Under Assumption~\ref{A. assum: general condition for approximation},
  $(F, R)$ is a boundedly-compact metric space and the resistance metric $R$ is recurrent.
  Moreover,
  it holds that 
  \begin{equation}  \label{A. thm eq: convergence of fractals}
    (V_{n}, R_{n}, \rho_{n}) \to (F, R, \rho)
  \end{equation}
  in the local Gromov--Hausdorff topology
  and the non-explosion condition is satisfied:
  \begin{equation}  \label{A. thm eq: convergence of fractals, non-explosion}
    \lim_{r \to \infty} 
    \liminf_{n \to \infty}
    R_{n}(\rho_{n}, B_{R_{n}}(\rho_{n}, r)^{c})
    = 
    \infty.
  \end{equation}
\end{thm}

\begin{exm} \label{A. exm: unbounded Sierpinski gasket}
  Since it is better to have an example in mind before moving to the proof of Theorem~\ref{A. thm: convergence of fractals}, 
  we describe how the above-mentioned setting is applied to the unbounded Sierpi\'{n}ski gasket.
  Set $W \coloneqq \RN^{2}$ and $\rho \coloneqq (0, 0)$.
  For convergence, we count indices from $0$.
  Let $\hat{V}_{0} \coloneqq \{x_{1}, x_{2}, x_{3} \} \subseteq \RN^{2}$
  consist of the vertices of an equilateral triangle of side length $1$ with $x_{1} = \rho$.
  Define $K^{(0)} \subseteq \RN^{2}$
  to be the compact subset 
  consisting of the boundary and interior of the triangle.
  Write $\psi_{i}(x) \coloneqq (x + x_{i})/2$ for $i = 1, 2, 3$.
  We then set 
  $\hat{V}_{n} \coloneqq \bigcup_{i=1}^{3} \psi_{i}(\hat{V}_{n-1})$.
  Set $V_{0} \coloneqq \bigcup_{k = 0}^{\infty} 2^{k} \hat{V}_{k}$,
  $V_{n} \coloneqq 2^{-n} V_{0}$ for $n \geq 1$
  and 
  $K^{(N)} \coloneqq 2^{N} K^{(0)}$ for $N \geq 1$.
  The edge set $E_{n} = E|_{n}$ on $V_{n}$ is the set of pairs of elements of $V_{n}$
  at a Euclidean distance $2^{-n}$ apart.
  Set $c|_{n}(x,y) \coloneqq (5/3)^{n}$ if $\{x,y\} \in E_{n}$.
  Then, we obtain a sequence of compatible electrical networks $G|_{n} = (V_{n}, E|_{n}, c|_{n})$,
  see Figure \ref{A. figure: the unbounded SG graphs}.
  The resulting resistance metric space $(F, R)$ is the unbounded Sierpi\'{n}ski gasket.
  If we set $G_{n} = G|_{n}$,
  then it is easy to check that Assumption~\ref{A. assum: general condition for approximation} is satisfied
  for this example.
  (For each $x \in V$, by choosing appropriately a vertex of a triangle on $V_{m}$
  that contains $x$ inside or on the boundary,
  one can construct the map $g_{m}$ 
  in Assumption~\ref{A. assum: general condition for approximation}\ref{A. assum item: approximation, neighborhood function},
  see Figure \ref{A. figure: the unbounded SG graphs}.)
\end{exm}

\begin{figure}[t]
  \centering
  \includegraphics[scale=0.78]{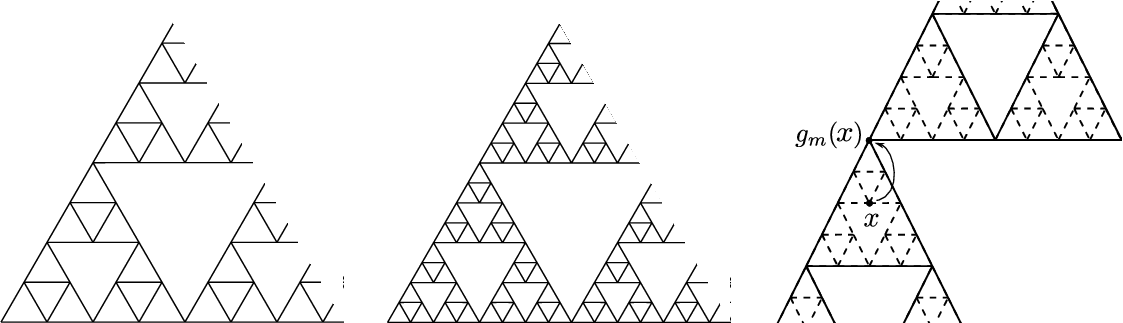}
  \caption{From Left to right,
  the unbounded Sierpi\'{n}ski gasket graphs $G_{0}$, $G_{1}$ and the associated map $g_{m}$.}
  \label{A. figure: the unbounded SG graphs}
\end{figure}

Towards proving Theorem~\ref{A. thm: convergence of fractals},
we start with proving the recurrence of $R$
and that $(F, R)$ is boundedly compact.
Henceforth, we suppose that Assumption~\ref{A. assum: general condition for approximation} is satisfied.

\begin{lem} \label{A. lem: the diameters are uniformly bounded}
  For each $N \geq 1$,
  it holds that 
  \begin{equation}  \label{A. lem eq: diameter bound}
    \limsup_{n \to \infty} \sup_{x \in V_{n}^{(N)}} R_{n}(\rho_{n}, x) < \infty,
    \quad 
    \sup_{x \in V^{(N)}} R(\rho, x) < \infty.
  \end{equation}
\end{lem}

\begin{proof}
  The triangle inequality yields that 
  \begin{equation}
    \sup_{x \in V_{n}^{(N)}}R_{n}(\rho_{n}, x) 
    \leq 
    \sup_{y \in V_{m}^{(N)}}
    R_{n}(\rho_{n}, y) 
    + 
    \sup_{x \in V_{n}^{(N)}}
    R_{n}(g_{m}(x), x).
  \end{equation} 
  From Assumption~\ref{A. assum: general condition for approximation}\ref{A. assum item: approximation, neighborhood function}
  and \ref{A. assum item: approximation, bound for V_m},
  we obtain the first inequality of \eqref{A. lem eq: diameter bound}.
  Since $V_{m}^{(N)}$ is assumed to be a finite set,
  we have that 
  $\sup_{x \in V_{m}^{(N)}} R(\rho, x) < \infty$.
  Therefore,
  following a similar argument,
  we obtain the second inequality of \eqref{A. lem eq: diameter bound}.
\end{proof}

\begin{lem} \label{A. lem: recurrence of limiting space}
  The resistance metric $R$ is recurrent.
\end{lem}

\begin{proof}
  Assumption~\ref{A. assum: general condition for approximation}\ref{A. assum item: approximation, conductance convergence} implies that 
  $R_{n}(\rho_{n}, V_{m} \setminus K^{(N)}) \to R(\rho, V_{m} \setminus K^{(N)})$ as $n \to \infty$.
  If $n \geq m$,
  then $R_{n}(\rho_{n}, V_{n} \setminus K^{(N)}) \leq R_{n}(\rho_{n}, V_{m} \setminus K^{(N)})$.
  Thus, we obtain that 
  \begin{equation}
    \lim_{m \to \infty} R(\rho, V_{m} \setminus K^{(N)})
    =
    \lim_{m \to \infty} \lim_{n \to \infty} R_{n}(\rho_{n}, V_{m} \setminus K^{(N)})
    \geq 
    \liminf_{n \to \infty} R_{n}(\rho_{n}, V_{n} \setminus K^{(N)}).
  \end{equation}
  This, combined with Assumption~\ref{A. assum: general condition for approximation}\ref{A. assum item: approximation, non-explosion},
  yields that 
  \begin{equation}  \label{A. eq: recurrence, quasi explosion}
    \lim_{N \to \infty} 
    \lim_{m \to \infty}
    R(\rho, V_{m} \setminus K^{(N)}) 
    = 
    \infty.
  \end{equation}
  Fix $M > 0$.
  By \eqref{A. eq: recurrence, quasi explosion} 
  and Assumption~\ref{A. assum: general condition for approximation}\ref{A. assum item: approximation, neighborhood function},
  there exist $N \geq 1$ and $m \geq 1$ such that 
  $R(\rho, V_{m} \setminus K^{(N)}) > M$ 
  and $\sup_{x \in V} R(x, g_{m}(x)) < M/4$.
  Let $u \in \cF$ be the unique function satisfying 
  $u(\rho) = 0$, $u|_{V_{m} \setminus K^{(N)}} \equiv 0$ 
  and $R(\rho, V_{m} \setminus K^{(N)}) = \cE(u, u)^{-1}$.
  For $x \in V \setminus K^{(N)}$,
  since we have that $g_{m}(x) \in V_{m} \setminus K^{(N)}$,
  it follows that $u(g_{m}(x)) = 0$.
  Hence,
  \eqref{3. eq: the basic inequality for resistance forms} yields that 
  \begin{equation}
    |u(x)| 
    = 
    |u(x) - u(g_{m}(x))|
    \leq 
    \sqrt{\cE(u, u) R(x, g_{m}(x))}
    \leq 
    1/2
  \end{equation}
  for all $x \in V \setminus K^{(N)}$.
  Define $v \coloneqq 2 \{(u - 2^{-1}) \vee 0 \}$.
  Then,
  $v \in \cF$,
  $v(\rho) = 1$ 
  and $v|_{V \setminus K^{(N)}} \equiv 0$.
  Moreover,
  \ref{3. dfn item: Markov property of resistance forms} yields that 
  $\cE(v, v) \leq 4 \cE(u, u)$.
  By the continuity of $v$,
  we have that $v|_{\cl(V \setminus K^{(N)})} \equiv 0$,
  where we recall that, for a set $A \subseteq F$,
  $\cl(A)$ denotes the closure of a set $A$ in $(F, R)$.
  Hence, it follows that 
  \begin{equation}
    R(\rho, \cl(V \setminus K^{(N)}))
    \geq 
    4^{-1} R(\rho, V_{m} \setminus K^{(N)})
    > 
    4^{-1} M.
  \end{equation}
  By Lemma~\ref{A. lem: the diameters are uniformly bounded},
  there exists $r>0$ such that $\sup_{x \in V^{(N)}} R(\rho, x) < r/2$.
  It is then the case that $B_{R}(\rho, r)^{c} \subseteq \cl(V \setminus K^{(N)})$.
  Thus,
  we deduce that 
  \begin{equation}
    R(\rho, B_{R}(\rho, r)^{c}) 
    > 4^{-1} M,
  \end{equation}
  which completes the proof.
\end{proof}

\begin{lem}
  The resistance metric space $(F, R)$ is boundedly compact.
\end{lem}

\begin{proof}
  Fix $r>0$.
  By Lemma~\ref{A. lem: recurrence of limiting space},
  we can find $N \geq 1$ 
  such that 
  $R(\rho, \cl(V \setminus K^{(N)})) > r$,
  which implies that $D_{R}(\rho, r) \subseteq F \setminus \cl(V \setminus K^{(N)})$.
  For $x \in F \setminus \cl(V \setminus K^{(N)})$,
  let $(x_{n})_{n \geq 1}$ be a sequence in $V$ converging to $x$
  (recall that $\cl(V) = F$).
  If there exists a subsequence $(x_{n_{k}})_{k \geq 1}$ such that $x_{n_{k}} \in V \setminus K^{(N)}$,
  then $x \in \cl(V \setminus K^{(N)})$,
  which is a contradiction.
  Hence, 
  it follows that 
  $F \setminus \cl(V \setminus K^{(N)}) \subseteq \cl(V^{(N)})$
  and $D_{R}(\rho, r) \subseteq \cl(V^{(N)})$.
  Fix $\varepsilon > 0$.
  By Assumption~\ref{A. assum: general condition for approximation}\ref{A. assum item: approximation, neighborhood function},
  there exists $m \geq 1$ such that 
  $\sup_{x \in V} R(x, g_{m}(x)) < \varepsilon$.
  Since $g_{m}(V^{(N)}) \subseteq V_{m}^{(N)}$,
  we deduce that 
  \begin{equation}
    D_{R}(\rho, r)
    \subseteq
    V^{(N)}
    \subseteq 
    \bigcup_{x \in V_{m}^{(N)}} D_{R}(x, \varepsilon).
  \end{equation}
  Recalling that $V_{m}^{(N)}$ is assumed to be a finite set,
  we obtain that $D_{R}(\rho, r)$ is totally bounded,
  which implies that $D_{R}(\rho, r)$ is compact.
\end{proof}

The key to obtaining the desired convergence is 
that the non-explosion condition ensures that the restricted metric $R_{n}|_{m}$ 
is approximated by the fused metric $R_{n}|_{m}^{(N)}$ (see Lemma~\ref{A. lem: approximation by fused metrics}).
To show this,
we use the following basic property of resistance forms
proven in \cite{Kajino_18_Introduction},
which is a modification of \cite[Lemma 6.5]{Kigami_12_Resistance}.

\begin{prop}  [{\cite[Corollary 2.39]{Kajino_18_Introduction}}] \label{A. prop: energy of product of functions}
  Let $(\cE, \cF)$ be a resistance form on $F$.
  Then, 
  for any bounded functions $u, v \in \cF$,
  $u \cdot v \in \cF$ and 
  \begin{equation}
    \cE(u \cdot v, u \cdot v)^{\frac{1}{2}}
    \leq 
    \|u\|_{\infty} \cE(v, v)^{\frac{1}{2}}
    +
    \|v\|_{\infty} \cE(u, u)^{\frac{1}{2}}.
  \end{equation}
\end{prop}

Using the above result,
we obtain a quantitative estimate on the difference between a resistance metric and a fused one.

\begin{cor} \label{A. cor: error estimate of fused metric}
  Let $(\cE, \cF)$ be a resistance form on $F$.
  We equip $F$ with the topology induced from the corresponding resistance metric.
  Fix a non-empty open set $B \subseteq F$ 
  and define, for $x, y \in B$,
  \begin{equation}
    R^{(B)}(x,y) 
    \coloneqq 
    \sup\{
      \cE(u, u )^{-1} \mid u(x)=1,\, u(y)=0,\, u\ \text{is constant on}\ B^{c}
    \},
  \end{equation}
  i.e., $R^{(B)}$ is the resistance metric obtained by fusing $B$ into a single point.
  It then holds that, for all $x, y \in B$,
  \begin{equation}
    |R(x,y) - R^{(B)}(x,y)|
    \leq 
    2 R(x, B^{c})^{-\frac{1}{2}} R(x,y)^{\frac{3}{2}}.
  \end{equation}
\end{cor}

\begin{proof}
  Let $u, v \in \cF$ be the functions such that 
  \begin{gather}
    u(x) = 1,\quad u(y) = 0,\quad R(x,y) = \cE(u, u)^{-1},\\
    v(x) = 1, \quad v|_{B^{c}} \equiv 0,\quad R(x, B^{c}) = \cE(v, v)^{-1}.
  \end{gather}
  Noting that $\|u\|_{\infty} \vee \|v\|_{\infty} \leq 1$,
  we obtain from Proposition~\ref{A. prop: energy of product of functions} that 
  \begin{equation}
    \cE(u \cdot v, u \cdot v) 
    \leq 
    \left(
      \cE(u, u)^{\frac{1}{2}} + \cE(v, v)^{\frac{1}{2}}
    \right)^{2}
    =
    \left(
      R(x, y)^{-\frac{1}{2}} + R(x, B^{c})^{-\frac{1}{2}}
    \right)^{2}.
  \end{equation}
  Since $(u \cdot v)(x) = 1$, $(u \cdot v) (y) = 0$ and $(u \cdot v) |_{B^{c}} \equiv 0$, 
  we deduce that 
  \begin{equation}
    R^{(B)}(x, y) 
    \geq 
    \left(
      R(x, y)^{-\frac{1}{2}} + R(x, B^{c})^{-\frac{1}{2}}
    \right)^{2}
  \end{equation}
  This yields that 
  \begin{equation}
    R(x, y) - R^{(B)}(x,y)
    \leq 
    \frac{2 R(x, B^{c})^{-\frac{1}{2}} R(x,y)}
      {R(x,y)^{-\frac{1}{2}} + R(x, B^{c})^{-\frac{1}{2}}}
    \leq 
    2 R(x, B^{c})^{-\frac{1}{2}} R(x,y)^{\frac{3}{2}}.
  \end{equation}
  From the above inequality and the fact that $R^{(B)}(x,y) \leq R(x,y)$,
  we obtain the desired result.
\end{proof}

We deduce from Corollary~\ref{A. cor: error estimate of fused metric} that 
the fused resistance metric $R_n|_m^{(N)}$ approximates the original resistance metric,
as follows.

\begin{lem} \label{A. lem: approximation by fused metrics}
  For each $m, N_{0} \geq 1$,
  it holds that 
  \begin{gather}
    \lim_{N \to \infty}
    \limsup_{n \to \infty}
    \sup_{x, y \in V_{m}^{(N_{0})}} 
    \left|
      R_{n}|_{m}^{(N)}(x,y) - R_{n}(x,y)
    \right|
    =0, \\
    \lim_{N \to \infty}
    \sup_{x, y \in V_{m}^{(N_{0})}}
    \left|
      R|_{m}^{(N)}(x,y) - R(x,y)
    \right|
    =0.
  \end{gather} 
\end{lem}

\begin{proof}
  By Lemma~\ref{A. lem: the diameters are uniformly bounded},
  we can follow the argument in the proof of Lemma~\ref{5. lem: uniform non-explosion condition} 
  to obtain,
  for each $N_{0} \geq 1$,
  \begin{equation}  \label{A. eq: approx by fused, uniform non-explosion}
    \lim_{N \to \infty} 
    \liminf_{n \to \infty}
    \inf_{x \in V_{n}^{(N_{0})}}
    R_{n}(x, V_{n} \setminus K^{(N)})
    = \infty, 
    \quad 
    \lim_{N \to \infty}
    \inf_{x \in V^{(N_{0})}}
    R(x, \cl(V \setminus K^{(N)}))
    = \infty.
  \end{equation}
  By \eqref{A. eq: approx by fused, uniform non-explosion},
  Lemma~\ref{A. lem: the diameters are uniformly bounded} 
  and Corollary~\ref{A. cor: error estimate of fused metric},
  we deduce
  the desired result.
\end{proof}

The above approximation enables us to obtain the convergence of resistance metrics.

\begin{prop}  \label{A. prop: uniform convergence of R_n to R}
  For each $N_{0} \geq 1$,
  it holds that 
  \begin{equation}
    \lim_{n \to \infty}
    \sup_{x, y \in V_{n}^{(N_{0})}}
    |R_{n}(x,y) - R(x,y)|
    = 0.
  \end{equation}
\end{prop}

\begin{proof}
  By the triangle inequality,
  we deduce that
  \begin{align}
    \sup_{x, y \in V_{n}^{(N_{0})}}
    |R_{n}(x,y) - R(x,y)|
    &\leq 
    2 \sup_{z \in V_{n}^{(N_{0})}} R_{n}(z, g_{m}(z))
    +
    2 \sup_{z \in V^{(N_{0})}} R(z, g_{m}(z)) \\
    &\quad
    +
    \sup_{x, y \in V_{n}^{(N_{0})}} 
    |R_{n}(g_{m}(x), g_{m}(y)) - R(g_{m}(x), g_{m}(y))|.
  \end{align}
  Since we have Assumption~\ref{A. assum: general condition for approximation}\ref{A. assum item: approximation, neighborhood function},
  it is enough to show that 
  \begin{equation}
    \lim_{n \to \infty}
    \sup_{x, y \in V_{m}^{(N_{0})}}
    |R_{n}(x, y) - R(x, y)|
    =0,
    \quad 
    \forall m \geq 1.
  \end{equation} 
  Again, the triangle inequality yields that,
  for each $x, y \in V_{m}^{(N_{0})}$,
  \begin{align}
    |R_{n}(x, y) - R(x,y)|
    &\leq 
    |R_{n}(x, y) - R_{n}|_{m}^{(N)}(x, y)|
    + 
    |R(x, y) - R|_{m}^{(N)}(x, y)| \\
    &\quad 
    +
    |R_{n}|_{m}^{(N)}(x, y) - R|_{m}^{(N)}(x, y)|.
  \end{align}
  Since $V_{m}^{(N_{0})}$ is a finite set,
  we obtain the desired result
  by using Assumption~\ref{A. assum: general condition for approximation}\ref{A. assum item: approximation, neighborhood function} 
  and Lemma~\ref{A. lem: approximation by fused metrics}.
\end{proof}

The above convergence yields convergence of restricted spaces in the pointed Gromov--Hausdorff topology.
This topology is a version of the pointed Gromov--Hausdorff--Prohorov topology 
induced by a metric defined by dropping measures from $d_{\mathrm{GHP}}$ in \eqref{2. eq: dfn of GHP metric};
see \cite[Section~4.1]{Khezeli_20_Metrization} for examples.

\begin{prop}  \label{A. prop: convergence of compact fractals}
  For each $N_{0} \geq 1$,
  it holds that 
  \begin{equation}
    (V_{n}^{(N_{0})}, R_{n}, \rho_{n})
    \to 
    (\cl(V^{(N_{0})}), R, \rho)
  \end{equation}
  in the pointed Gromov--Hausdorff topology,
  where the metrics $R_{n}$ and $R$ are restricted to $V_{n}^{(N_{0})}$ and $\cl(V^{(N_{0})})$,
  respectively.
\end{prop}

\begin{proof}
  Assumption~\ref{A. assum: general condition for approximation}
  \ref{A. assum item: approximation, neighborhood function} and~\ref{A. assum item: approximation, bound for V_m},
  combined with \cite[Theorem~2.6]{Abraham_Delmas_Hoscheit_13_A_note},
  imply that $\{(V_{n}^{(N_{0})}, R_{n}, \rho_{n}) \mid n \geq 1\}$ is precompact in the pointed Gromov Hausdorff topology.
  So it suffices to show that the limit of any convergent subsequence is $(V^{(N_{0})}, R, \rho)$.
  To simplify the subscripts,
  we suppose that $(V_{n}^{(N_{0})}, R_{n}, \rho_{n})$ converges to 
  a rooted compact metric space $(S, d^{S}, \rho_{S})$,
  and we will prove that there exists a root-preserving isometry 
  from $(S, d^{S}, \rho_{S})$ to $(\cl(V^{(N_{0})}), R, \rho)$.
  We may assume that $(V_{n}^{(N_{0})}, R_{n})$ and $(S, d^{S})$
  are isometrically embedded into a common rooted compact metric space $(M, d^{M}, \rho_{M})$
  in such a way that 
  $V_{n}^{(N_{0})} \to S$ in the Hausdorff topology in $M$
  and $\rho_{n} = \rho_{S} = \rho_{M}$
  as elements in $M$.
  Let $S_{0}$ be a countable dense subset of $S$
  containing $\rho_{S}$.
  For $x \in S_{0}$,
  we can find $h_{n}(x) \in V_{n}^{(N_{0})}$ such that $h_{n}(x) \to x$ in $M$.
  Since $h_{n}(x) \in V^{(N_{0})}$ and $(\cl(V^{(N_{0})}), R)$ is compact,
  there exist a subsequence $(n_{k}^{x})_{k \geq 1}$
  and an element $f(x) \in \cl(V^{(N_{0})})$ 
  such that $h_{n_{k}}(x) \to f(x)$ in $(\cl(V^{(N_{0})}), R)$.
  By a diagonal argument,
  we may assume the subsequence $(n_{k}^{x})_{k \geq 1}$ is independent of $x$
  and write $(n_{k})_{k \geq 1}$ for it. 
  Note that we may further assume that $f(\rho_{S}) = \rho$.
  For $x, y \in S_{0}$,
  we have that 
  \begin{align}
    |d^{S}(x,y) - R(f(x), f(y))|
    &\leq 
    d^{M}(x, h_{n_{k}}(x)) + d^{M}(y, g_{n_{k}}(y)) \\
    &\quad
    +
    |R_{n_{k}}(h_{n_{k}}(x), h_{n_{k}}(y)) - R(h_{n_{k}}(x), h_{n_{k}}(y))| \\
    &\quad
    +
    |R(h_{n_{k}}(x), h_{n_{k}}(y)) - R(f(x), f(y))|.
  \end{align}
  This, combined with Proposition~\ref{A. prop: uniform convergence of R_n to R},
  implies that $f \colon S_{0} \to \cl(V^{(N_{0})})$ is root-and-distance-preserving.
  By a standard argument,
  the domain of $f$ is extended to $S$ so that $f \colon S \to \cl(V^{(N_{0})})$ is still root-and-distance-preserving.
  For $x \in V^{(N_{0})}$,
  we have that $\tilde{h}_{n}(x) \coloneqq x \in V_{n}^{(N_{0})}$ for all sufficiently large $n$.
  By the convergence $V_{n}^{(N_{0})} \to S$ in $M$,
  there exist a subsequence $(\tilde{n}_{k}^{x})_{k \geq 1}$ and $\tilde{f}(x) \in S$
  such that $\tilde{h}_{\tilde{n}_{k}^{x}}(x) \to \tilde{f}(x)$ in $M$.
  By the same argument as before,
  we may assume that $(\tilde{n}_{k}^{x})_{k \geq 1}$ is independent of $x$
  and write $(\tilde{n}_{k})_{k \geq 1}$ for it.
  For $x, y \in V^{(N_{0})}$,
  we have that 
  \begin{equation}
    |R(x, y) - d^{S}(\tilde{f}(x), \tilde{f}(y))|
    \leq 
    |R(x,y) - R_{\tilde{n}_{k}}(x,y)|
    + 
    |d^{M}(\tilde{h}_{\tilde{n}_{k}}(x), \tilde{h}_{\tilde{n}_{k}}(y)) - d^{M}(\tilde{f}(x), \tilde{f}(y))|.
  \end{equation}
  Again,
  Proposition~\ref{A. prop: uniform convergence of R_n to R} yields 
  that $\tilde{f} \colon V^{(N_{0})} \to S$ is distance-preserving,
  and it is extended to a distance-preserving map $\tilde{f} \colon \cl(V^{(N_{0})}) \to S$.
  By \cite[Theorem~1.6.14]{Burago_Burago_Ivanov_01_A_course},
  we deduce that $f \colon S \to \cl(V^{(N_{0})})$ is a root-preserving isometry,
  which completes the proof.
\end{proof}

Finally, we prove Theorem~\ref{A. thm: convergence of fractals}.

\begin{proof} [Proof of Theorem~\ref{A. thm: convergence of fractals}]
  For $r>0$,
  by Assumption~\ref{A. assum: general condition for approximation}\ref{A. assum item: approximation, non-explosion} 
  and the proof of Lemma~\ref{A. lem: recurrence of limiting space},
  we can find $N \geq 1$ 
  such that 
  \begin{equation}
    \liminf_{n \to \infty} R_{n}(\rho_{n}, V_{n} \setminus K^{(N)}) > r,
    \quad 
    R(\rho, \cl(V \setminus K^{(N)})) > r.
  \end{equation}
  These imply that 
  $D_{R_{n}}(\rho_{n}, r) \subseteq V_{n}^{(N)}$ for all sufficiently large $n$ 
  and $D_{R}(\rho, r) \subseteq \cl(V^{(N)})$.
  From Proposition~\ref{A. prop: convergence of compact fractals},
  we deduce that,
  for all but countably many $r>0$,
  \begin{equation}
    (D_{R_{n}}(\rho_{n}, r), R_{n}, \rho_{n})
    \to 
    (D_{R}(\rho, r), R, \rho)
  \end{equation}
  in the pointed Gromov--Hausdorff topology,
  where the metrics are restricted appropriately.
  Hence, we obtain the convergence \eqref{A. thm eq: convergence of fractals}.
  For $N \geq 1$,
  by Lemma~\ref{A. lem: the diameters are uniformly bounded},
  we can find $r>0$ such that 
  \begin{equation}
    \sup_{x \in V_{n}^{(N)}} R_{n}(\rho_{n}, x) < r
  \end{equation}
  for all sufficiently large $n$,
  which implies that $B_{R_{n}}(\rho_{n}, r)^{c} \subseteq V_{n} \setminus K^{(N)}$.
  This, combined with Assumption~\ref{A. assum: general condition for approximation}\ref{A. assum item: approximation, non-explosion},
  implies the non-explosion condition \eqref{A. thm eq: convergence of fractals, non-explosion}.
\end{proof}

Now, we apply Theorem~\ref{A. thm: convergence of fractals} to the random conductance model on the unbounded Sierpi\'{n}ski gasket.
Recall the setting of Example \ref{A. exm: unbounded Sierpinski gasket}.
Here,
we put random conductances on $(V_{n}, E_{n})$ as follows.
Let $c_{n} = \{c_{n}(e) \mid e \in E_{n}\}$ be a collection of i.i.d.\ strictly-positive random variables 
built on a probability space $(\Omega, \mathcal{G}, \mathbf{P})$.
We assume that there exist constants $C_{1}, C_{2} > 0$ 
satisfying $C_{1} < c_{n}(e) < C_{2}$ with probability $1$ for all $e \in E_{n}$ and $n$.
We regard $c_{n}$ as a random conductance set on $V_{n}$ 
to obtain a random electrical network $G_{n} = (V_{n}, E_{n}, c_{n})$ with root $\rho_{n} \coloneqq (0,0)$. 
Set $a_{n} = (5/3)^{n}$.
The homogenization result of \cite{Kumagai_04_Homogenization} implies that 
there exists a sequence of deterministic, compatible electrical networks $G|_{n} = (V_{n}, E|_{n}, c|_{n})$
such that 
$a_{n} c_{n}|_{m}^{(N)} \to c|_{m}^{(N)}$ 
in $L^{1}(\Omega, \mathbf{P})$,
which corresponds to 
Assumption~\ref{A. assum: general condition for approximation}\ref{A. assum item: approximation, conductance convergence}.
Since the random conductances are uniformly bounded,
the other conditions of Assumption~\ref{A. assum: general condition for approximation} are satisfied 
with probability $1$. 
Therefore, 
we obtain the convergence \eqref{A. thm eq: convergence of fractals} in probability.
Let $K_{0}$ be the Sierpi\'{n}ski gasket,
i.e.,
the unique non-empty compact subset of $\RN^{2}$
satisfying $K_{0} = \bigcup_{i=1}^{3} \psi_{i}(K_{0})$.
We then define $K_{\infty} \coloneqq \bigcup_{n \geq 0} 2^{n} K_{0}$
equipped with the induced topology from $\RN^{2}$.
We note that there exists a homeomorphism $H \colon K_{\infty} \to F$
such that $H(x) = x$ for all $x \in V$ 
and hence we may identify $F$ with $K_{\infty}$.
We write $\mu_{K}$ for a self-similar measure on $K$
corresponding to the $\log_2 3$ Hausdorff measure in the Euclidean metric,
which is unique up to a constant multiple.
We then define $\mu$
to be a measure on the unbounded Sierpi\'{n}ski gasket $F$
characterized by $\mu|_{K} = \mu_{K}$
and $3 \mu = \mu \circ \psi_{1}^{-1}$.
As mentioned in \cite[Remark 6.19]{Croydon_Hambly_Kumagai_17_Time-changes},
there exists a deterministic constant $c_{0} > 0$
such that  
the random measures $b_{n}^{-1} \mu_{G_{n}}$,
where we set $b_{n} = c_{0}3^{n}$,
converge to $\mu$ on $F$
in probability with respect to the vague topology 
for Radon measures on $\RN^{2}$.
(NB.\ In \cite[Remark 6.19]{Croydon_Hambly_Kumagai_17_Time-changes},
the compact Sierpi\'{n}ski gasket $K_{0}$ is considered,
but it is easy to extend the result to the unbounded Sierpi\'{n}ski gasket.)
As a consequence,
it holds that 
\begin{equation}
  (V_{n}, a_{n}^{-1}R_{n}, \rho_{n}, b_{n}^{-1}\mu_{G_{n}})
  \xrightarrow{\mathrm{d}}
  (F, R, \rho, \mu)
\end{equation}
in the local Gromov--Hausdorff-vague topology.
The assumption of the lower bound on the random conductances
and \cite[Proposition~7.2]{Noda_pre_Convergence} imply that 
the metric-entropy condition is satisfied.
Theorem~\ref{1. thm: random, main result} yields 
the convergence of the random walks and local times on $G_{n}$.
The same result holds for random conductances with finite expectation
without assuming boundedness from above,
and for homogenization on fractals in a more general setting,
see \cite{Croydon_21_The_random} and \cite[Section~6]{Croydon_Hambly_Kumagai_17_Time-changes}.

%%%%%%%%%%%%%%%%%%%%%%%%%%%%%%%%%%%%%%%%%%%%%%%%%%%%%%%%%%%%%%%%%%%%%%%%%%%%%%%%%%%%%%%%%%%%%%%%%%%%%%%%%%%%%%%%%%%%%%%%%%%%%%%%%%
% Convergence of traces
%%%%%%%%%%%%%%%%%%%%%%%%%%%%%%%%%%%%%%%%%%%%%%%%%%%%%%%%%%%%%%%%%%%%%%%%%%%%%%%%%%%%%%%%%%%%%%%%%%%%%%%%%%%%%%%%%%%%%%%%%%%%%%%%%

\section{Convergence of traces}
\label{A. sec: convergence of traces}

Under Assumption~\ref{1. assum: deterministic version}\ref{1. assum item: deterministic, convergence of spaces},
it is an immediate consequence of the convergence with respect to the local Gromov--Hausdorff-vague topology 
that, for all but countably many $r>0$,
\begin{equation}  \label{A. eq: convergence of restrictions}
  (\hat{V}_{n}^{(r)}, \hat{R}_{n}^{(r)}, \hat{\rho}_{n}^{(r)}, \hat{\mu}_{n}^{(r)})
  \to 
  (F^{(r)}, R^{(r)}, \rho^{(r)}, \mu^{(r)})
\end{equation}
in the Gromov--Hausdorff--Prohorov topology,
where we recall the restriction operator $\cdot^{(r)}$ from \eqref{1. eq: dfn of restriction operator}.
However,
this does not necessarily imply the convergence of traces 
$(\tilde{V}_{n}^{(r)}, \tilde{R}_{n}^{(r)}, \tilde{\rho}_{n}^{(r)}, \tilde{\mu}_{n}^{(r)})$ 
because $\hat{\mu}_{n}^{(r)} \neq \tilde{\mu}_{n}^{(r)}$ in general
(recall the notation for traces from Section~\ref{sec: proof of main result 1}).
As implied in Corollary~\ref{4. cor: difference for trace measure at one point},
there is a difference between  $\hat{\mu}_{n}^{(r)}$ and $\tilde{\mu}_{n}^{(r)}$
due to the effect of jumps of the Markov chain to the outside of the trace.
In this appendix, we provide a sufficient condition for the convergence of the traces
(Theorems~\ref{A. thm: convergence of deterministic traces} and \ref{A. thm: convergence of random traces} below).

Assume that we are in the setting of Theorem~\ref{1. thm: deterministic, main result}.
We define 
the total scaled conductance on edges crossing over a set $A$ (with respect to the scaled metric $\hat{R}_{n}$)
by 
\begin{equation}
  \mathsf{CC}_{n}(A)
  \coloneqq 
  \sum_{\substack{x  \in A
          \\ y \notin A}}
  b_{n}^{-1} \mu_{G_{n}}(x,y).
\end{equation}
We now suppose that, for each $n$, we have a compact subset $K_{n}$ of $(\hat{V}_{n}, \hat{R}_{n})$
containing $\rho_{n}$.
We write 
\begin{equation}
  \hat{R}_{n}^{(K_{n})} \coloneqq \hat{R}_{n}|_{K_{n} \times K_{n}},
  \quad
  \hat{\rho}_{n}^{(K_{n})} \coloneqq \hat{\rho}_{n},
  \quad 
  \hat{\mu}_{n}^{(K_{n})}(\cdot)
  \coloneqq
  \hat{\mu}_{n}(\cdot \cap K_{n}).
\end{equation}
Similarly,
given a compact subset $K$ of $(F, R)$,
we define restrictions $R^{(K)},\, \rho^{(K)}$ and $\mu^{(K)}$.
We consider the following conditions:
convergence of restricted spaces and decay of $\mathsf{CC}_n(K_n)$.

\begin{assum}  \leavevmode
  \label{A. assum: deterministic, subsets condition}
  \begin{enumerate} [label = \textup{(\roman*)}]
    \item \label{A. assum item: deterministic, space convergence}
      It holds that 
      \begin{equation}
        (K_{n}, \hat{R}_{n}^{(K_{n})}, \hat{\rho}_{n}^{(K_{n})}, \hat{\mu}_{n}^{(K_{n})})
        \to 
        (K, R^{(K)}, \rho^{(K)}, \mu^{(K)})
      \end{equation}
      in the Gromov--Hausdorff--Prohorov topology
      for some compact subset $K$ of $(F, R)$.
    \item \label{A. assum item: deterministic, CC convergence}
      It holds that $\lim_{n \to \infty} \mathsf{CC}_{n}(K_{n}) = 0$.
  \end{enumerate}
\end{assum}

Recall from Section~\ref{sec: trace of electrical networks} that 
$G_{n}|_{K_{n}}$ denotes the trace of the electrical network $G_{n}$ onto $K_{n}$.
We equip $G_{n}|_{K_{n}}$ with the root $\rho_{n}$.
We then write 
\begin{equation}
  \tilde{R}_{n}^{(K_{n})} \coloneqq a_{n}^{-1} R_{G_{n}|_{K_{n}}},
  \quad
  \tilde{\rho}_{n}^{(K_{n})} \coloneqq \rho_{G_{n}|_{K_{n}}},
  \quad 
  \tilde{\mu}_{n}^{(K_{n})} \coloneqq b_{n}^{-1} \mu_{G_{n}|_{K_{n}}}.
\end{equation}
As before,
we note that 
$\tilde{R}_{n}^{(K_{n})} = \hat{R}_{n}^{(K_{n})}$ and 
$\tilde{\rho}^{(K_{n})} = \hat{\rho}_{n}^{(K_{n})}$,
but $\tilde{\mu}_{n}^{(K_{n})} \neq \hat{\mu}_{n}^{(K_{n})}$ in general.

\begin{thm} \label{A. thm: convergence of deterministic traces}
  Under Assumption~\ref{A. assum: deterministic, subsets condition},
  it holds that 
  \begin{equation}
    (K_{n}, \tilde{R}_{n}^{(K_{n})}, \tilde{\rho}_{n}^{(K_{n})}, \tilde{\mu}_{n}^{(K_{n})})
    \to 
    (K, R^{(K)}, \rho^{(K)}, \mu^{(K)})
  \end{equation}
  in the Gromov--Hausdorff--Prohorov topology.
\end{thm}

\begin{proof}
  Fix a subset $A \subseteq K_{n}$.
  By Corollary~\ref{4. cor: difference for trace measure at one point},
  we deduce that 
  \begin{align}
    \hat{\mu}_{n}^{(K_{n})}(A)
    &\leq 
    \sum_{x \in A}
    b_{n}^{-1} 
    c_{G_{n}|_{K_{n}}}(x)
    +
    \sum_{x \in A}
    \sum_{y \notin K_{n}}
    b_{n}^{-1} c_{G_{n}}(x, y)  \\
    &\leq 
    \tilde{\mu}_{n}^{(K_{n})}(A) 
    + 
    \mathsf{CC}_{n}(K_{n}),
  \end{align}
  which implies that 
  \begin{equation}  \label{A. eq: deterministic, difference between restriction and trace}
    \dP^{K_{n}}( \hat{\mu}_{n}^{(K_{n})}, \tilde{\mu}_{n}^{(K_{ n})}) 
    \leq 
    \mathsf{CC}_{n}(K_{n}),
  \end{equation}
  where $\dP^{K_{n}}$ denotes the Prohorov metric on $\cMfin(K_{n})$.
  By Assumption~\ref{A. assum: deterministic, subsets condition}\ref{A. assum item: deterministic, space convergence},
  we can embed all the metric spaces $(K_n, \tilde{R}_n^{(K_n)})$ and $(K, R^{(K)})$
  isometrically into a common rooted compact metric space $(M, d^M, \rho_M)$
  in such a way that 
  $\tilde{\rho}_n^{(K_n)} = \rho^{(K)} = \rho_M$ as elements of $M$,
  and $\tilde{\mu}_n^{(K_n)} \to \mu^{(K)}$ weakly as measures on $M$.
  From Assumption~\ref{A. assum: deterministic, subsets condition}\ref{A. assum item: deterministic, CC convergence}
  and \eqref{A. eq: deterministic, difference between restriction and trace},
  it follows that $\tilde{\mu}_n^{(K_n)} \to \mu^{{(K)}}$ weakly.
  Since we have that $\hat{R}_{n}^{(K_{n})} = \tilde{R}_{n}^{(K_{n})}$ 
  and $\hat{\rho}_{n}^{(K_{n})} = \tilde{\rho}_{n}^{(K_{n})}$ (by definition),
  we obtain the desired result.
\end{proof}

We next consider random electrical networks.
We proceed in the setting of Theorem~\ref{1. thm: random, main result}
and
suppose that $(K_{n}, \hat{R}_{n}^{(K_{n})}, \hat{\rho}_{n}^{(K_{n})}, \hat{\mu}_{n}^{(K_{n})})$
is a random element of $\mathbb{F}_{c}$,
where $K_{n}$ is a subset of the random set $\hat{V}_{n}$
and the space $\mathbb{F}_c$ is recalled from Definition~\ref{1. dfn: the space F}.
The following is a version of Assumption~\ref{A. assum: deterministic, subsets condition} for random electrical networks.

\begin{assum}  \leavevmode
  \label{A. assum: random, subsets condition}
  \begin{enumerate} [label = \textup{(\roman*)}]
    \item \label{A. assum item: random, space convergence}
      It holds that 
      \begin{equation}
        (K_{n}, \hat{R}_{n}^{(K_{n})}, \hat{\rho}_{n}^{(K_{n})}, \hat{\mu}_{n}^{(K_{n})})
        \xrightarrow{\mathrm{d}}
        (K, R^{(K)}, \rho^{(K)}, \mu^{(K)})
      \end{equation}
      in the Gromov--Hausdorff--Prohorov topology
      for some random element $(K, R^{(K)}, \rho^{(K)}, \mu^{(K)})$
      of $\mathbb{F}_{c}$,
      where $K$ is a compact subset of $(F, R)$.
    \item \label{A. assum item: random, CC convergence}
      It holds that $\mathsf{CC}_{n}(K_{n}) \xrightarrow{\mathrm{p}} 0$.
  \end{enumerate}
\end{assum}

\begin{thm} \label{A. thm: convergence of random traces}
  Under Assumption~\ref{A. assum: random, subsets condition},
  it holds that 
  \begin{equation}
    (K_{n}, \tilde{R}_{n}^{(K_{n})}, \tilde{\rho}_{n}^{(K_{n})}, \tilde{\mu}_{n}^{(K_{n})})
    \xrightarrow{\mathrm{d}} 
    (K, R^{(K)}, \rho^{(K)}, \mu^{(K)})
  \end{equation}
  in the Gromov--Hausdorff--Prohorov topology.
\end{thm}

\begin{proof}
  From \eqref{A. eq: deterministic, difference between restriction and trace} 
  and Assumption~\ref{A. assum: random, subsets condition}\ref{A. assum item: random, CC convergence},
  it follows that $\dP^{K_{n}}( \hat{\mu}_{n}^{(K_{n})}, \tilde{\mu}_{n}^{(K_{ n})}) \xrightarrow{\mathrm{p}} 0$.
  Similarly to the proof of Theorem~\ref{A. thm: convergence of deterministic traces},
  we obtain the desired result.
\end{proof}

\begin{rem} \label{A. rem: proc and local times converge for traces}
  Assumption~\ref{A. assum: deterministic, subsets condition} implies that 
  if one takes $r$ sufficiently large,
  then
  $K_{n} \subseteq \hat{V}_{n}^{(r)}$ holds for all $n$.
  It then follows that the sequence of which the convergence is considered in Theorem~\ref{A. thm: convergence of deterministic traces}
  satisfies the non-explosion condition.
  Hence,
  if the sequence satisfies the metric-entropy condition,
  then we obtain the convergence of the associated processes and local times.
  The same is true for Theorem~\ref{A. thm: convergence of random traces}.
\end{rem}

\begin{exm}
  Recall the random conductance model on the unbounded Sierpi\'{n}ski gasket 
  we considered at the end of Appendix \ref{A. sec: application}.
  Fix $N \geq 1$. 
  Since $\mu \left( \partial (\cl(V^{(N)})) \right) = 0$,
  where $\partial \cdot$ denotes the boundary with respect to $(F, R)$, 
  it is elementary to check that 
  $b_{n}^{-1} \mu_{G_{n}}|_{V_{n}^{(N)}}$ converges to $\mu|_{\cl(V^{(N)})}$ 
  in probability
  with respect to the weak topology for finite measures on $\RN^{2}$.
  This, combined with Proposition~\ref{A. prop: convergence of compact fractals},
  yields that 
  \begin{equation}
    (V_{n}^{(N)}, a_{n}^{-1} R_{n}, \rho_{n}, b_{n}^{-1} \tilde{\mu}_{n}) 
    \xrightarrow{\mathrm{d}}
    (\cl(V^{(N)}), R, \rho, \mu|_{\cl(V^{(N)})})
  \end{equation}
  in the Gromov--Hausdorff--Prohorov topology,
  where metrics are restricted appropriately
  and we set $\tilde{\mu}_{n} \coloneqq \mu_{G_{n}|_{V_{n}^{(N)}}}$,
  i.e.,
  the measure associated with the trace electrical network $G_{n}|_{V_{n}^{(N)}}$.
  Moreover,
  the boundedness of the random conductances 
  and the fact that the cardinality of 
  $\{x \in V_{n} \mid \exists y \notin V_{n} \setminus V_{n}^{(N)} \text{such that}\ \{x, y\} \in E_{n} \}$
  is $4$
  imply Assumption~\ref{A. assum: random, subsets condition}\ref{A. assum item: random, CC convergence}.
  Hence, we obtain the convergence of traces $G_{n}|_{V_{n}^{(N)}}$.
  As we mentioned in Remark \ref{A. rem: proc and local times converge for traces},
  we also obtain the convergence of the random walks and local times on $G_{n}|_{V_{n}^{(N)}}$.
\end{exm}

\section*{Acknowledgement}
I would like to thank my supervisor Dr David Croydon for his support and fruitful discussions. 
This work was supported by 
JSPS KAKENHI Grant Number JP 24KJ1447
and 
the Research Institute for Mathematical Sciences, 
an International Joint Usage/Research Center located in Kyoto University.

\bibliographystyle{amsplain}
\bibliography{ref_discrete_time}

\end{document}